\documentclass[10pt, a4paper]{article}

\usepackage{setspace}
\usepackage[top=3cm,bottom=3cm,left=3cm,right=2.5cm]{geometry}
\usepackage[latin1]{inputenc}
\usepackage{fancyhdr}
\usepackage{makeidx}
\usepackage{amsmath,amsfonts,amssymb,bm, amsthm}
\usepackage{mathrsfs}
\usepackage{color}
\usepackage{enumitem}
\usepackage[english]{babel}
\usepackage{graphicx}
\usepackage{graphics}
\usepackage[T1]{fontenc}
\usepackage{float}
\usepackage{placeins}
\graphicspath{{dessin_TikZ/}}
\usepackage{fancyhdr}
\usepackage{mathpazo}
\usepackage{xcolor}
\usepackage{pgfplots}
\usepackage{array}
\usepackage{indentfirst}
\usepackage{mathtools}
\usepackage{pdfpages}
\usepackage{multirow}
 \usepackage{titlesec} 
  \usetikzlibrary{arrows.meta}
\usepackage{latexsym}
\usepackage{caption}
\usepackage{tikz}
\usepackage{tikz-cd}
\usetikzlibrary{%
    decorations.pathreplacing,%
    decorations.pathmorphing,%
    patterns,
    calc,
    intersections
}
\usepgfplotslibrary{fillbetween}
\usepackage{adjustbox}
\usepackage{hyperref}
\usepackage{esint}
\usepackage{stmaryrd}

\fancypagestyle{first}{ % pages de tetes de chapitre
\fancyhead{} % supprime lentete

\fancyfoot[L]{\footnotesize $^1$ CEA-CESTA, Le Barp, France\\
$^2$ POEMS, CNRS, Inria, ENSTA, Institut Polytechnique de Paris, 91120 Palaiseau, France}}

\makeatletter
\renewcommand{\maketitle}{
\begin{center}
\normalsize {\bfseries\@title} \\[0.3cm] 
\normalsize \begin{tabular}[t]{c}%
\@author
\end{tabular}
\end{center}
}
\makeatother

\definecolor{color1}{rgb}{0.22,0.45,0.70}% light blue
\definecolor{color2}{rgb}{0.9, 0.17, 0.31}
\definecolor{color3}{rgb}{0.8, 0.25, 0.33}
\definecolor{blue2}{rgb}{0.2,0.2,0.7}
\definecolor{green2}{rgb}{0.13, 0.7, 0.67}
\definecolor{green3}{rgb}{0.3,0.6,0.4}
 \definecolor{mygray}{gray}{0.90}
\definecolor{mylightgray}{gray}{0.90}
\definecolor{plum}{rgb}{0.56, 0.27, 0.52}
\newcommand{\C}{\mathbb{C}}
\newcommand{\N}{\mathbb{N}}
\newcommand{\E}{\mathbb{E}}
\newcommand{\F}{\mathcal{F}}

\newcommand{\Pa}{\mathcal{P}}
\newcommand{\Pro}{\mathbb{P}}
\newcommand{\R}{\mathbb{R}}
\newcommand{\dst}{\displaystyle}
\newcommand{\eps}{\varepsilon}
\newcommand{\Var}{\textrm{Var}}

\newcommand{\vsd}{\vspace{0.2cm}}
\newcommand{\vsu}{\vspace{0.1cm}}

\newcommand{\inte}{n\in\mathbb{N}}

\newcommand{\re}[1]{{\rm Re}\left[#1\right]} 
\newcommand{\im}[1]{{\rm Im}\left[#1\right]}

\newcommand{\bsx}{\boldsymbol{x}}
\newcommand{\bsxpar}{\bsx_\shortparallel}
\newcommand{\bsy}{\boldsymbol{y}}
\newcommand{\bsypar}{\bsy_\shortparallel}
\newcommand{\bsz}{\boldsymbol{z}}
\newcommand{\bszpar}{\bsz_\shortparallel}

\newcommand{\bsw}{\boldsymbol{w}}
\newcommand{\bsn}{\boldsymbol{n}}
\newcommand{\bsg}{\boldsymbol{g}}
\newcommand{\bsh}{\boldsymbol{h}}

\newcommand{\bsG}{\boldsymbol{G}}

\newcommand{\bstheta}{\boldsymbol{\theta}}

\newcommand{\ii}{\mathrm{i}}
\newcommand{\dd}{\mathrm{d}}

\newcommand{\Lchi}{\text{\large $\chi$}}

\newtheorem{definition}{Definition}

\newtheorem{remark}{Remark}
\newtheorem{hypothesis}{Hypothesis}

\newtheorem{theorem}{Theorem}
\newtheorem{proposition}[theorem]{Proposition}
\newtheorem{corollary}[theorem]{Corollary}
\newtheorem{lemma}[theorem]{Lemma}

\pgfplotsset{
    standard/.style={
        axis x line=middle,
        axis y line=middle,
        every axis x label/.style={at={(current axis.right of origin)}, anchor=south east},
        every axis y label/.style={at={(current axis.above origin)},anchor=north east}
    }
}

\title{\normalsize \bfseries SCATTERING FROM A THIN RANDOM COATING OF NANOPARTICLES:\\ THE DIRICHLET CASE}
\date{}

\author{Amandine Boucart$^{1,2}$, Sonia Fliss$^2$ and Laure Giovangigli$^2$}

\begin{document}
\thispagestyle{first}
\maketitle 

% REQUIRED
\begin{abstract}
We study the time-harmonic scattering by a heterogeneous object covered with a thin layer of randomly distributed sound-soft nanoparticles. The size of the particles, their distance between each other and the layer's thickness are all of the same order but small compared to the wavelength of the incident wave. Solving the Helmholtz equation in this context can be very costly and the simulation depends on the given distribution of particles. To circumvent this, we propose, via a multi-scale asymptotic expansion of the solution, an effective model where the layer of particles is replaced by an equivalent boundary condition. The coefficients that appear in this equivalent boundary condition depend on the solutions to corrector problems of Laplace type defined on unbounded random domains. Under the assumption that the particles are distributed given a stationary and mixing random point process, we prove that those problems admit a unique solution in the proper space. We then establish quantitative error estimates for the effective model and present numerical simulations that illustrate our theoretical results.
\end{abstract}

% REQUIRED
% \begin{keywords}
% Stochastic homogenization, effective boundary conditions, quantitative error estiamtes, Helmholtz equations, random thin layers.
% \end{keywords}

% % REQUIRED
% \begin{MSCcodes}
% 68Q25, 68R10, 68U05
% \end{MSCcodes}

%\tableofcontents

In this article we study the time-harmonic electromagnetic scattering by a heterogeneous object covered with a thin layer of randomly distributed sound-soft nanoparticles. The size of the particles, their spacing and the layer's thickness are of the order of a few nanometers and very small compared to the wavelength of the incident wave which is of the order of the centimeter. Two difficulties arise when trying to solve Maxwell's equation in this context : first, on account of the very different scales, the computational cost with a standard numerical method is prohibitive; second the distribution of particles for a given object is unknown. Our goal is to propose an effective model where the layer of particles is replaced by an equivalent boundary condition and whose solution approximates the original solution with an error that we can estimate.

Methods to build such effective models can be classified into two categories : methods based on optimization techniques and methods based on asymptotic techniques. The principle of the first class of methods is, given an effective boundary condition, to optimize the coefficients appearing in the condition so that the effective solution fits in a certain sense as much as possible a reference solution. Although the effective model is simple to implement, it is difficult to estimate its accuracy with respect to the physical parameters of the problem. See for instance \cite{hoppe2018impedance,stupfel2015implementation}.

Our work lies in the second class of methods. It is based on a multi-scale asymptotic expansion which allows us to take into account the fast variations of the physical solution close to the layer and its slower variations far from the layer. These techniques in the context of scattering problem by thin layers are well known when the thin layer is penetrable and modeled by a constant coefficient (see for instance  \cite{Bartoli2002,Haddar2002,Vial2003,Caloz2006}) or when the thin layer is periodic, \emph{i.e.} the distribution of the particles is supposed to be periodic in the direction of the interface (see for instance \cite{Abboud1996,Delourme2010,Delourme2012,Claeys2013,Marigo2016}). Note that in this case, the technique works whichever the nature of the particles is: penetrable, sound-soft or sound-hard. It seems that our work is the first one dealing with scattering problem by a thin layer with random properties. Let us however mention \cite{Kristensson} whose study is based on the "heuristic" quasi-crystalline approximation. Our work has of course lots of links with the literature on volume stochastic homogenenization (see for instance \cite{Gloria2019,DuerinckxGloria2017, DuerinckxGloria2019, DuerinckxGloria2021}). We call it "volume" in contrast with our problem which could be called "surface" stochastic homogenization. Indeed, in volume stochastic homogenization, the question is to deal with (in general elliptic) PDEs with (stationary and ergodic) random coefficients and to look for an effective PDE in general of the same nature but with constant coefficients. Our problem is different since the random medium is constituted only by a thin layer and we wish to replace the thin layer by an effective boundary condition, while the PDE outside the layer stays the same. One will find of course similarities in the analysis (for instance in order to obtain the quantitative error estimates) but some difficulties are really specific to surface homogenization. Let us mention that similar problems were treated in the context of Laplace equation \cite{Chechkin2009,Amirat2011} or Stokes equations \cite{Basson2006,GerardVaretMasmoudi,DalibardGerardVaret,ElJarroudi2019} in presence of random rough boundaries. The difficulty of dealing with a random distribution of particles instead of rough boundaries lies in, as we will see later, the derivation of Poincar\'e inequalities.

The problem that we consider in this paper is a simplification of the original application. 
First of all, we consider the time-harmonic acoustic scattering by an object. 
In other words, the unknown is a scalar function and the equation is Helmholtz's. 
When the dimension of the domain is $d=2$, this problem can be derived from 
the study of electromagnetic scattering by a 3D object that is invariant in one direction,
where the incident field is TE- or TM-polarized. More generally, this study is an important 
first step before addressing the more general  case of electromagnetism, on which we are working. 
Also, for further simplification, we neglect the complexity of the geometry of the object and assume 
that the object is planar. Moreover, in the applications, the object is covered by multple layers of material whose goal is to attenuate 
the radar cross section of the object. This multilayer medium is modeled and replaced by an impedance 
boundary condition that is imposed on the boundary of the object. Finally, we impose  Dirichlet 
boundary conditions on the particles. This corresponds, if we go back to the electromagnetic 
equivalent (see above) at $d=2$, to perfectly conducting particles and a TE-polarization of the field. 
Imposing Neumann boundary conditions constitutes another challenge from a theoretical point of view and will be the subject of a forthcoming paper.

The paper is organized as follows: The model is presented in Section \ref{sub:presentation}. We detail in particular the random setting with a classical stationary and ergodic assumption. Let us emphasize Hypothesis \nameref{hyp:mixing} that ensures that 
arbitrarily large portions of the layer without Dirichlet particles can exist but only with a small probability. 
We also establish well-posedness of the scattering problem. The formal asymptotic expansion is described
in Section \ref{sec:formalasympexp}. The proposed Ansatz involves so called far-field terms that approximate the solution far from the layer
and the so called near-field terms, depending on microscopic variables and the longitudinal macroscopic variables, that capture specifically the scattering phenomena near the layer. The near-field terms satisfy Laplace-type PDEs parametrized by the macroscopic variables, set in a random halfspace.
We study in Section \ref{sec:NF_pbs} the well-posedness of the near-field problems for which Hypothesis \nameref{hyp:mixing} 
is crucial. Imposing that the limit of the near-fields vanishes at infinity yields the boundary condition satisfied by each far-field term. The behavior at infinity of the near-fields is studied in Section
\ref{sec:limbehavior}. One can then derive an effective model (see Section \ref{sec:eff_model})
whose solution approaches the first two far-field terms. Quantitative error estimates are derived in Section \ref{sec:errest}. Those estimates are established under 
a new Hypothesis \nameref{hyp:Linfty}, that supposes that the distance from any point in the layer to the nearest particule is bounded. If we assume additionally that the particles' distribution process satifies a quantitative mixing assumption (see Hypothesis \nameref{hyp:mix}), we can impove the convergence rate of the effective model. A certain number of technical lemmas that are necessary to the proof of those improved estimates are proved in the appendices at the end of the paper. Finally, we present numerical simulations illustrating the theoretical results of the paper.

\paragraph*{Notations used throughout the paper.}
The integer $d=2$ or $3$ denotes the dimension.
Let $\bsx\coloneqq (x_1,x_2)$ and $\bsx\coloneqq (x_1,x_2,x_3)$ denote the Cartesian coordinate system in $\R^d$ respectively for $d= 2$ and $d=3$. For all  $\bsx\in\R^d$,
$\bsx_\shortparallel$ denotes the tangential components of $\bsx$, i.e. $\bsx_\shortparallel\coloneqq x_1$ when $d=2$ and $\bsx_\shortparallel\coloneqq (x_1,x_2)$ when $d=3$. 

The index $\omega$ indicates a dependency in the random distribution of the particles. We will specify the random setting that we consider in Section \ref{sub:random_setting}.

The positive real number $\eps>0$ denotes a small parameter, all the domains and functions that depend on $\eps$ are indexed by $\eps$.

For any measurable bounded domain $B$, the spatial average over $B$ is denoted by $\fint_{B}\coloneqq  {|B|}^{-1}\int_{B}$, where $|B|$ is the Lebesgue measure of $B$.

The centered cube of side $R$ of dimension $d-1$ (we insist on the dimension $d-1$) is denoted $\square_R\coloneqq (-R/2,R/2)^{d-1}$.
       \section{Presentation of the model}\label{sub:presentation}

\subsection{Domain of propagation}
Let us consider an infinite plane denoted $\Sigma_0\coloneqq \{x_d=0\}$ covered by a thin layer $\mathcal{L}_\eps\coloneqq   \R^{d-1}\times (0,\eps h)$ containing  a set, denoted $\vsu\Pa^\omega_\eps$, of randomly distributed particles of size $\eps$. Let $D_\eps^\omega\coloneqq  \R^{d-1}\times \R^+\setminus \overline{\Pa_\eps^\omega}$ be the domain above $\Sigma_0$ outside $\Pa_\eps^\omega$. 

 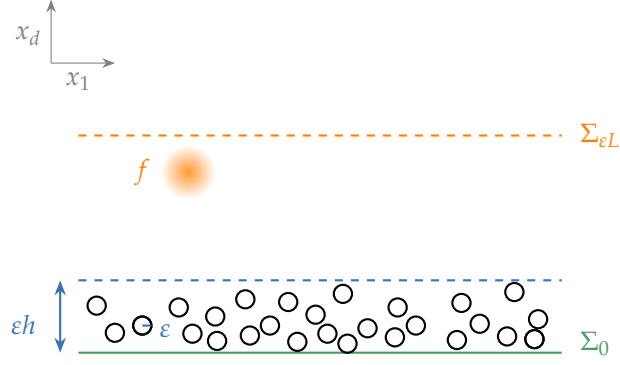
\begin{figure}
 \begin{center}
      \begin{tikzpicture}[scale=1.2]
    \draw[gray,-{Stealth}] (-3.5,3.2) -- (-2.8,3.2);
    \draw [gray] (-3.2,3.2) node[below] {$x_1$};
    \draw [gray,-{Stealth}] (-3.5,3.2) -- (-3.5,3.9);
    \draw [gray] (-3.5,3.5) node[left] {$x_d$}; 
    %%%%%%
%    \draw[color3,thick,>=stealth,->] (-3,2.8) -- (-2.25,2.3);
%    \draw[color3,thick] (-3.2,2.5) -- (-2.8,3.1);
%    \draw[color3,thick] (-3,2.3) -- (-2.6,2.9);
    %\draw[color3,thick] (-3.3,3.5) node[above] {\small Onde incidente};
    
%    \draw[line width=.2mm, color=color3!80!white, fill=color2!10!white] plot [smooth cycle, tension=0.689] coordinates {(-2.8,1.7) (-2.5, 2.2) (-2.3,1.85) (-2.1,1.5) };
%    \draw[color3,thick] (-2.5,1.9) node{\small $f$};
  \fill[color=white, outer color=white, inner color =orange!80!white](-2, 2) circle(3mm);         
                                \draw[color=orange](-2.5,2) node{$f$}; 
    %%%%%%
    \draw[green3,thick] (-3.2,0) -- (2.1,0);  
     %\draw[color1,thick,dashed] (-3.2,0) -- (2.1,0);  
        \draw[color1,thick,dashed] (-3.2,0.8) -- (2.1,0.8);  
                \draw[orange,thick,dashed] (-3.2,2.4) -- (2.1,2.4);  
    %\draw [green3] (1.3,0) node[below] {$\nabla u^{\varepsilon}.\vec n +\gamma u^{\varepsilon}=0  $};
    %\draw [green3] (6.2,0) node[right] {(CI)} ; 
       \draw [green3] (2.2,0.1) node[right] {$\Sigma_0$} ; 
            \draw [orange] (2.2,2.4) node[right] {$\Sigma_{\eps L}$} ; 
    %%%%%%
    %\draw [color3] (1.3,2.4) node[below] {$-\Delta u^{\varepsilon} - k^2 u^{\varepsilon} = 0 $};
    %%%%%%
    %\draw[>=stealth,<-]  (-3.15,0.65) -- (-3.5,0.65);
    %\draw (-3.4,0.65) node[left] {$\nabla u^{\varepsilon}.\vec n =0 $};
    %%%%%%
    \draw [thick] (-3,0.52) circle(0.1);
    \draw [thick] (-2.5,0.3) circle(0.1);
        \draw [thick] (-1.95,0.21) circle(0.1);
         \draw [thick] (-1.68,0.13) circle(0.1);
    \draw [thick] (-2.1,0.5) circle(0.1);
    \draw [thick] (-1.7,0.4) circle(0.1);
    \draw [thick] (-1.1,0.3) circle(0.1);
     \draw [thick] (-1.32,0.19) circle(0.1);
      \draw [thick] (-1.37,0.59) circle(0.1);
    \draw [thick] (-0.9,0.56) circle(0.1);
    \draw [thick] (-0.3,0.65) circle(0.1);
    \draw [thick] (0.3,0.5) circle(0.1);
    \draw [thick] (0.5,0.30) circle(0.1);
    \draw [thick] (1.0,0.55) circle(0.1);
    \draw [thick] (1.8,0.15) circle(0.1);
     \draw [thick] (-2.8,0.22) circle(0.1);
    \draw [thick] (-2.5,0.3) circle(0.1);
     \draw [thick] (-0.47,0.21) circle(0.1);
    \draw [thick] (-0.25,0.1) circle(0.1);
    \draw [thick] (-0.03,0.27) circle(0.1);
       \draw [thick] (1.2,0.32) circle(0.1);
        \draw [thick] (1.5,0.18) circle(0.1);
        \draw [thick] (1.58,0.67 ) circle(0.1);
            \draw [thick] (0.95,0.14) circle(0.1);
    \draw [thick] (1.8,0.15) circle(0.1);
        \draw [thick] (1.84,0.37) circle(0.1);
          \draw [thick] (-0.8,0.12) circle(0.1);
             \draw [thick] (-0.6,0.42) circle(0.1);
              \draw [thick] (0.27,0.17) circle(0.1);
%    \draw [thick] (2.3,0.30) circle(0.25);
%    \draw [thick] (2.7,0.8) circle(0.25);
%    \draw [thick] (3.1,0.4) circle(0.25);
%    \draw [thick] (3.7,0.3) circle(0.25);
%    \draw [thick] (3.9,0.8) circle(0.25);
%    \draw [thick] (4.5,0.65) circle(0.25);
%    \draw [thick] (5.1,0.85) circle(0.25);
%    \draw [thick] (5.4,0.30) circle(0.25);
%    \draw [thick] (5.9,0.50) circle(0.25);
    \draw [color=color1,-,thick] (-2.5,0.30) -- (-2.4,0.30); 
      \draw [color=color1,{Stealth[length=2mm]}-{Stealth[length=2mm]},thick] (-3.4,0) -- (-3.4,0.8); 
    %%%%%%%
    \draw[color1] (-2.25,0.07) node[above]{$\varepsilon$} ;    
    \draw[color1] (-3.8,0.1) node[above]{$\varepsilon h$} ;
    %%%%%%%
   % \draw (5.8,2.0) node[above]{\small $k=\dfrac{2\pi}{\lambda}$} ;
   % \draw[color1] (5.8,1.4) node[above]{\small$\tcbhighmath[top=0.1cm, bottom=0.1cm,left=0.1cm,right=0.1cm]{\varepsilon \ll \lambda}$} ;
    %%%%%%%
   % \draw [color3] (4.4,3) node[above]{\small$u^\varepsilon-u_{inc}$ \textit{sortant}} ;
  \end{tikzpicture}  
 \end{center}
 \caption{Illustration of the geometry of the model}
 \end{figure}
For any $L>\eps h$, let us introduce 
$\Sigma_{L} \coloneqq   \{x_d=L\}$, $\mathcal{B}^{\omega}_{L}\coloneqq  \mathbb{R}^{d-1}\times(0,L) \setminus \overline{\mathcal{P}^\omega_\eps}$  the strip below $\Sigma_{L}$ where lies the set of particles (and which then depends on the particles' distribution), $\mathcal{B}^\infty_{L}\coloneqq  \mathbb{R}^{d-1}\times(L,+\infty)$ the half-space above $\Sigma_{L}$ that does not depend on the distribution and finally for $\eps h<L'<L$, $\mathcal{B}_{L',L}\coloneqq \mathcal{B}^\infty_{L'}\setminus \mathcal{B}^\infty_{L}$ the strip between $\Sigma_{L'}$ and $\Sigma_L$.

\subsection{Problem formulation}\label{sub:Pb}
For a given source function $f\in L^2(D_\eps^\omega)$ whose support is compact and lies far away from the layer, \textit{i.e.} there exists $L'<L\in(\eps h,+\infty)$ such that
\begin{equation}
	\label{eq:source}
 \textrm{supp} f \subset \mathcal{B}_{L',L} \quad\text{with}\quad \textrm{supp} f\;\text{compact};
\end{equation}
we look for the solution $u_\eps^\omega$ of the Helmholtz equation
\begin{equation}
	\label{eq:Helm}
	- \Delta u_\eps^\omega -k^2u_\eps^\omega=f\quad\text{ in } D^\omega_\eps 
\end{equation}
where $k>0$ is the wavenumber. The infinite plane models a multi-layer object through a Robin boundary condition so that
\begin{equation}
	\label{eq:Robin}
	-\partial_{x_d} u_\eps^\omega + \ii  k \gamma \  u_\eps^\omega = 0 \quad\text{ on } \ \Sigma_0,
\end{equation}
where the surface impedance coefficient $\gamma\in\C$ is such that \begin{equation}\label{eq:Regamma}\re{\gamma}>0.\end{equation}
At the boundary of the particles $\partial {\Pa}_\eps^\omega$, we impose to the field a homogeneous Dirichlet condition
\begin{equation}\label{eq:dir}
     u_\eps^\omega=0\quad \text{ on } \ \partial \Pa^\omega_\eps.
\end{equation}
Finally, the problem formulation has to be completed with a radiation condition. As in \cite{chandler2005existence,chandler2010variational} (where a Dirichlet boundary condition on the infinite hyper-plane is considered) and in \cite{Hu:2015,badenriess} (where an impedance boundary condition is considered), we make use of the so-called {\it upward propagating radiation condition} (UPRC). This outgoing wave condition is detailed in the following section. 

\subsection{Outgoing wave condition based on a Dirichlet-to-Neumann operator}\label{sub:outgoing}

To define the outgoing propagation radiation condition, we refer to \cite{chandler1999scattering,chandler2005existence} which consider 2D scattering by a rough Dirichlet surface (defined as the graph of a function) and \cite{badenriess} which considers the case of 2D or 3D scattering by a rough impedance surface.  

We introduce the fundamental solution of the Helmholtz equation in $\R^d$ given by
\[
  \begin{array}{|ll}
    \text{for}\; d=2,&\dst\Phi({\bf x},{\bf y})=\dst\frac{\ii}{4} H^{(1)}_0(k|\bsx-\bsy|)\\
    \text{for}\; d=3,&\dst\Phi({\bf x},{\bf y})=\frac{\exp{\big(\ii k|{\bsx}-\bsy|\big)}}{4\pi|\bsx-\bsy|}\; 
  \end{array}
 \quad \bsx,\bsy\in\R^d,\;\bsx\neq\bsy,
\]
where $H^{(1)}_0$ is the Hankel function of the first kind of order 0. The UPRC then states that
\begin{equation}\label{eq:UPRC}
	 u_\eps^\omega (\bsx)= 2\int_{\Sigma_L}\partial_{x_d} \Phi (\bsx, (\bsy_\shortparallel,L)) u_\eps^\omega\big|_{\Sigma_{L}} (\bsy_\shortparallel)d\bsy_\shortparallel,\quad \bsx\in \mathcal{B}_{L}^\infty.
\end{equation}
For this integral representation to be well-defined, it is not necessary for $u_\eps^\omega|_{\Sigma_{L}}$ to be in $L^2(\Sigma_{L})$. For example for $d=2$, $u_\eps^\omega|_{\Sigma_{L}}\in L^\infty(\Sigma_{L})$ is sufficient since $\partial_{x_d}\Phi = O({\bsy_\shortparallel^{-3/2}}) $. Nevertheless, in the case where $u_\eps^\omega|_{\Sigma_{L}}\in L^2(\Sigma_{L})$, which is the case as we shall see in the next section, \eqref{eq:UPRC} can be rewritten in terms of the Fourier transform of $u_\eps^\omega|_{\Sigma_{L}}$ as follows
\begin{equation}
  u^\omega_\eps(\bsx) \ = \ {(2\pi)^{-(d-1)/2}}\int_{\mathbb{R}^{d-1}} \widehat{\phi}_L(\boldsymbol{\zeta}) \  e^{\ii (x_d-L)\sqrt{k^2-|\boldsymbol{\zeta}|^2}} \ e^{i\bsx_\shortparallel\cdot \boldsymbol{\zeta}} \  \dd\boldsymbol{\zeta},\quad \bsx=(\bsx_\shortparallel, x_d)\in \mathcal{B}_{L}^\infty.
\end{equation}
where by convention $\sqrt{k^2-|\boldsymbol{\zeta}|^2}=\ii\sqrt{|\boldsymbol{\zeta}|^2-k^2}$ when $k< |\boldsymbol{\zeta}|$ and $$\widehat{\phi}_L(\boldsymbol{\zeta})\coloneqq (2\pi)^{-(d-1)/{2}}\int_{\mathbb{R}^{d-1}}\phi_L(\bsy_\shortparallel) \ e^{-\ii \bsy_\shortparallel\cdot\boldsymbol{\zeta}}\ \mathrm{d}\bsy_\shortparallel,\quad \boldsymbol{\zeta}\in \mathbb{R}^{d-1},$$ is the Fourier transform of $\phi_L\coloneqq u^\omega_\eps\big|_{\Sigma_{L}}$ (identifying $\Sigma_{L}$ and $\mathbb{R}^{d-1}$).

Let us recall the characterization of Sobolev spaces in terms of Fourier transforms: for $s>0$ 
\begin{equation}\label{eq:def_Hs}\begin{array}{ll}
\vsd&\dst  H^s\big(\Sigma_{ L}\big) \coloneqq \left \{ \varphi\in L^2\big(\Sigma_{L}\big),\  \big(|\boldsymbol{\zeta}|^2+1\big)^{s/2}\widehat{\varphi}\in L^2\big(\mathbb{R}^{d-1}\big)\right\},\\
\textrm{and}&\dst H^{-s}\big(\Sigma_{L} \big) \coloneqq  \Big( H^s\big(\Sigma_{L}\big)\Big)'=\left\{ \varphi\in\mathcal{S}'(\R^{d-1}), \ \Big(|\boldsymbol{\zeta}|^2+1 \Big)^{-s/2}\widehat{\varphi}\in L^2(\R^{d-1})\right\}.
\end{array}\end{equation}

\begin{definition}\label{definition_DtN_sol_ref_alea}
The Dirichlet-to-Neumann (DtN) operator $\Lambda^k:H^{\frac{1}{2}}\big(\Sigma_{L}\big)\to H^{-\frac{1}{2}}\big(\Sigma_{L}\big)$ is defined as
% \begin{equation}
%   \forall \varphi\in H^{1/2}\big(\Sigma_{L}\big)\ ,\quad \ \Lambda^k \varphi(\bsy_\shortparallel) \ = \ {(2\pi)^{-(d-1)/2}} \int_{\mathbb{R}^{d-1}} i\sqrt{k^2-|\boldsymbol{\zeta}|^2} \ \widehat{\varphi}(\boldsymbol{\zeta}) \ e^{i\bsy_\shortparallel\cdot\boldsymbol{\zeta}} \ , d\boldsymbol{\zeta},\quad \bsy_\shortparallel\in\Sigma_L
% \end{equation}
% and thus 
\begin{equation}\forall \varphi,\psi\in H^{\frac{1}{2}}\big(\Sigma_{L}\big), \quad\big\langle\Lambda^k \varphi,\psi \big\rangle_{\Sigma_{L}}= {(2\pi)^{-(d-1)/2}}\displaystyle  \int_{\mathbb{R}^{d-1}} \ii \sqrt{k^2-|\boldsymbol{\zeta}|^2}\ \widehat{\varphi}(\boldsymbol{\zeta}) \ \overline{\widehat{\psi}(\boldsymbol{\zeta})} \ \dd\boldsymbol{\zeta},\end{equation}
where here and in the sequel $\big\langle\cdot,\cdot \big\rangle_{\Sigma_{L}}$ denotes the sesquilinear duality product between $H^{-\frac{1}{2}}(\Sigma_{L})$ and $H^{\frac{1}{2}}(\Sigma_{L})$.
\end{definition}

The operator $\Lambda^k$ satisfies then the following properties (see for instance \cite{chandler2005existence}).

\begin{proposition}\label{prop:prop_operateur_DtN_alea}
The operator $\Lambda^k:H^{\frac{1}{2}}\big(\Sigma_{L}\big)\to H^{-\frac{1}{2}}\big(\Sigma_{L}\big)$ is a continuous operator such that
\begin{equation}\forall\varphi\in H^{\frac{1}{2}}\big(\Sigma_{L}\big),\quad
     \re{\big\langle\Lambda^k \varphi,\varphi  \big\rangle_{\Sigma_{L}} }\leq 0 \quad \text{and}\quad 
     \im{\big\langle\Lambda^k \varphi , \varphi \big\rangle_{\Sigma_{L}}} \geq 0.
\end{equation}
\end{proposition}

% Since we impose a Robin condition with $\re{\gamma}>0$ on $\Sigma_0$, and the source term is assumed to be $L^2$, we can prove that there exists a unique solution to the problem \eqref{eq:Helm}, \eqref{eq:Robin}, \eqref{eq:dir} completed with \eqref{eq:UPRC} such that $u^\omega_\eps\big|_{\mathcal{B}^{\omega,0}_{L}}\in H^1\big(\mathcal{B}^{\omega,0}_{L}\big)$ and $\forall L,L'>0$ such that $L'>L>\eps(h)$, $u^\omega_\eps\Big|_{D_\eps^{L,L'}}\in H^1\big(D_\eps^{L,L'}\big)$ with $D_\eps^{L,L'}=\mathbb{R}^{d-1}\times ( L, L')$, so that $u^\omega_\eps\big|_{\Sigma_{L}}\in H^{1/2}\big(\Sigma_{L}\big)$. We detail the aforementioned proof in the next section.

%%%%%%%%%%%%%%%%%%%%%%%%%%%%%%%
\subsection{Well-posedness of the scattering problem}
%%%%%%%%%%%%%%%%%%%%%%%%%%%%%%%

The problem given in Section \ref{sub:Pb} and \ref{sub:outgoing} can be rewritten as follows. 

Find $u_\eps^\omega$ in $H^1\big(\mathcal{B}^{\omega}_{L}\big)$ such that \begin{equation}\label{eq:P0eps}
\begin{array}{|ll}
    -\Delta u^\omega_\eps-k^2u^\omega_\eps = f \quad &\text{ in } \quad \mathcal{B}^{\omega}_{L} \ ,  \\ 
    -\partial_{x_d}u^\omega_\eps +\ii k \gamma u^\omega_\eps = 0&\text{ on } \quad \Sigma_0 \ ,  \\
      u_\eps^\omega  =  0 & \text{ on } \quad \partial\mathcal{P}_\eps^\omega,\\
     -\partial_{x_d} u^\omega_\eps + \Lambda^k u^\omega_\eps = 0 & \text{ on } \quad \Sigma_{L} \ ,
\end{array}
\end{equation}
Let $H^1_0(\mathcal{B}^{\omega}_{L})\coloneqq \Big\{u\in H^1\big(\mathcal{B}^{\omega}_{L}\big), \ u=0\;\text{on}\; \partial \mathcal{P}^\omega_\eps\Big\}$. The variational formulation associated with \eqref{eq:P0eps} is : 

Find $u^{\omega}_\eps \in H^1_0(\mathcal{B}^{\omega}_{L})$ such that 
\begin{equation}\tag{FV}\label{FV_alea_ref}
\forall v \in H^1_0(\mathcal{B}^{\omega}_{L}),\quad a_\eps^\omega\big(u^{\omega}_\eps,v \big) \ = \ l_\eps^\omega(v ) \ \ , \quad 
\end{equation}
with
\begin{equation}\label{eq:a_epsomega}
  \forall u,v \in H^1_0(\mathcal{B}^{\omega}_{L}),\quad a_\eps^\omega(u,v) \ \coloneqq  \ \int_{\mathcal{B}^{\omega}_{L}} \nabla u \cdot \nabla\overline{v} \ - \  k^2\int_{\mathcal{B}^{\omega}_{L}} u \ \overline{v}- \ \big<\Lambda^k u , v\big>_{\Sigma_{L}} \ + \ \ii k\gamma \int_{\Sigma_0}u \ \overline{v},
\end{equation}
and 
\begin{equation*}
  \forall v \in H^1_0(\mathcal{B}^{\omega}_{L}),\quad l_\eps^\omega(v) \ \coloneqq  \ \int_{\mathcal{B}^{\omega}_{L}} f \ \overline{v}. 
\end{equation*}

The standard arguments to prove that the problem is of Fredholm type, \emph{i.e.} that the operator associated with the sesquilinear form $a_\eps^\omega$ is a compact perturbation of a coercive operator, do not apply here because, on contrary to a periodic distribution of particles, the problem cannot be reformulated in a bounded domain. Let us mention that in \cite{chandler1997impedance}, the author  proves that the scattering problem by a 2D plane with a Robin boundary condition (with $\gamma\in L^\infty$ such that $\mathcal{R}e(\gamma)\geq \eta>0$) is well posed, using an integral equation reformulation of the problem . This result has been extended in \cite{zhang2003integral} to the case of a rough $C^{1,1}$-surface. 
In the case of a more general rough Lipschitz surface, the well-posedness has been proved in \cite{baden2020existence}. The proof in this last paper can be directly extended to our setting. 

%
%In this last paper, the author considers first the low frequency case (low with respect to the surface variations). We have extended this proof to our configuration where the particles are small compared to the wavelength (which corresponds somehow to the low frequency case). We prove that in this setting the problem is coercive.

\begin{theorem}{\cite[Lemma 15]{baden2020existence}}\label{thm:ref_alea}
   Let $\ell_\eps^\omega$ be in $[H^1_0(\mathcal{B}^{\omega}_{L})]'$, \emph{i.e.} a continuous antilinear form on $H^1_0(\mathcal{B}^{\omega}_{L})$. For all $\eps>0$ and $\omega$, there exists a unique $u_\eps^\omega \in H^1_0(\mathcal{B}^{\omega}_{L})$ solution of 
   \begin{equation*}
    \forall v \in H^1_0(\mathcal{B}^{\omega}_{L}),\quad a_\eps^\omega\big(u^{\omega}_\eps,v \big) \ = \ \ell_\eps^\omega(v ),
    \end{equation*}
     where $a_\eps^\omega$ is defined in \eqref{eq:a_epsomega}. Moreover, there exists $C>0$ that depends only on $k$, $L$ and $\gamma$ such that
   \begin{equation}
       \|u^\omega_\eps \|_{H^1(\mathcal{B}^{\omega}_{L})} \leq C \| \ell^\omega_\eps\|_{[H^1_0(\mathcal{B}^{\omega}_{L})]'}.
   \end{equation}
\end{theorem}

\subsection{Random setting}\label{sub:random_setting}
%%%%%%%%%%%%%%%%%%%%%%%%%%%%%%%

We present in this section the different assumptions that we make on the random distribution of particles covering the object. \vsd

Let us consider the infinite normalized strip $\mathcal{L}\coloneqq \R^{d-1}\times (0, h)$. \vsd Let $\{{\bsx}^\omega_n\}_{n\in\mathbb{N}}$ denote the point process corresponding to the centers of the particles. For $n\in\mathbb{N}$ we denote by $B(\bsx_n^{\omega})$ the particle with radius 1 centered at $\bsx^\omega_n$. We suppose in the whole paper the following 

\begin{itemize}
\item[-] for all $n\in\N$, $B(\bsx_n^{\omega}) \in \mathcal{L}$,
\item[-] the particles lie at least at a distance of $\delta$ from one another and from $\Sigma_0$,
\item[-]$\{{\bsx}^{\omega}_n\}_{n\in\mathbb{N}}$ is stationary, \textit{i.e.} the distribution law $\Pro$ of $\{{\bsx}^{\omega}_n\}_{n\in\mathbb{N}}$ is independent of the spatial variable. \end{itemize}
Let $\Omega$ be the set of point processes in the strip such that the set of particles satisfy the previous assumptions. We equip $\Omega$ with its cylindrical $\sigma$-algebra $\F$.\vsd

We introduce the translation operator $\tau : \Omega\times \R^{d-1}\to \Omega$ defined as follows 
\begin{equation*}
     \forall (\omega, \bsypar)\in\Omega\times\R^{d-1}, \quad\tau(\omega,\bsypar)=\omega',\quad\quad  \text{where}\quad \bsx_n^{\omega'}=\bsx_n^{\omega}+\bsypar,\quad \text{for all }n\in \mathbb{N}. 
\end{equation*} 
Note that $\tau$ is measurable on $\Omega\times \R^{d-1}$. In what follows we denote $\tau_{\bsypar}\coloneqq \tau(\cdot,\bsypar)$ for $\bsypar\in\R^{d-1}$.\vsd

The action $(\tau_{\bsxpar})_{\bsxpar\in \R^{d-1}}$ of the group $(\R^{d-1}, +)$ on $(\Omega, \F, \Pro)$ verifies then
$$\forall \bsxpar,\bsypar\in\R^{d-1}, \quad \tau_{\bsxpar+\bsypar}=\tau_{\bsxpar}\circ\tau_{\bsypar},$$
$$\forall \bsxpar\in\R^{d-1},\forall A\in \F,\quad \Pro(\tau_{\bsxpar}A)=\Pro(A)$$
where the last property means that $\tau$ preserves the law distribution and is due to the stationarity of the point process.
Equipped with this action on $(\Omega, \F, \Pro)$, we can now state the definition of a stationary process defined on $\Omega\times\R^{d-1}$.

\begin{definition}[Stationarity]\label{definition:stationary}$f:\R^{d-1}\times\Omega\to\R^p$ is said to be stationary with respect to $\tau$ if 
\begin{equation*}\forall \bsxpar,\bsypar\in\R^{d-1}, \quad f^\omega(\bsxpar+\bsypar)=f^{\tau_{\bsypar}\omega}(\bsxpar),\quad \omega\in\Omega.
\end{equation*}
\end{definition} It is easy to show that if $f$ is stationary then
\begin{equation*}
\E[f]=\E\left[\int_{\square_1} f(\bsxpar)d\bsxpar\right]=\left[\fint_{\square_R} f(\bsxpar)d\bsxpar\right],
\end{equation*}for all $R>0$.
We suppose moreover that the action $(\tau_{\bsxpar})_{\bsxpar\in\R^{d-1}}$ is ergodic. Let us recall the definition of an ergodic action below.
 \begin{definition}[Ergodicity]$(\tau_{\bsxpar})_{\bsxpar\in\R^{d-1}}$ is said to be ergodic if any $\tau$-invariant event has probability 0 or 1, \textit{i.e.}
\begin{equation*}\forall A\in\F,\quad (\forall \bsxpar\in\R^{d-1}, \tau^{-1}_{\bsxpar} A=A)\implies (\Pro(A)\in\{0,1\}).
\end{equation*}\end{definition}\vsd
An easy way to grasp the concept of ergodicity is through Birkhoff ergodic Theorem.
\begin{theorem}[Birkhoff ergodic Theorem]\label{theorem:Birkhoff}Let $f$ be a stationary process with respect to an ergodic action $(\tau_{\bsxpar})_{\bsxpar\in\R^{d-1}}$ such that $f\in L^p(\Omega, L^1_{loc}(\R^{d-1}))$. Then almost surely (\textit{a.s.}) and in $L^p(\Omega)$ for any measurable bounded domain $K\subset\R^{d-1}$\vspace{-0.2cm}\begin{equation*}\forall \bsx\in\R^d,\quad \fint_{K_t}f^\omega(\bsypar+\bsxpar)\dd \bsypar\xrightarrow[t\to \infty]{}\E[f],\end{equation*}where $K_t\coloneqq \left\{\bsxpar\in\R^{d-1}|t^{-1}\bsxpar\in K\right\}$ denotes the homothetic dilation with ratio $t>0$ of $K$.\end{theorem}
In other words, ergodicity implies that averaging with respect to the randomness corresponds to averaging spatially.
We refer to \cite{shiryaev1984probability,reuter1986u} for the proof of this theorem.\vsd

For $\bsy\in\R^d$ and $\omega\in\Omega$, let $R^\omega(\bsy)$ be the distance from $\bsy$ to the nearest $\bsx^{\omega}_n$, \textit{i.e.}
\begin{equation}\label{eq:def_R}
R^\omega : \left\{ 
\begin{aligned}
  \mathbb{R}^{d} & \to \  \mathbb{R}_+ \ , \\
   \bsy & \mapsto \  \min_{n\in \mathbb{N}} |\bsy-\bsx_n^\omega|. 
\end{aligned}
\right.
\end{equation}
By stationarity of $\big\{\bsx_n \big\}_{n\in \mathbb{N}}$, for all $y_d>0$, $(\omega, \bsypar)\mapsto R^\omega(\bsypar,y_d)$ is stationary.  

\begin{hypothesis}[H1]
\makeatletter\def\@currentlabelname{(H1)}\makeatother
\label{hyp:mixing}There exists $m>2d$ such that \begin{equation*}\sup_{y_d\in[0,h]}\quad \E[R(\cdot, y_d)^m]<+\infty.\end{equation*}\end{hypothesis}
Intuitively Hypothesis \nameref{hyp:mixing} implies that there exist arbitrarily large portions of $\mathcal{L}$ without particle but the probability to sample a large portion without particle is small. Classical hard-core point processes such as hard-core Poisson point process or random parking satisfy stationarity ergodicity and hypothesis \nameref{hyp:mixing} (see \cite{DuerinckxGloria2017}). Note that the necessity for $m$ to be greater than $2d$ is important to obtain Lemma \ref{lem:weightedtrace}.

Let us define the weight
\begin{equation}\label{eq:h}\mu^\omega(\bsy)\coloneqq \begin{array}{|ll}R^\omega(\bsy_\shortparallel, y_d)^{-m}&\text{if}\,\,y_d\in[0, h],\\(y_d^2+R^\omega(\bsy_\shortparallel, h)^{2m})^{-1}&\text{if}\,\,y_d> h,\end{array}\end{equation} for $\bsy\in D^\omega$ and $\omega\in\Omega$. Assumption \nameref{hyp:mixing}  can then be rewritten as: there exists 
C independent of $y_d$ such that \begin{equation}\label{eq:Eh-1}\forall y_d>0,\quad \quad\E[\mu(\cdot, y_d)^{-1}]\leq C y_d^2.\end{equation}

%\begin{remark} 
%  The analysis remains unchanged for more complex inclusion processes for examples with independent and identically distributed radii  or independent and identically distributed shapes. 
%\end{remark}
%

The set of normalized particles is denoted $\dst\Pa^\omega\in \mathcal{L}$ and the set of scaled particles of size $\eps$, $\dst\Pa^\omega_\eps=\eps\Pa^\omega\subset\mathcal{L}^\eps=\eps \mathcal{L}$.

%%%%%%%%%%%%%%%%%%%%%%%%%%%%%%%
\subsection{Well-posedness of the scattering problem in the random setting}
%%%%%%%%%%%%%%%%%%%%%%%%%%%%%%%

For any Banach space $V$, we define $L^2\big(\Omega, V\big)$ the space of measurable applications $u:\Omega\to V$ such that
$\mathbb{E}\left[ \|u \|^2_V\right] <+\infty$, that we equip with the usual norm
\begin{equation*}\forall \ u \in L^2\big(\Omega, V\big), \quad \| u\|_{L^2(\Omega,V)}^2 \coloneqq  \mathbb{E}\left[ \big\|u\big\|^2_V \right].\end{equation*}
Since the constant  in Theorem \ref{thm:ref_alea} is independent of $\omega\in\Omega$, the following result can be easily deduced.
\begin{corollary}\label{cor:wpue}
Let $\ell_\eps$ be a continuous antilinear form on $L^2\big(\Omega,H^1_0(\mathcal{B}^{\omega}_{L})\big)$ then for all $\eps>0$, there exists a unique $u_\eps \in L^2\big(\Omega,H^1_0(\mathcal{B}^{\omega}_{L})\big)$ such that for a.e. $\omega\in\Omega$, $u_\eps^\omega$ is solution of \eqref{FV_alea_ref}. Moreover, there exists $C>0$ depending only on $k$, $L$ and $\gamma$ such that 
\begin{equation}
\mathbb{E}\left[\|u_\eps^\omega\|^2_{H^1(\mathcal{B}^{\omega}_{L})} \right]\ \leq \ C \   \mathbb{E}\Big[\|\ell_\eps\|_{\big[H^1_0(\mathcal{B}^{\omega}_{L})\big]'}^2 \Big].
\end{equation}
\end{corollary}

%%%%%%---------------------------------
\section{A formal asymptotic expansion}\label{sec:formalasympexp}
%%%%%%---------------------------------

To build approximations of $u_\eps^\omega$ at different orders of $\eps$ we use multiscale expansion techniques. The first step is to propose a multiscale aymptotic expansion as an Ansatz for $u_\eps^\omega$. We then inject this Ansatz in the equations verified by $u_\eps^\omega$ and obtain formally the equations verified by the different terms of the expansion. This step is detailed in this section.\vsd

For a.e. $\omega\in\Omega$, we divide the upper half-space $D^\omega_\eps$ in two regions : 
\begin{itemize}
\item[-]the domain $\mathcal{B}_{\eps H}^{\omega}\coloneqq \R^{d-1}\times (0,\eps H)\setminus \overline{\Pa}_\eps^\omega$ below the plane $\Sigma_{\eps H} \coloneqq \{x_d=\eps H\}$, and 
\item[-]the domain $\mathcal{B}_{\eps H}^{\infty}\coloneqq  \R^{d-1}\times (\eps H,+\infty)$ above $\Sigma_{\eps H}$. 
\end{itemize} 
Here $H>h$ is a parameter, the only constraint being that the set of particles is always in the domain $\mathcal{B}_{\eps H}^{\infty}$. We will explain in Section \ref{subsection:effective} how this parameter should be chosen.

We propose the following Ansatz for the solution $u_\eps$. For all $\omega\in\Omega$, 
 
\begin{equation}\label{dvt_asymptotique}
  \begin{aligned}
      u_\eps^\omega(\bsx) &= \sum_{n\in \mathbb{N}} \eps^n \  U_n^{\omega,NF}\left(\bsx_\shortparallel;\dfrac{\bsx_\shortparallel}{\eps},\dfrac{x_d}{\eps}\right)  \ + \  \sum_{n\in \mathbb{N}} \eps^n \  u_n^{\omega,FF}\left(\bsx\right),& \quad \bsx\in \mathcal{B}_{\eps H}^{\infty} \\
      u_\eps^\omega(\bsx) &= \sum_{n\in \mathbb{N}} \eps^n \  U_n^{\omega,NF}\left(\bsx_\shortparallel;\dfrac{\bsx_\shortparallel}{\eps},\dfrac{x_d}{\eps}\right),&  \quad \bsx\in \mathcal{B}^{\omega}_{\eps H}.
  \end{aligned}
\end{equation}
where
\begin{itemize}
    \item[-] the so called far-field (FF) terms $u_n^{\omega,FF}$  defined only in $\mathcal{B}_{\eps H}^{\infty}$ depend only on the macroscopic variable $\bsx\coloneqq (\bsx_\shortparallel,x_d)$ and \textit{a priori} on the random distribution;     
    \item[-] the so-called near-field (NF) terms $U_n^{\omega,NF}$ depend on $\bsx_\shortparallel$ and on a microscopic variable $\bsy\coloneqq (\bsy_\shortparallel,y_d)=\left(\eps^{-1}{\bsx_\shortparallel}, \eps^{-1}{x_d}\right)$.
\end{itemize}\vsd
Given the stationarity of the particles distribution, we suppose that the near-field terms are stationary with respect to the tangential microscopic variables 
\begin{equation}
    \forall n\in\N,\;\forall \bsx_\shortparallel\in\R^{d-1},\; \forall y_d \in \mathbb{R}^+,\quad (\omega,\bsy_\shortparallel)\mapsto U_n^{\omega,NF}(\bsx_\shortparallel; \bsy_\shortparallel, y_d) \quad\text{is stationary.}
    \end{equation}
    We want for $u_\eps^{\omega}$ to be well approximated by the far-field terms far away from the particles so we impose
\begin{equation}
    \forall n\in\N,\;\forall \bsx_\shortparallel \in \R^{d-1},\;\forall \bsy_\shortparallel\in\R^{d-1},\quad\dst\lim_{y_d\to+\infty}U_n^{\omega,NF}(\bsx_\shortparallel;\bsy_\shortparallel,y_d)=0\;\text{a.s}.
\end{equation}After injecting the Ansatz into \eqref{eq:P0eps}, we obtain the following :
\begin{itemize}
    \item For all $n\in \mathbb{N}$, the far-field term $u^{\omega, FF}_n$ should satisfy :
    \begin{equation}\label{eq:cascade_champ_loin_alea}
        \begin{array}{|lr}
         -\Delta u_n^{\omega,FF}  -   k^2  u_n^{\omega,FF}=f\delta_{0,n} &  \text{ in } \quad \mathcal{B}_{\eps H,L}\\
         -\partial_{x_d}  u_n^{\omega,FF} + \Lambda^k  u_n^{\omega,FF} = 0 & \text{ on } \quad \Sigma_{L} 
        \end{array}
    \end{equation}
    with $\delta_{0,n}=1$ if $n=0$, $\delta_{0,n}=0$ otherwise and $L>\eps H$ (that is always possible since the interface where the DtN condition is imposed is artificial). A boundary condition for $u^{\omega, FF}_n$ on $\Sigma_{\eps H}$ is missing for now and will be derived by studying the near field terms. 
   \item  The near-field terms $U_n^{\omega, NF}$ are solutions of Laplace-type problems posed in an infinite half-space containing particles of size 1 following a stationary and ergodic distribution,  $\bsx_\shortparallel$ playing the role of a parameter: 
for all $n\in\mathbb{N}$, for a.e.  $\bsx_\shortparallel\in\R^{d-1}$, the near-field term $\bsy\mapsto U^{\omega,NF}_n(\bsx_\shortparallel;\bsy)$ verifies 
\begin{equation}\label{eq:NF}
\begin{array}{|l}
  -\Delta_{\bsy} U^{\omega,NF}_n = \ 2 \nabla_{\bsx_\shortparallel}\cdot \nabla_{\bsy_\shortparallel} U^{\omega,NF}_{n-1}\ + \ \Delta_{\bsx_\shortparallel} U^{\omega,NF}_{n-2}\ + \ k^2 \ U^{\omega,NF}_{n-2}\quad\text{ in } \quad\mathcal{B}_{H}^{\omega} \cup \mathcal{B}_{ H}^{\infty}, \\
  \partial_{y_d} U^{\omega,NF}_n = \ ik\gamma \ U^{\omega,NF}_{n-1}\;\text{ on } \; \Sigma_0, \quad\text{and}\quad
  U^{\omega,NF}_n \ = \ 0\; \text{ on } \; \partial \mathcal{P}^\omega, \\
  \Big[U^{\omega,NF}_n \Big]_{\Sigma_H} = \ -u^{\omega, FF}_n\big|_{\Sigma_{\eps H}}(\bsx_\shortparallel) \;\;\text{and}\;\;
  \Big[-\partial_{y_d}U^{\omega,NF}_n \Big]_{\Sigma_H} = \ \partial_{x_d}u^{\omega, FF}_{n-1}\big|_{\Sigma_{\eps H}}(\bsx_\shortparallel), 
\end{array}
\end{equation}
where $u^{\omega, FF}_{-1}=0$ and $U^{\omega,NF}_ {-1}=U^{\omega,NF}_{-2}=0$ by convention and $$[v]_{\Sigma_H}\coloneqq \left(v\big|_{\mathcal{B}_{H}^{\infty}}\right)\Big|_{\Sigma_H}-\left(v\big|_{\mathcal{B}_{H}^{\omega}}\right)\Big|_{\Sigma_H}.$$\vsd
\end{itemize}
The functional framework and the assumptions on the source terms for which these problems \eqref{eq:NF} set in an unbounded domain are well posed are not straight-forward. Neither is the condition for which the solution tends to $0$ at infinity. The next two sections are devoted to the analysis of this problem. First in section  \ref{sec:NF_pbs} we study the well-posedness of the near field problems. Next in section
\ref{sec:limbehavior}, we focus on the behavior at infinity of the near-fields.

%%%%%%---------------------------------
\section{Well-posedness  of the near-field problems}\label{sec:NF_pbs}
%%%%%%---------------------------------

The near-field problems \eqref{eq:NF} can be rewritten as follows: for $n \geq 0$, we look for a stationary  $U^{\omega,NF}_n$ solution for a.e. $\bsx_\shortparallel\in \mathbb{R}^{d-1}$ and $\omega\in\Omega$ of 
\begin{equation}\label{eq:NFGen}\begin{array}{|l}
  -\Delta_{\bsy} U^{\omega,NF}_n = F^\omega_{n-2}+\nabla \cdot \bsG^{\omega}_{n-1} \quad\text{ in } \quad\mathcal{B}_{H}^{\omega} \cup \mathcal{B}_{ H}^{\infty}\ , \\
  \partial_{y_d} U^{\omega,NF}_n = \Psi^{\omega}_{n-1}\;\text{ on } \; \Sigma_0, \quad\text{and}\quad
  U^{\omega,NF}_n \ = \ 0 \;\text{ on } \; \partial \mathcal{P}^\omega , \\
  \Big[U^{\omega,NF}_n \Big]_{\Sigma_H} = \alpha^{\omega, D}_{n}\;\;\text{and}\;\;
  \Big[-\partial_{y_d}U^{\omega,NF}_n \Big]_{\Sigma_H} = \alpha^{\omega, N}_{n-1}, 
\end{array}\end{equation}
where $F_{n-2}(\bsx_\shortparallel; \ \cdot, \ y_d )$, $\bsG_{n-1}(\bsx_\shortparallel; \ \cdot , \ y_d )$ and $H_{n-1}(\bsx_\shortparallel; \ \cdot )$ are stationary for any $y_d>0$. $\mathbf{F}^{\omega}_{n-2}(\bsx_\shortparallel; \ \cdot \ )$ depends on the previous term $U^{\omega,NF}_{n-2}(\bsx_\shortparallel; \ \cdot \ )$ and $\bsG^{\omega}_{n-1}(\bsx_\shortparallel; \ \cdot , \ y_d )$  and $\Psi^{\omega}_{n-1}(\bsx_\shortparallel; \ \cdot)$ on $U^{\omega,NF}_{n-1}(\bsx_\shortparallel; \ \cdot \ )$. $\alpha^{\omega, D}_{n}(\bsx_\shortparallel)$ and  $\alpha^{\omega, N}_{n-1}(\bsx_\shortparallel)$ depend respectively on the far field terms $u^{\omega,NF}_n(\bsx_\shortparallel)$ and $\partial_{x_d}u^{\omega, FF}_{n-1}(\bsx_\shortparallel)$.  Furthermore, $\alpha^{\omega, D}_n(\bsx_\shortparallel)$ and $\alpha^{\omega, FF}_{n-1}(\bsx_\shortparallel)$ are functions independent  of $\bsy_\shortparallel$. \vsd

In \eqref{eq:NFGen} $\bsx_\shortparallel$ plays the role of a parameter. For convenience we will omit the dependency on $\bsx_\shortparallel$ in what follows. After lifting the jump of the solution across $\Sigma_H$ we are interested in the well-posedness of the problems verified by $\widetilde{U}^{\omega, NF}_n\coloneqq U^{\omega,NF}_n-\alpha_n^{\omega, D} \ \mathbb{1}_{\mathcal{B}^\infty_H}$, \textit{i.e.} problems of the form
\begin{equation}\label{eq:NFtGen}
\begin{array}{|ll}
   -\Delta_{\bsy} \widetilde{U}^{\omega} =F^\omega+\nabla \cdot \bsG^{\omega}\quad \text{ in } \quad\mathcal{B}_{H}^{\omega} \cup \mathcal{B}_{ H}^{\infty}, \\
  -\partial_{y_d} \widetilde{U}^{\omega}=\Psi^{\omega} \; \text{ on } \; \Sigma_0, \quad\text{and}\quad
   \widetilde{U}^{\omega}= \ 0 \; \text{ on } \; \partial \mathcal{P}^\omega, \\
 \Big[\widetilde{U}^{\omega}  \Big]_H = \ 0 \quad \text{ and }\quad   \Big[-\partial_{y_d}\widetilde{U}^{\omega}  \Big]_H = \ \alpha^{\omega,N},
\end{array}
\end{equation}for given sources $F$, $G$, $\Psi$ and $\alpha^N$. Since one can find arbitrary large boxes $\square_R\times[0,H]$ without Dirichlet particles and since $\mathcal{B}_H^\infty$ is also unbounded im the $y_d$-direction, we cannot find a solution in the classical $L^2$-framework. Instead we consider weighted Sobolev spaces with weight $\mu$ previously introduced in \eqref{eq:h}.

Let us introduce the following weighted functional space
\begin{multline}\label{eq:w0D}
\mathcal{W}_0(D)\coloneqq  \left\{ \text{a.s. } v^\omega\in H_{loc}^1\big(D^\omega\big), \ \mathbb{1}_{D}v\big(\cdot,y_d\big) \text{ stationary for any }y_d>0, \right.\\\left.\dst \text{a.s. }v^\omega=0\,\text{on }\partial \Pa^\omega, \,\mathbb{E}\left[ \int_{\square_1}\int_0^{+\infty} \mathbb{1}_{D} \ (\mu(\bsy)|v(\bsy)|^2+|\nabla v(\bsy)|^2) \dd\bsy\right]<+\infty \right\},
\end{multline}
equipped with the norm
\begin{equation*}
\forall v\in\mathcal{W}_0(D), \quad \|v\|^2_{\mathcal{W}_0(D)}\coloneqq \mathbb{E}\left[ \int_{\square_1}\int_0^{+\infty} \mathbb{1}_{D} \ (\mu(\bsy)|v(\bsy)|^2+|\nabla v(\bsy)|^2) \dd\bsy\right].
\end{equation*}
We introduce the following functional spaces adapted to our source terms
\begin{multline*}
\mathcal{L}^2(D)\coloneqq  \left\{ \text{a.s. } v^\omega\in L_{loc}^2\big(D^\omega\big), \ \mathbb{1}_{D}v\big(\cdot,y_d\big) \text{ stationary for any }y_d>0,\right.\\\left.\mathbb{E}\left[ \int_{\square_1}\int_0^{+\infty} \mathbb{1}_{D} \ |v|^2(\bsy) \dd\bsy\right]<+\infty \right\},\end{multline*}
and for $t\geq 0,$ $\dst\mathcal{L}^2\big(\Sigma_t\big)\coloneqq \left\{ \varphi\in L^2_{loc}\big(\Sigma_t, L^2(\Omega)\big), \varphi \text{ stationary}, \ \mathbb{E}\left[ |\varphi|^2 \right]< +\infty\right\}.$ For all $v\in \mathcal{L}^2(D)$, let us denote $\dst\|v\|^2_{\mathcal{L}^2(D)}\coloneqq \mathbb{E}\left[ \int_{\square_1}\int_0^{+\infty} \mathbb{1}_{D} \ |v|^2(\bsy) \dd\bsy\right]$. 

\begin{remark}Note that $\mathcal{W}_0(D)\not\subset \mathcal{L}^2(D).$\end{remark}
Let us write the variational formulation associated to \eqref{eq:NFtGen} in $\mathcal{W}_0(D)$. We look for $\widetilde{U}$ in $\mathcal{W}_0(D)$ such that for any $V\in \mathcal{W}_0(D)$, 
\begin{equation}\label{FV_Dirichlet_pb_releve_borne_regularise}
  a(\widetilde{U},V)=l(V)
\end{equation}
where $a:\mathcal{W}_0(D)\times \mathcal{W}_0(D) \to \mathbb{C}$ is the sesquilinear form  defined for all $U,V\in\mathcal{W}_0(D)$ by
\begin{equation*}\label{eq:bil_NF}
a(U,V)\coloneqq \mathbb{E}\left[\int_{\square_1}\int_{0}^L \mathbb{1}_{D} \Big[ \nabla {U}(\bsy)\cdot \nabla\overline{V} (\bsy) \Big]\,\dd\bsy\right],
\end{equation*}
and $l:\mathcal{W}_0(D)\to  \mathbb{C}$ the antilinear form defined for all $V\in\mathcal{W}_0(D)$ by
\begin{equation*}\label{eq:lin_NF}\begin{multlined}
l(V)\coloneqq \E\left[\int_{\square_1}\int_0^L \mathbb{1}_{D} (F^\omega\overline{V}^\omega+\bsG^\omega\cdot\nabla\overline{V}^\omega)(\bsy)\,\dd\bsy+\int_{\square_1} (\Psi^\omega \ \overline{V^\omega}\big|_{\Sigma_0} + \alpha^{N} \ \overline{V}^\omega\big|_{\Sigma_H})(\bsypar)\,\dd\bsypar\right].
\end{multlined}\end{equation*}
The objective of the next two subsections is to show the following proposition.
\begin{proposition}\label{prop:Utnwellposed}
Suppose that $\Pa$ verifies hypothesis \nameref{hyp:mixing} then for all $\mu^{-\frac{1}{2}}F\in\mathcal{L}^2(D)$, $\bsG\in\mathcal{L}^2(D)^n$, $\mu^{-\frac{1}{2}}\Psi\in\mathcal{L}^2(\Sigma_0)$ and $\mu^{-\frac{1}{2}}\alpha^N\in \mathcal{L}^2(\Sigma_H)$,  Problem \eqref{FV_Dirichlet_pb_releve_borne_regularise} is well posed in $\mathcal{W}_0(D)$. Moreover the unique solution $\widetilde{U}\in\mathcal{W}_0(D)$ verifies
\begin{equation}\label{eq:estUt}
\|\widetilde{U}\|_{\mathcal{W}_0(D)}\lesssim \|\mu^{-\frac{1}{2}}\Psi\|_{\mathcal{L}^2(\Sigma_0)}+\|\mu^{-\frac{1}{2}}\alpha^{N}\|_{\mathcal{L}^2(\Sigma_H)}+\|\mu^{-\frac{1}{2}}F\|_{\mathcal{L}^2(D)}+\|G\|_{\mathcal{L}^2(D)}.
\end{equation}
\end{proposition}
\begin{remark}
One can show that the unique solution  in $\mathcal{W}_0(D)$ of Problem \eqref{FV_Dirichlet_pb_releve_borne_regularise} satisfies almost surely Problem \eqref{eq:NFtGen} in the sense of distributions. That is why, in what follows, we designate the unique solution of \eqref{eq:NFtGen}  in $\mathcal{W}_0(D)$ the variational solution.
\end{remark}
\subsection{Functional inequalities in $\mathcal{W}_0(D)$}

We start by proving that $(\mathcal{W}_0(D), \|\cdot\|_{\mathcal{W}_0(D)})$ is a Hilbert space.

\begin{lemma}Assume that $\Pa$ verifies hypothesis \nameref{hyp:mixing} then $(\mathcal{W}_0(D), \|\cdot\|_{\mathcal{W}_0(D)})$ is a Hilbert space.
\end{lemma}

\begin{proof}It is easy to see that $\mathcal{W}_0(D)$ is a Pre-Hilbert space. Let us prove completeness. For all $R>0$ and $\omega\in\Omega$, let $D^\omega_R\coloneqq D^\omega\cap(\square_R\times (0,+\infty)).$
Since a.s. $\mu^{-1}\in L^1_{loc}(D^\omega_R)$ , the weighted Sobolev space
\begin{equation*}
W_0(D^\omega_R)\coloneqq \left\{v\in H_{loc}^1\big(D^\omega_R\big), v=0\,\text{on }\partial \Pa^\omega, \int_{\square_R}\int_0^{+\infty} \mathbb{1}_{D} \ (\mu(\bsy)|v(\bsy)|^2+|\nabla v(\bsy)|^2) \dd\bsy<+\infty\right\}
\end{equation*}
equipped with the norm
\begin{equation*}
\forall v\in W_0(D^\omega_R), \quad \|v\|_{W_0(D^\omega_R)}\coloneqq  \int_{\square_R}\int_0^{+\infty} \mathbb{1}_{D} \ (\mu(\bsy)|v(\bsy)|^2+|\nabla v(\bsy)|^2) \dd\bsy
\end{equation*}
is a Hilbert space \cite[Theorem 1.11]{kufner1984define}. Therefore $(L^2(\Omega, W_0(D^\omega_R)), \|\cdot\|_{L^2(\Omega, W_0(D^\omega_R))})$ is also a Hilbert space. 

Let $(v_n)_{n\in\inte}$ be a Cauchy sequence in $\mathcal{W}_0(D)$. By stationarity $(v_n)_{n\in\inte}$ is a Cauchy sequence in  $L^2(\Omega, W_0(D^\omega_R))$ for any $R>0$. As  $\dst\lim_{R\to\infty}\uparrow \square_R=\R^{d-1}$, there exists $v\in H^1_{loc}(\R^d,L^2(\Omega))$ such that for all $R>0$, $v_n\xrightarrow[n\to\infty]{}v$ in $L^2(\Omega, W_0(D^\omega_R))$.

It remains to show that $\mathbb{1}_D v(\cdot,x_d)$ is stationary for all $x_d>0$. Let $n\in\mathbb{N}$. By stationarity of $\mathbb{1}_D v_n(\cdot,x_d)$, we have for a.e. $\omega\in \Omega$, $\bsx\in\R^{d-1}\times(0,+\infty)$ and $\bsy_\shortparallel\in\R^{d-1}$
\begin{equation*}\E\left[\left|\mathbb{1}_{D}v_n^\omega((\bsx_\shortparallel+\bsy_\shortparallel, x_d))-\mathbb{1}_{D}v_n^{\tau_{\bsy_\shortparallel}\omega}(\bsx)\right|\right]=0.\end{equation*}
We would like to pass to the limit in this equality but we only have that $v_n\xrightarrow[]{}v$ in $L^2(\Omega, W_0(D^\omega_R))$.
Let $R$ such that $\bsx_\shortparallel, \bsx_\shortparallel+\bsy_\shortparallel\in \square_{R}$. Using \eqref{eq:Eh-1} and the stationarity of $\mu$ we obtain by Cauchy-Schwarz inequality
\begin{multline}\label{eq:stati}
  \E\left[\int_{\square_{R}}\int_0^{x_d}\left|\mathbb{1}_{D}v^\omega((\tilde{\bsx}_\shortparallel+\bsy_\shortparallel, \tilde{x}_d))-\mathbb{1}_{D}v^{\tau_{\bsy_\shortparallel}\omega}(\tilde{\bsx})\right|d\tilde{\bsx}\right]\leq\E\left[\int_{\square_{R}}\int_0^{x_d}\mu^{-1}(\tilde{\bsx})d\tilde{\bsx}\right]\\
\left(\E\left[\int_{\square_{R}}\int_0^{x_d}\mu^\omega((\tilde{\bsx}_\shortparallel+\bsy_\shortparallel, \tilde{x}_d))\left|\mathbb{1}_{D}v_n^\omega((\tilde{\bsx}_\shortparallel+\bsy_\shortparallel, \tilde{x}_d))-\mathbb{1}_{D}v^\omega((\tilde{\bsx}_\shortparallel+\bsy_\shortparallel, \tilde{x}_d))\right|^2d\tilde{\bsx}\right]\right. \\
\left.+\E\left[\int_{\square_{R}}\int_0^{x_d}\mu^\omega((\tilde{\bsx}_\shortparallel+\bsy_\shortparallel, \tilde{x}_d))\left|\mathbb{1}_{D} v_n^{\tau_{\bsy_\shortparallel}\omega}(\tilde{\bsx})-\mathbb{1}_{D} v^{\tau_{\bsy_\shortparallel}\omega}(\tilde{\bsx})\right|^2d\tilde{\bsx}\right]\right).
\end{multline}
Since $v_n\xrightarrow[]{} v$ in $\vsu L^2(\Omega, W_0(D^\omega_{R}))$, we deduce by passing to the limit in this last inequality that $\mathbb{1}_D v(\cdot,x_d)$ is stationary for all $x_d>0$.
\end{proof}

In order to show well-posedness of \eqref{eq:NFtGen} we now derive a trace inequality for functions in $\mathcal{W}_0(D)$.
\begin{lemma}[Weighted trace Theorem]\label{lem:weightedtrace}
Let $\Pa$ verify Hypothesis \nameref{hyp:mixing}. Let $\alpha\in\{0\}\cup[h,+\infty)$. There exists $C_\alpha$ such that for any $v\in \mathcal{W}_0(D)$
\begin{equation*}
\mathbb{E}\left[\int_{\square_1}\mu^\omega(\bsxpar,\alpha)\left|v^\omega\big|_{{\Sigma_\alpha}}\right|^2(\bsy_\shortparallel)\dd\bsy_\shortparallel\right] \ \leq C_\alpha \ \|v\|_{\mathcal{W}_0(D)}^2.
\end{equation*}
\end{lemma}

\begin{proof} We first prove the result for $\alpha\in\{0,h\}$. Let $v\in\mathcal{W}_0(D)$. Let us consider the Vorono{\"\i} diagram $(V_n^\omega)_{n\in \mathbb{N}}$ in $\mathbb{R}^{d-1}\times[0,h]$ associated with the point process $\{\bsx_n^\omega\}_{n\in \mathbb{N}}$ defined by the sets 
\begin{equation*}
  V_n^\omega \coloneqq  \left\{ \bsy\in \mathbb{R}^{d-1}\times[0,h] \text{ such that } \big|\bsy-\bsx_n^\omega\big| = \min_{j\in \mathbb{N}} \Big|\bsy-\bsx_j^\omega\Big| \ \right\} , \ \ n\in \mathbb{N} . 
\end{equation*} 
For $\omega\in\Omega, n\in\mathbb{N}$ let $\Sigma_n^\omega(\alpha)$ denote the intersection of the hyperplane $\{y_d=\alpha\}$ with the Vorono{\"\i} cell $V_n^\omega$ . Let $(\rho, \bstheta)$ be the spherical coordinate system centered at $\bsx_n^\omega$. Since $V_n^\omega$ is convex, $\Sigma_n^\omega(\alpha)$ can be parametrized as $\Sigma_n^\omega(\alpha)\coloneqq \left\{(r(\bstheta), \bstheta), r(\bstheta)\in [1,+\infty), \bstheta\in [\bstheta_0, \bstheta_1]\subset\mathbb{S}^{d-1}\right\}$. Since $v^\omega\in H^1(V_n^\omega)$ with $v^\omega=0$ on $\partial B(\bsx^\omega_n)$, there exists $(v_p^\omega)_{p\in\N}$ in $\mathcal{C}^\infty(V_n^\omega)^\N$ with $v_p^\omega=0$ on $\partial B(\bsx^\omega_n)$  such that $v_p^\omega\to v^\omega$ in $H^1(V_n^\omega)$. We have for all $p$, \begin{equation*}\forall \bstheta\in [\bstheta_0, \bstheta_1],\quad \mathbb{1}_{D^\omega}v_p^\omega(r(\bstheta),\bstheta)=\int_1^{r(\bstheta)} \mathbb{1}_{D^\omega}\partial_r v^\omega_p(s,\bstheta) ds.\end{equation*}
We deduce for any non-negative function $\tilde{\mu}$ on $(1,+\infty)$, for all $\bstheta\in [\bstheta_0, \bstheta_1]$
\begin{equation*}\mathbb{1}_{D^\omega}|v_p^\omega(r(\bstheta),\bstheta)|^2\tilde{\mu}(r(\bstheta))r(\bstheta)^{d-1}\leq\int_1^{r(\bstheta)}\mathbb{1}_{D^\omega}|\partial_r  v_p^\omega(s, \bstheta)|^2s^{d-1}ds\left(\int_1^{r(\bstheta)} s^{1-d}ds\right)\tilde{\mu}(r(\bstheta))r(\bstheta)^{d-1}.\end{equation*}
For $\tilde{\mu}(\rho)\coloneqq \dst\left(\int_1^{\rho} s^{1-d}ds\right)^{-1}\rho^{1-d}$, we get after integrating with respect to $\bstheta$
\begin{equation*}\int_{\Sigma_n^\omega(\alpha)}\mathbb{1}_{D^\omega}\tilde{\mu}^\omega(\bsypar,\alpha)|v_p^\omega(\bsypar,\alpha)|^2\,\dd\bsypar\leq\int_{\bstheta_0}^{\bstheta_1}\int_1^{r(\bstheta)}\mathbb{1}_{D^\omega}|\partial_r  v_p^\omega(s, \bstheta)|^2\, s^{d-1}\dd s \dd\bstheta.\end{equation*}
We can now pass to the limit when $p\to+\infty$ to obtain the first inequality of 
\begin{equation}\label{eq:wvn}
  \begin{aligned}
  \int_{\Sigma_n^\omega(\alpha)}\mathbb{1}_{D^\omega}\tilde{\mu}^\omega(\bsypar,\alpha)|v^\omega\big|_{\Sigma_\alpha}|^2(\bsypar)\,\dd\bsypar&\leq\int_{\bstheta_0}^{\bstheta_1}\int_1^{r(\bstheta)}\mathbb{1}_{D^\omega}|\partial_r  v^\omega(s, \bstheta)|^2\, s^{d-1}\dd s \dd\bstheta\\&\lesssim\int_{\Sigma_n(\alpha)\times(0,h)}\mathbb{1}_{D^\omega}|\nabla v^\omega|^2(\bsy)\,\dd\bsy.
  \end{aligned}\end{equation}
%is the prism with base $\Sigma_n(\alpha)$ enclosing the pyramid $\{(\rho,\bstheta), \rho\in[0,r(\bstheta)], \theta\in[\theta_0,\theta_1]\}$.
By definition \eqref{eq:h}, for $\bsy\in V_n^\omega$, $\mu^\omega(\bsy)=r(\bstheta)^{-m}$ with $m>2d$. Therefore $\dst\bsy\mapsto{\mu^\omega}/{\tilde{\mu}^\omega}(\bsy)$ is bounded by monotonicity and there exists a constant independent of $\bsy\in V_n^\omega$ such that $\mu^\omega(\bsy)\lesssim\tilde{\mu}^\omega(\bsy)$. Consequently \eqref{eq:wvn} holds replacing $\tilde{\mu}$ by $\mu.$ For any $L>0$, we sum over all $\Sigma^\omega_n(\alpha)$ such that $\bsx_n^\omega\in\square_L\times(0,h)$ to obtain 
\begin{equation*}
\int_{\bigcup_{n\in N_L}\Sigma_n^\omega(\alpha)} \mathbb{1}_{D^\omega} \ {\mu}^\omega(\bsypar,\alpha)|v^\omega\big|_{\Sigma_\alpha}|^2(\bsypar)\,\dd\bsypar \lesssim \int_{\bigcup_{n\in N_L}\Sigma_n^\omega(\alpha)\times(0,h)} \mathbb{1}_{D^\omega} \ | \nabla v^\omega|^2(\bsy)\,\dd\bsy, 
\end{equation*} 
where $N_L\coloneqq \Big\{n \in \mathbb{N} \text{ such that  } \bsy_n^\omega\in\square_L\times(0,h)\Big\}$. Finally we apply Birkhoff Theorem (Theorem \ref{theorem:Birkhoff}) to conclude. \vsu

Let us consider now $\alpha>h.$ We write for $v\in\mathcal{C}^\infty(\overline{\square_1\times(h,\alpha)})$ \begin{equation*}\forall\bsy_\shortparallel\in\square_1,\quad v(\bsypar, \alpha)=v(\bsypar, h)+\int^{\alpha}_{h}\partial_{y_d}v(\bsy)\dd y_d.\end{equation*}By Cauchy-Schwarz inequality we get\begin{equation*}\forall\bsy_\shortparallel\in\square_1,\quad |v(\bsypar, \alpha)|^2\leq 2 \left(|v(\bsypar, h)|^2+(\alpha-h)\int^{\alpha}_{h}|\partial_{y_d}v(\bsy)|^2\dd y_d\right).\end{equation*}
We integrate over $\square_1$ to obtain
\begin{multline*}\int_{\square_1}\mu^\omega(\bsy_\shortparallel,\alpha)\left|v\big|_{\Sigma_\alpha}\right|^2(\bsy_\shortparallel)d\bsy_\shortparallel\leq 2\left((\alpha-h)\int_{h}^\alpha\int_{\square_1}\mu^\omega(\bsy_\shortparallel,\alpha)|\partial_{y_d} v^\omega|^2(\bsy)\dd\bsy\right.\\\left.+\int_{\square_1}\mu^\omega(\bsy_\shortparallel,\alpha)|\left|v\big|_{\Sigma_h}\right|^2(\bsy_\shortparallel)\dd\bsy_\shortparallel\right).
\end{multline*} By density, this equality also holds for $v\in {\cal W}_0(D)$.  On one hand, by definition \eqref{eq:h} and since $\alpha>h$, we have $\mu^\omega(\bsypar,\alpha)\leq 1$ for all $\bsypar\in\square_1$ so we deduce for the first term of the r.h.s.
\begin{equation}\label{eq:1t}\E\left[\int_{h}^\alpha\int_{\square_1}\mu^\omega(\bsy_\shortparallel,\alpha)|\partial_{y_d} v|^2(\bsy)\dd\bsy\right]\leq\E\left[\int_h^\alpha\int_{\square_1} \mathbb{1}_{D^\omega}|\nabla v^\omega|^2(\bsy)\dd\bsy\right].\end{equation}
On the other hand using the fact that $\mu^\omega(\bsypar,\alpha)\leq \mu^\omega(\bsypar,h)$ for all $\bsypar\in\square_1$ and the first part of the proof, we get for the second term of the r.h.s.
\begin{equation}\label{eq:2t}\E\left[\int_{\square_1}\mu^\omega(\bsy_\shortparallel,\alpha)|\left|v^\omega\big|_{\Sigma_h}\right|^2(\bsy_\shortparallel)\dd\bsy_\shortparallel\right]\lesssim\E\left[\int_0^h\int_{\square_1} \mathbb{1}_{D^\omega}|\nabla v^\omega|^2(\bsy)\dd\bsy\right].\end{equation}
Combining \eqref{eq:1t} and \eqref{eq:2t} we obtain the desired result.

\end{proof}
Let us now derive a random Hardy-type inequality in $\mathcal{W}_0(D)$.
\begin{proposition}[Random Hardy's inequality]\label{prop:wHardy} Suppose that $\Pa$ verifies hypothesis \nameref{hyp:mixing}. Then
\begin{equation}\label{eq:wHardy}\forall v\in\mathcal{W}_0(D),\quad \mathbb{E}\left[ \int_{\square_1}\int_0^{+\infty} \mathbb{1}_{D} \ \mu(\bsy)|v|^2(\bsy)\,\dd\bsy\right]\lesssim\mathbb{E}\left[ \int_{\square_1}\int_0^{+\infty} \mathbb{1}_{D} \ |\nabla v|^2(\bsy)\,\dd\bsy\right].\end{equation}
\end{proposition}
\begin{proof}   
We first prove that for all $v\in\mathcal{W}_0(D)$
\begin{equation}\label{eq:wHardyL}\mathbb{E}\left[ \int_{\square_1}\int_0^{h} \mathbb{1}_{D} \ \mu(\bsy)|v|^2(\bsy)\dd\bsy\right]\lesssim\mathbb{E}\left[ \int_{\square_1}\int_0^{h} \mathbb{1}_{D} \ |\nabla v|^2(\bsy)\dd\bsy\right].\end{equation}
Let us consider the Vorono{\"\i} diagram $(V_n^\omega)_{n\in \mathbb{N}}$ in $\mathbb{R}^{d-1}\times[0,h]$ associated with the point process $\{\bsx_n^\omega\}_{n\in \mathbb{N}}$.
Let $n\in\mathbb{N}$ and $\omega\in\Omega$. We first prove that the following inequality holds a.s. in $V_n^\omega$
\begin{equation}\label{eq:wHardyV}\int_{V_n^\omega}\mathbb{1}_{D^\omega}\mu^\omega(\bsy)|v^\omega|^2(\bsy)\,\dd\bsy\lesssim\int_{V_n^\omega}\mathbb{1}_{D^\omega}|\nabla v^\omega|^2(\bsy)\,\dd\bsy.\end{equation}The proof of \eqref{eq:wHardyV} relies on similar arguments as in the proof of Lemma 1.3 of \cite{beliaev1996darcy}. We reproduce it here for completeness. 
Let $(\rho, \bstheta)$ be the spherical coordinate system centered at $\bsx_n^\omega$. Since $V_n^\omega$ is convex, it can be parametrized as $V_n^\omega\coloneqq \left\{(\rho, \bstheta), \rho\in(1,R(\bstheta)], \bstheta\in \mathbb{S}^{d-1}\right\}$. Let $v^\omega\in\mathcal{C}^\infty(\overline{V_n^\omega})$ such that  $v^\omega=0$ on $\partial B(\bsx^\omega_n)$. We have for all $(\rho, \bstheta)\in V_n^\omega$
\begin{equation*}v^\omega(\rho,\bstheta)=\int_1^\rho \partial_r v^\omega(s,\bstheta) \dd s.\end{equation*}
Therefore we derive by Cauchy-Schwarz inequality
\begin{equation}\label{eq:polarVorono{\"\i}}|v^\omega(\rho,\bstheta)|^2\leq\int_1^\rho |\partial_r  v^\omega(s, \bstheta)|^2s^{d-1}\dd s\int_1^\rho s^{1-d}\dd s,\end{equation}
and for $\tilde{\mu}$ a non-negative function on $(1,+\infty)$
\begin{equation*}\int_{1}^{R(\bstheta)}\tilde{\mu}(\rho)|v^\omega(\rho,\bstheta)|^2\rho^{d-1}\dd\rho\leq\int_1^{R(\bstheta)}\int_1^\rho |\partial_r  v^\omega(s, \bstheta)|^2s^{d-1}\dd s\left(\int_1^\rho s^{1-d}\dd s\right)\tilde{\mu}(\rho)\rho^{d-1}\dd\rho.\end{equation*}
By Fubini-Tonelli's Theorem we get
\begin{equation*}\int_{1}^{R(\bstheta)}\tilde{\mu}(\rho)|v^\omega(\rho,\bstheta)|^2\rho^{d-1}\dd\rho\leq\int_1^{R(\bstheta)} |\partial_r  v^\omega(s, \bstheta)|^2s^{d-1}\left(\int_s^{R(\bstheta)}\left(\int_1^\rho s^{1-d}\dd s\right)\tilde{\mu}(\rho)\rho^{d-1}\dd\rho \right)\dd s.\end{equation*}
For $\tilde{\mu}(\rho)\coloneqq \rho^{-m}$ with $m>d$, $\dst\int_1^{+\infty} \left(\int_1^\rho s^{1-d}ds\right)\tilde{\mu}(\rho)\rho^{d-1}d\rho<+\infty$. We integrate with respect to $\bstheta$ to get
\begin{equation*}
  \int_{V_n^\omega}\mathbb{1}_{D^\omega}\tilde{\mu}^\omega(\bsy)|v^\omega|^2(\bsy)\,\dd\bsy\lesssim\int_{V_n^\omega}\mathbb{1}_{D^\omega}|\nabla v^\omega|^2(\bsy)\,\dd\bsy.\end{equation*}
To obtain \eqref{eq:wHardyV} for $v\in{\cal W}_0(D)$, we use a density argument and that for $\bsy\in V_n^\omega$, $\mu^\omega(\bsy)\leq\tilde{\mu}^\omega(\bsy)$. We conclude using Birkhoff Theorem (Theorem \ref{theorem:Birkhoff}).

%For any $L>0$, we sum over all cells $V^\omega_n$ such that $\bsx_n^\omega\in\square_L\times(\delta,h)$ to obtain 
%\begin{equation*}
% \int_{\bigcup_{n\in N_L}V_n^\omega} \mathbb{1}_{D^\omega} \ \mu^\omega|v^\omega|^2 \lesssim  \int_{\bigcup_{n\in N_L}V_n^\omega} \mathbb{1}_{D^\omega} \ | \nabla v^\omega|^2, 
%\end{equation*} 
%where $N_L\coloneqq \Big\{n \in \mathbb{N} \text{ such that  } \bsx_n^\omega\in\square_L\times(\delta,h)\Big\}$. Finally we apply Birkhoff theorem (Theorem \ref{theorem:Birkhoff}) by averaging spatially with respect to $\bsy_\shortparallel$ and taking the limit $L\uparrow +\infty$.\eqref{eq:wHardyL} follows.

Let us now prove \eqref{eq:wHardy}. By integrating by parts, we obtain
\begin{equation*}\begin{array}{ll}\vsd\dst\int_{h}^{+\infty}\int_{\square_1}\mu^\omega(\bsy)|v^\omega|^2(\bsy)\,\dd\bsy&\dst\leq \frac{1}{2} \int_{h}^{+\infty}\int_{\square_1}\frac{|v^\omega|^2(\bsy)}{(y_d+(R^\omega(\bsypar,h))^{m})^2}\,\dd\bsy,
  \\&\leq\dst \int_{h}^{+\infty}\int_{\square_1}\frac{\overline{v^\omega}\partial_{y_d}v^\omega(\bsy)}{y_d+(R^\omega(\bsypar,h))^{m}}\dd\bsy+\int_{\square_1}\frac{\left|v^\omega\big|_{\Sigma_h}\right|^2(\bsy_\shortparallel)}{h+(R^\omega(\bsypar,h))^{m}}\dd\bsy_\shortparallel.\end{array}
\end{equation*}We deduce the following by applying Cauchy-Schwarz inequality and using the weighted trace theorem (Lemma \ref{lem:weightedtrace})
\begin{multline*}\E\left[\int_{h}^{+\infty}\int_{\square_1}\mu^\omega(\bsy)|v^\omega|^2(\bsy)\,\dd\bsy\right]\leq\mathbb{E}\left[\int_{h}^{+\infty}\int_{\square_1}\frac{|v^\omega(\bsy)|^2}{(y_d+(R^\omega)^m)^2}\dd\bsy\right]^{\frac{1}{2}}\mathbb{E}\left[\int_{h}^{+\infty}\int_{\square_1}|\partial_{y_d}v^\omega|^2(\bsy)\,\dd\bsy\right]^{\frac{1}{2}}\\+\mathbb{E}\left[ \int_{\square_1}\int_0^{+\infty} \mathbb{1}_{D} \ |\nabla v|^2(\bsy)\,\dd\bsy\right].
\end{multline*}
\eqref{eq:wHardy} is then a consequence of Young's inequality.
\end{proof}

\begin{remark}\label{rem:wHardyW0D} Note that using similar arguments we can prove that 
\begin{equation}\label{wHardyCC}\int_{D^\omega}\mu^\omega(\bsy)|v^\omega|^2(\bsy)\,\dd\bsy\lesssim\int_{D^\omega}|\nabla v^\omega|^2(\bsy)\,\dd\bsy,\end{equation}where
for $v^\omega\in W_0(D^\omega)$ for $\omega\in\Omega$, where \begin{equation*}
W_0(D^\omega)\coloneqq \left\{v\in H_{loc}^1\big(D^\omega\big), v=0\,\text{on }\partial \Pa^\omega, \int_{\R^{d-1}}\int_0^{+\infty} \mathbb{1}_{D}\left( \ \mu(\bsy)|v|^2(\bsy)+|\nabla v|^2(\bsy)\right)\, \dd\bsy<+\infty\right\}.
\end{equation*}\end{remark}

\subsection{Proof of proposition \ref{prop:Utnwellposed}}

 The continuity of the sesquilinear form $a$ given by \eqref{eq:bil_NF} is straightforward. The coercivity of $a$ is guaranteed by the weighted Hardy's inequality \eqref{eq:wHardy}. The continuity of the anti-linear form $l$ given by \eqref{eq:lin_NF} can be deduced from Cauchy-Schwarz inequality and the weighted trace lemma (Lemma \ref{lem:weightedtrace}). Indeed for all $V\in\mathcal{W}_0(D)$

\begin{equation*}\begin{array}{ll}
\dst\vsd\E\left[\int_{\square_1}\int_0^L \mathbb{1}_{D} (F^\omega\overline{V}^\omega+\bsG^\omega\cdot\nabla\overline{V}^\omega)(\bsy)\,\dd\bsy\right]&\dst\leq \|\mu^{-\frac{1}{2}}F\|_{\mathcal{L}^2(D)}\|\mu^{\frac{1}{2}}V\|_{\mathcal{L}^2(D)}+\|G\|_{\mathcal{L}^2(D)}\|\nabla V\|_{\mathcal{L}^2(D)},\\&\dst\leq\left(\|\mu^{-\frac{1}{2}}F\|_{\mathcal{L}^2(D)}+\|G\|_{\mathcal{L}^2(D)}\right)\|V\|_{\mathcal{W}_0(D)},\end{array}
\end{equation*}
and
\begin{equation}\label{eq:estl}
\E\left[ \int_{\square_1}(\Psi^\omega \ \overline{V}^\omega\big|_{\Sigma_0} + \alpha^{\omega,N} \ \overline{V}^\omega\big(\cdot, \ H\big)) \right]\leq (\|\mu^{-\frac{1}{2}}\Psi\|_{\mathcal{L}^2(\Sigma_0)}+\|\mu^{-\frac{1}{2}}\alpha^{N}\|_{\mathcal{L}^2(\Sigma_H)}\|V\|_{\mathcal{W}_0(D)}.
\end{equation}By Lax-Milgram Theorem, problem \eqref{eq:NFtGen} is well posed in $\mathcal{W}_0(D)$. The estimate \eqref{eq:estUt} is a direct consequence of the coercivity of $a$ and estimate \eqref{eq:estl}.

Finally, using similar arguments than in , onw shows that the unique solution satisfies almost surely Problem \eqref{eq:NFtGen} in the sense of distributions.
\subsection{Application to the first near-field terms}

Let us now apply this uniqueness and existence result to our first near-field terms, and more specifically to  $\widetilde{U}^{\omega, NF}_n\coloneqq U^{\omega,NF}_n+u_{n}^{\omega, FF}\big|_{\Sigma_{\eps H}}\mathbb{1}_{\mathcal{B}^\infty_H}$ which satisfy \eqref{eq:NFtGen} with \begin{equation*}F\coloneqq \Delta_{\bsx_\shortparallel} U^{\omega,NF}_{n-2}\ + \ k^2 \ U^{\omega,NF}_{n-2},\quad G\coloneqq 2 \nabla_{\bsx_\shortparallel} U^{\omega,NF}_{n-1},\quad \Psi\coloneqq ik\gamma \ U^{\omega,NF}_{n-1},\quad\textrm{and}\quad\alpha^N\coloneqq  \partial_{x_d}u^{\omega, FF}_{n-1}\big|_{\Sigma_{\eps H}}.\end{equation*}
By Proposition \ref{prop:Utnwellposed}, for almost every $\bsx_\shortparallel\in \mathbb{R}^{d-1}$, there exists a unique $\widetilde{U}_{0}^{NF}(\bsx_\shortparallel; \ \cdot \ )\in \mathcal{W}_0(D)$ solution of  
\begin{equation}\label{eq:pb_champ_proche0_Dirichlet_alea}
\begin{array}{|l}
   -\Delta_{\bsy} \widetilde{U}_0^{\omega,NF}(\bsx_\shortparallel; \ \cdot \ )  =0 \quad \text{ in } \quad\mathcal{B}_{H}^{\omega} \cup \mathcal{B}_{ H}^{\infty}, \\
  -\partial_{y_d} \widetilde{U}^{\omega,NF}_0(\bsx_\shortparallel; \ \cdot \ )=0\; \text{ on } \;\Sigma_0, \quad\text{and}\quad
   \widetilde{U}^{\omega,NF}_0 (\bsx_\shortparallel; \ \cdot \ )= \ 0 \;\text{ on } \; \partial \mathcal{P}^\omega, \\
 \Big[\widetilde{U}^{\omega,NF}_0(\bsx_\shortparallel; \ \cdot \ )  \Big]_H = \ 0 \quad \text{ and }\quad   \Big[-\partial_{y_d}\widetilde{U}^{\omega,NF}_0(\bsx_\shortparallel; \ \cdot \ )  \Big]_H =0.
\end{array}\end{equation}
We deduce that $\widetilde{U}_{0}^{\omega,NF}=0$ and then $U_{0}^{\omega,NF}(\bsx_\shortparallel;\bsy)=-u_{0}^{\omega, FF}\big|_{\Sigma_{\eps H}}(\bsx_\shortparallel) \mathbb{1}_{\mathcal{B}^\infty_H}(\bsy)$. Since we look for a solution that tends to $0$ when $y_d$ tends to $+\infty$, we obtain the following boundary condition for $u_0^{FF}$ on $\Sigma_{\eps H}$
   \begin{equation}\label{eq:condition_limite_champ_loin0_Dirichlet_alea}
\text{a.s.  }u_{0}^{\omega, FF}\big|_{\Sigma_{\eps H}}(\bsx_\shortparallel) = 0 . 
\end{equation}
Note that $u_0^{FF}$ that satisfies \eqref{eq:cascade_champ_loin_alea} is then independent of $\omega$.

By proposition \ref{prop:Utnwellposed}, for almost any $\bsx_\shortparallel\in \mathbb{R}^{d-1}$, there exists a unique $\widetilde{U}_{1}^{NF}(\bsx_\shortparallel; \ \cdot \ )\in \mathcal{W}_0(D)$ solution  of  
\begin{equation}\label{eq:pb_champ_proche1_Dirichlet_alea}
\begin{array}{|l}
   -\Delta_{\bsy} \widetilde{U}_1^{\omega,NF} (\bsx_\shortparallel; \ \cdot \ ) =0 \quad \text{ in } \quad\mathcal{B}_{H}^{\omega} \cup \mathcal{B}_{ H}^{\infty}, \\
  -\partial_{y_d} \widetilde{U}^{\omega,NF}_1(\bsx_\shortparallel; \ \cdot \ )=0\; \text{ on } \; \Sigma_0, \quad\text{and}\quad
   \widetilde{U}^{\omega,NF}_1 (\bsx_\shortparallel; \ \cdot \ )= \ 0 \; \text{ on } \;\partial \mathcal{P}^\omega, \\
 \Big[\widetilde{U}^{\omega,NF}_1(\bsx_\shortparallel; \ \cdot \ )  \Big]_H = \ 0 \quad \text{ and }\quad   \Big[-\partial_{y_d}\widetilde{U}^{\omega,NF}_1(\bsx_\shortparallel; \ \cdot \ )  \Big]_H =\partial_{x_d} u_{0}^{\omega, FF}\big|_{\Sigma_{\eps H}}(\bsx_\shortparallel).
\end{array}
\end{equation}
where we used the fact that $U_{0}^{\omega,NF}= 0$.\vsd

By linearity, $\widetilde{U}_{1}^{\omega,NF}(\bsx_\shortparallel;\bsy)=\partial_{x_d} u_{0}^{\omega, FF}\big|_{\Sigma_{\eps H}}(\bsx_\shortparallel) W_1^\omega(\bsy),$ where $W_1^\omega$ is the unique solution in $\mathcal{W}_0(D)$ of  \begin{equation}\label{eq:w10}
\begin{array}{|l}
   -\Delta_{\bsy} W_1^\omega=0 \quad \text{ in } \quad\mathcal{B}_{H}^{\omega} \cup \mathcal{B}_{ H}^{\infty}, \\
  -\partial_{y_d} W_1^\omega=0\;\text{ on } \; \Sigma_0, \quad\text{and}\quad
 W_1^\omega = \ 0 \; \text{ on } \; \partial \mathcal{P}^\omega, \\
 \Big[W_1^\omega  \Big]_H = \ 0 \quad \text{ and }\quad   \Big[-\partial_{y_d}W_1^\omega  \Big]_H =1.
\end{array}
\end{equation}
If we want to impose that $U_{1}^{\omega,NF}$ that is given by 
\begin{equation}\label{eq:U_1_expr}
  U_{1}^{\omega,NF}(\bsx_\shortparallel;\bsy)= -u_{1}^{\omega, FF}\big|_{\Sigma_{\eps H}}(\bsx_\shortparallel) \mathbb{1}_{\mathcal{B}^\infty_H}(\bsy)+\partial_{x_d} u_{0}^{\omega, FF}\big|_{\Sigma_{\eps H}}(\bsx_\shortparallel) W_1^\omega(\bsy),
\end{equation}
 tends to $0$ when $y_d$ tends to $+\infty$, we need to understand the behavior at infinity of $W_1^\omega$, which is done in the following section.

%%%%%%---------------------------------
\section{Integral representation and behavior at infinity of the near-fields}\label{sec:limbehavior}
%%%%%%---------------------------------

In order to study the behavior of the near-fields when $y_d$ tends to $+\infty$, we first derive an integral representation for $\widehat{U}^\omega$ the unique solution of \eqref{eq:NFtGen} in $\mathcal{W}_0(D)$ for $F=0$, $G=0$,$\mu^{-\frac{1}{2}}\Psi\in\mathcal{L}^2(\Sigma_0)$ and $\mu^{-\frac{1}{2}}\alpha^N\in \mathcal{L}^2(\Sigma_H)$. Let us introduce the Green's function associated to the Laplace operator in $\R^d$ 
\begin{equation}\label{fonction_Laplace_Green}
   \Gamma(\bsz) \coloneqq  
  -\frac{1}{2\pi} \ \mathrm{ln} \ |\bsz| \quad \text{ for } \quad d=2 \ , \quad
  \frac{1}{4\pi|\bsz|} \quad  \text{ for } \quad d=3 \  
  ,\quad \bsz\in \mathbb{R}^{d}.
  \end{equation}
\begin{proposition}[Integral representation for $\widehat{U}^\omega$]\label{prop:representation_integrale_U_alea} Suppose that $\Pa$ verifies Hypothesis \nameref{hyp:mixing}.
A.s.  for $\bsy\in \mathbb{R}^{d-1}\times(L, +\infty)$, $\widehat{U}^\omega(\bsy)$ has the following integral representation 
\begin{equation}\label{eq:repint_W10}
\widehat{U}^\omega\big(\bsy\big) = -2\int_{\R^{d-1}}\partial_{z_d}\Gamma\big(\bsy, (\bszpar,L)\big) \ \varphi^\omega\big(\bsz_\shortparallel\big) \mathrm{d}\bsz_\shortparallel \end{equation}
where $\varphi$ denotes the trace of $\widehat{U}^\omega$ on $\Sigma_L$.
\end{proposition}
We postpone the proof of Proposition \ref{prop:representation_integrale_U_alea} to appendix \ref{sec:proofir}. With this integral representation we prove that a.s. and in $L^2(\Omega)$ $\widehat{U}$ tends to a deterministic constant as $y_d\to+\infty$. This constant is the ensemble average of its trace on an hyperplane above the layer of particles. In the periodic case one can establish a similar result with the constant being the spatial mean on the hyperplane over a period.
Before proving this result, let us start with a useful technical lemma.
\begin{lemma}
  Let $\Gamma$ be the Green's function defined in \eqref{fonction_Laplace_Green}. Let us introduce 
  \begin{equation}\label{eq:der_Gamma}
\partial_d\Gamma\coloneqq 
\begin{cases}
  \partial_{z_1}\partial_{z_2}\Gamma  & \text{ for }  d=2 \\
    \partial_{z_1}\partial_{z_2}\partial_{z_3}\Gamma& \text{ for } d=3,
\end{cases}
 \quad\text{and}\quad \pi(\bsz_\shortparallel)\coloneqq 
\begin{cases}
  z_1  & \text{ for }  d=2 \\
   z_1z_2& \text{ for } d=3.
\end{cases}
  \end{equation}
  We have for a fixed $R$ and for  $y_d$ large enough that 
  \begin{equation}\label{eq:ineq1}
  \int_{\square_R} \left|\partial_d\Gamma(\bsz_\shortparallel,y_d-L)\right| d\bsz_\shortparallel \lesssim \frac{1}{y_d^{3(d-1)}}\quad  \text{and}\quad
  \int_{\R^{d-1}\setminus\square_R}\left| \partial_d\Gamma(\bsz_\shortparallel,y_d-L)\left(y_d^{d-1}+\pi(\bsz_\shortparallel)\right)\right|d\bsz_\shortparallel  \lesssim 1 
  \end{equation}
  \begin{equation}\label{eq:ineq2}
 \int_{\square_R}  \left|\partial_{z_i}\partial_d\Gamma(\bsz_\shortparallel,y_d-L)\right| d\bsz_\shortparallel \lesssim \frac{1}{y_d^{3(d-1)}} \quad  \text{and}\quad
 \int_{\R^{d-1}\setminus\square_R}  \left|\partial_{z_i}\partial_d\Gamma(\bsz_\shortparallel,y_d-L)\,\pi(\bsz_\shortparallel)\right|d\bsz_\shortparallel  \lesssim \frac{1}{y_d}
  ,\;  i\in\llbracket 1,d\rrbracket,
  \end{equation}
  where the constants depend only on $R$ and $L$.
\end{lemma}
\begin{proof}
  A straighforward computation yields 
  \[  
  \partial_d\Gamma(\bsz)= \frac{1}{\pi}\frac{z_1z_2}{(z_1^2+z_2^2)^2}  \; \text{ for }  d=2 
  \quad\text{and}\quad  \frac{1}{4\pi}\frac{-15 z_1z_2z_3}{(z_1^2+z_2^2+z_3^2)^{7/2}}\; \text{ for } d=3,
  \]
  it is then easy to show the first inequality of \eqref{eq:ineq1}. Moreover, by a change of variable $\bsz_\shortparallel\mapsto \bsz_\shortparallel/y_d$, we obtain
  the second inequality of \eqref{eq:ineq1}
  \begin{equation*}
  \int_{\R^{d-1}\setminus\square_R} \left|\partial_d\Gamma(\bsz_\shortparallel,y_d-L)\left(y_d^{d-1}+\pi(\bsz_\shortparallel)\right)\right|\,d\bsz_\shortparallel \lesssim 
  \left|\begin{array}{l}\vsd\dst\int_{\R}\frac{|u|(1+|u|)\,du}{(u^2+1)^2} \; \text{ for }  d=2 \qquad\qquad
  \\
  \dst \int_{\R^2}\frac{|u_1u_2|(1+|u_1u_2|)\,du_1du_2}{(u_1^2+u_2^2+1)^{7/2}}\; \text{ for }  d=3.\end{array}\right.
  \end{equation*}
  The proof of \eqref{eq:ineq2} can be done using similar arguments.
\end{proof}
\begin{proposition}\label{prop:limNF} Let $\Pa$ verify Hypothesis \nameref{hyp:mixing}.
$\widehat{U}^\omega$ the unique solution of \eqref{eq:NFtGen}  in $\mathcal{W}_0(D)$ with $F=0$, $G=0$,$\mu^{-\frac{1}{2}}\Psi\in\mathcal{L}^2(\Sigma_0)$ and $\mu^{-\frac{1}{2}}\alpha^N\in \mathcal{L}^2(\Sigma_H)$ verifies 

\begin{equation}\label{eq:lim_infty}
 \text{a.s.}\quad \lim_{y_d\to +\infty} \widehat{U}^\omega\big(\bsy_\shortparallel, y_d\big) = \mathbb{E}\big[ \varphi\big];\quad \bsy_\shortparallel \in \mathbb{R}^{d-1} ,
 \end{equation}
this limit being locally uniform in $\bsy_\shortparallel$ and

 \begin{equation}\label{eq:lim_infty2}
\mathbb{E}\left[\Big| \widehat{U}^\omega\big( \cdot,y_d\big) - \mathbb{E}\big[ \varphi\big]  \Big|^2 \right] \underset{y_d\to+\infty}{\longrightarrow} 0  \ \quad\text{and}\quad
y_d^{2}  \  \mathbb{E}\left[ \Big|\nabla \widehat{U}^\omega\big(\cdot,y_d\big) \Big|^2 \right] \underset{y_d\to+\infty}{\longrightarrow} 0 \ .  
 \end{equation}
\end{proposition}
Recall that by Lemma \ref{lem:weightedtrace}, $\mu^{\frac{1}{2}}_{|_{\Sigma_L}} \varphi\in\mathcal{L}^2(\Sigma_L)$, and thus since $\mu^{-1}_{|_{\Sigma_L}}\in L^1(\Omega)$ by Hypothesis \nameref{hyp:mixing}, $\varphi\in L^1(\Omega)$ by Cauchy-Schwarz inequality.
\begin{proof}
The proof is based as in section 4.6 of \cite{Basson2006} on the integral representation established in Proposition \ref{prop:representation_integrale_U_alea} and Birkhoff ergodic Theorem.
\\\\
Let us first rewrite $\widehat{U}^\omega\big( \cdot,y_d\big) - \mathbb{E}[ \varphi]$ using the integral representation \eqref{eq:repint_W10}.  We then have \begin{equation}
\begin{aligned}\label{eq:U-EphiI}
\widehat{U}^\omega\big( \cdot,y_d\big) -\mathbb{E}[ \varphi]\coloneqq & -2 \int_{\mathbb{R}^{d-1}} \partial_{z_d} \Gamma\big(\bsz_\shortparallel,y_d-L\big) \big(\varphi^\omega(\bsy_\shortparallel-\bsz_\shortparallel)-\mathbb{E}[\varphi]\big)\mathrm{d}\bsz_\shortparallel, \\  
=& -2 \int_{\mathbb{R}^{d-1}} \partial_{z_d} \Gamma\big( \bsz_\shortparallel,y_d-L \big)  \partial_{\bsz_\shortparallel}\int_{0}^{\bsz_\shortparallel} \big(\varphi^\omega(\bsy_\shortparallel-\mathbf{w})-\mathbb{E}\big[\varphi\big]\big)\mathrm{d}\mathbf{w} \mathrm{d}\bsz_\shortparallel,\\ 
=& (-1)^d 2\int_{\mathbb{R}^{d-1}} \partial_{d}\Gamma \big(\bsz_\shortparallel,y_d-L\big) \pi(\bsz_\shortparallel) \left[\fint_{0}^{\bsz_\shortparallel} \varphi^\omega(\bsy_\shortparallel-\mathbf{w})\mathrm{d}\mathbf{w}-\mathbb{E}[\varphi]  \right] \mathrm{d}\bsz_\shortparallel.
\end{aligned}
\end{equation}
where we used the notations introduced in \eqref{eq:der_Gamma} and also that $\partial_{\bsz_\shortparallel}\coloneqq \partial_{z_1}$ for $d=2$ and $\coloneqq\partial_{z_1}\partial_{z_2}$ for $d=3$ and $\dst\int_{0}^{\bsz_\shortparallel} \coloneqq \int_{0}^{z_1}$ for $d=2$ and $\dst\coloneqq\int_{0}^{z_1} \int_{0}^{z_2}$ for $d=3$.

We first prove \eqref{eq:lim_infty}. Let $\eps>0$. By Birkhoff ergodic Theorem (Theorem \ref{theorem:Birkhoff}), we know that there exists $R>0$ such that for all $R'>R$
\begin{equation*}
 \text{a.s.}\quad \left| \fint_{\Box_{R'}} \varphi^\omega\big(\bsy_\shortparallel-\mathbf{w}\big)\, \mathrm{d}\mathbf{w}\ - \ \mathbb{E}[ \varphi] \right| \ \leq \ \eps,\quad \bsy_\shortparallel\in \Box_{R'} \ . 
\end{equation*}
On one hand we obtain then using the second inequality of \eqref{eq:ineq1}{
\[
  \left|\int_{\R^{d-1}\setminus\Box_R}\hspace{-0.3cm}\partial_{d}\Gamma \big(\bsz_\shortparallel,y_d-L \big)\pi(\bsz_\shortparallel) \left[ \fint_{0}^{\bsz_\shortparallel}\varphi^\omega\big(\bsy_\shortparallel-\mathbf{w}\big) \mathrm{d}\mathbf{w}-\mathbb{E}[\varphi]\right] \mathrm{d}\bsz_\shortparallel\right| 
   \leq \eps. %\ \int_{\R^{d-1}\setminus\Box_R}\hspace{-0.3cm} \left|\partial_{d}\Gamma \big(\bsz_\shortparallel,y_3-L \big)\pi(\bsz_\shortparallel) \right| \mathrm{d}\bsz_\shortparallel \ . 
\]
\\
On the other hand we have
\begin{multline*}
    \left|\int_{\Box_R} \partial_d\Gamma \big(\bsz_\shortparallel,y_d-L \big)\pi(\bsz_\shortparallel)  \left[ \fint_{0}^{\bsz_\shortparallel} \varphi^\omega\Big(\bsy_\shortparallel-\mathbf{w}\big)\, \mathrm{d}\mathbf{w}-\mathbb{E}[\varphi^\omega]\right] \,\mathrm{d}\bsz_\shortparallel\right| \\
       \lesssim  \ \left( \int_{\Box_R} \Big|\varphi^\omega\big(\bsy_\shortparallel-\mathbf{w}\big)\Big| \, \mathrm{d}\mathbf{w}+ R^2\mathbb{E}[ \varphi^\omega]\right) \int_{\Box_R} \left|\partial_d\Gamma \big(\bsz_\shortparallel,y_d-L\big) \right| \mathrm{d}\bsz_\shortparallel \ ,  
\end{multline*}
and it suffices to use the first inequality of \eqref{eq:ineq1} to obtain the wanted result.}
\\\\
Let us now prove the first limit in \eqref{eq:lim_infty2}.
Let $\eps>0$. By Birkhoff's Theorem, since $\int_{\square_1} \varphi\in L^2(\Omega)$ there exists $R>0$ such that for all $R'>R$
\begin{equation}\label{inegalite_pointb_preuve_prop_alea}
 \mathbb{E}\left[ \left|\fint_{\square_{R'}} \varphi(\bsz_\shortparallel) \mathrm{d}\bsz_\shortparallel - \mathbb{E}\big[ \varphi\big] \right|^2 \right] \ \leq \ \eps \ . 
\end{equation}
By applying Cauchy-Schwarz inequality and taking the expectation, we obtain for $\bsy_\shortparallel\in \mathbb{R}^{d-1}$, $y_d>L$, 
\begin{multline}\label{eq:I}
\mathbb{E}\left[\left|\widehat{U}^\omega\big( \cdot,y_d\big)-\mathbb{E}[ \varphi]\right|^2 \right]
\lesssim \int_{\R^{d-1}\setminus\square_R} \left| \pi(\bsz_\shortparallel) \partial_d\Gamma\big(\bsz_\shortparallel,y_d-L\big)\right|\mathrm{d}\bsz_\shortparallel \\ 
 \int_{\R^{d-1}\setminus\square_R} \left|\pi(\bsz_\shortparallel) \partial_d\Gamma\big(\bsz_\shortparallel,y_d-L\big)  \right| \ \mathbb{E}\left[\left| \fint_0^{\bsz_\shortparallel} \varphi\big(\bsy_\shortparallel-\mathbf{w}\big) \  \mathrm{d}\mathbf{w} -\mathbb{E}\big[\varphi \big] \right|^2\right]\mathrm{d}\bsz_\shortparallel\\
 +\int_{\square_R} \Big|\partial_d\Gamma\big(\bsz_\shortparallel ,y_d-L\big) \Big| \mathrm{d}\bsz_\shortparallel \int_{\square_R} \left| \partial_d\Gamma\big(\bsz_\shortparallel,y_d-L\big)  \right| \ \mathbb{E}\left[\left| \int_0^{\bsz_\shortparallel} (\varphi\big(\bsy_\shortparallel-\mathbf{w}\big) \  -\mathbb{E}\big[\varphi \big])\mathrm{d}\mathbf{w}  \right|^2\right] \mathrm{d}\bsz_\shortparallel. 
\end{multline}
To bound the second term of the right-hand side of \eqref{eq:I}, note that since $\varphi\in L^2_{loc}\big(\mathbb{R}^{d-1},L^2(\Omega)\big)$, we have
\begin{multline}\label{eq:majoration2}
    \left|\int_{\square_R} \partial_d\Gamma \big(\bsz_\shortparallel,y_d-L \big) \E\left[\left| \int_{0}^{\bsz_\shortparallel} (\varphi\big(\bsy_\shortparallel-\mathbf{w}\big) \  -\mathbb{E}\big[\varphi \big])\mathrm{d}\mathbf{w}\right|^2\right] \mathrm{d}\bsz_\shortparallel\right| \\
       \lesssim  R^{d-1}\mathbb{E}\left[\int_{\square_R} \left|\varphi\big(\bsy_\shortparallel-\mathbf{w}\big) \  -\mathbb{E}\big[\varphi \big]\right|^2\mathrm{d}\mathbf{w}\right]\int_{\square_R} \left|\partial_d\Gamma \big(\bsz_\shortparallel,y_d-L\big) \right| \mathrm{d}\bsz_\shortparallel.  
\end{multline}
We then use the first inequality of \eqref{eq:ineq1} to conclude.
{Let us now bound the first term of the right-hand side of \eqref{eq:I}. From \eqref{inegalite_pointb_preuve_prop_alea} we obtain \begin{multline*}
  \left|\int_{\R^{d-1}\setminus\square_R}\partial_d\Gamma\big(\bsz_\shortparallel,y_d-L \big)\pi(\bsz_\shortparallel) \mathbb{E}\left[\left| \fint_0^{\bsz_\shortparallel} \varphi\big(\bsy_\shortparallel-\mathbf{w}\big) \  \mathrm{d}\mathbf{w} -\mathbb{E}\big[\varphi \big] \right|^2\right]\mathrm{d}\bsz_\shortparallel\right| \\ 
   \leq \eps \ \int_{\R^{d-1}\setminus\square_R} \left|\partial_d\Gamma\big(\bsz_\shortparallel,y_d-L \big)\pi(\bsz_\shortparallel) \right| \mathrm{d}\bsz_\shortparallel. 
\end{multline*}
By the second inequality of \eqref{eq:ineq1} the integral in the right-hand side is bounded. Finally we proved that for $y_d$ large enough
\begin{equation*}
 \mathbb{E}\left[ \left| \widehat{U}^\omega\big( \cdot,y_d\big)-\mathbb{E}[ \varphi]\right|^2\right] \ \lesssim \ \frac{1}{y_d^{6(d-1)}} + C\eps. 
\end{equation*}
To prove the second limit in \eqref{eq:lim_infty2}, we replace $\widehat{U}^\omega-\mathbb{E}[ \varphi]$ by $\nabla  \widehat{U}^\omega$ and $\partial_d\Gamma$ by $\nabla\partial_d\Gamma$ in \eqref{eq:I} and use the two inequalities of \eqref{eq:ineq2}.}
\end{proof}
In the periodic case, the convergence of $\widehat{U}$ to its limit as $y_d\to+\infty$ is exponential with a rate inversely proportional to the period. In the random stationary ergodic setting, the speed of convergence can be estimated provided that its trace $\varphi$ satisfies some additional mixing assumption.

\begin{proposition}\label{prop:U_comportement_inf_alea_quant}
  If 
  \begin{equation}\label{eq:ergod_quant}
 R^{d-1}\mathbb{E}\left[\left|\fint_{\square_R} \varphi\big(\bsz_\shortparallel\big)\mathrm{d}\bsz_\shortparallel - \mathbb{E}[\varphi]\right|^2\right]\underset{R\to+\infty}{=}\mathcal{O}(1),
  \end{equation}
  then
  \begin{equation}\label{eq:lim_infty3}
    y_d^{d-1}\mathbb{E}\left[\Big| \widehat{U}^\omega\big( \cdot,y_d\big) - \mathbb{E}\big[ \varphi\big]  \Big|^2 \right] +    y_d^{d+1} \  \mathbb{E}\left[ \Big|\nabla  \widehat{U}^\omega\big(\cdot,y_d\big) \Big|^2 \right] \underset{y_d\to+\infty}{=}\mathcal{O}(1).  
    \end{equation}
\end{proposition}
\begin{proof}
  Let us fix $\eps>0$. We suppose now that there exist $C\geq0$ and  $R>0$ such that for all $R'>R$
  \begin{equation*}
(R')^{d-1}\mathbb{E}\left[\left| \fint_{\square_{R'}} \varphi^\omega\big(\bsz_\shortparallel\big)\mathrm{d}\bsz_\shortparallel\ - \ \mathbb{E}[ \varphi] \right|\right] \ \leq C. 
  \end{equation*} 
{
Using similar arguments as in the proof of Proposition \ref{prop:limNF}, we have for $\bsy_\shortparallel\in \mathbb{R}^{d-1}$, $y_d>L$, 
\begin{multline*}
\E\left[\left| \widehat{U}^\omega\big( \cdot,y_d\big) - \mathbb{E}\big[ \varphi\big]\right|^2 \right]
\lesssim \int_{\R^{d-1}\setminus\square_R} \left|\partial_d\Gamma\big(\bsz_\shortparallel,y_d-L\big)\right|\mathrm{d}\bsz_\shortparallel \\ \int_{\R^{d-1}\setminus\square_R} \left|\pi(\bsz_\shortparallel) \partial_d\Gamma\big(\bsz_\shortparallel,y_d-L\big)  \right| \ \pi(\bsz_\shortparallel)\mathbb{E}\left[\left| \fint_0^{\bsz_\shortparallel} \varphi\big(\bsy_\shortparallel-\mathbf{w}\big) \  \mathrm{d}\mathbf{w} -\mathbb{E}\big[\varphi \big] \right|^2\right]\mathrm{d}\bsz_\shortparallel  
\\+ \int_{\square_R} \Big|\partial_d\Gamma\big(\bsz_\shortparallel ,y_d-L\big) \Big| \mathrm{d}\bsz_\shortparallel  \int_{\square_R} \left| \partial_d\Gamma\big(\bsz_\shortparallel,y_d-L\big)  \right| \ \mathbb{E}\left[\left| \int_0^{\bsz_\shortparallel} \varphi\big(\bsy_\shortparallel-\mathbf{w}\big) \  \mathrm{d}\mathbf{w} -\mathbb{E}\big[\varphi \big] \right|^2\right] \mathrm{d}\bsz_\shortparallel.
 \end{multline*}
We can bound the second term of the right-hand side as in the proof of Proposition \ref{prop:limNF} (see \eqref{eq:majoration2}). To bound the first term, we use the second inequality of \eqref{eq:ineq1}.}

{To prove the second part of \eqref{eq:lim_infty3}, we decompose similarly  $\E\left[\left|\nabla  \widehat{U}^\omega\right|^2\right]$ and use the two inequalities of \eqref{eq:ineq2}.}
\end{proof}

\section{Effective model}\label{sec:eff_model}
%%%%%%%%%%%%%%%%%%%%%%%%%%%%%%%
\subsection{Construction of $U^{NF}_1$}
%%%%%%%%%%%%%%%%%%%%%%%%%%%%%%%

Recall that $U_{1}^{NF}$ is given by \eqref{eq:U_1_expr} where $W_1$ is the unique solution in $\mathcal{W}_0(D)$ of \eqref{eq:w10}. In particular $W_1$ satisfies proposition \ref{prop:limNF}. As we look for a solution $U_{1}^{\omega,NF}$ which tends to $0$ when $y_d$ tends to $+\infty$, we obtain the following boundary condition for $u_1^{FF}$ on $\Sigma_{\eps H}$
\begin{equation}\label{eq:condition_limite_champ_loin1_Dirichlet_alea}
\text{a.s.  } \  -u_{1}^{\omega, FF}\big|_{\Sigma_{\eps H}}(\bsx_\shortparallel) + \partial_{x_d} u_0^{FF}\big|_{\Sigma_{\eps H}}(\bsx_\shortparallel) \ c_1= 0,
\end{equation}
where $c_1$ is given by \begin{equation}\label{eq:constante_c01_Dirichlet_alea}
  c_1\coloneqq \lim_{y_d\to +\infty}W_1^\omega(\cdot,y_d)=\mathbb{E}\left[ W_1\big|_{\Sigma_L}\right] . 
\end{equation}
Notice that $u_1^{FF}$ that satisfies \eqref{eq:cascade_champ_loin_alea} is deterministic.

%%%%%%%%%%%%%%%%%%%%%%%%%%%%%%%
\subsection{Construction of the first far-field terms}
%%%%%%%%%%%%%%%%%%%%%%%%%%%%%%%

The boundary conditions \eqref{eq:condition_limite_champ_loin0_Dirichlet_alea} and \eqref{eq:condition_limite_champ_loin1_Dirichlet_alea} which guarantee the existence of two near-field terms going to 0 at infinity, allow us to complete the equations satisfied by the far field terms. We can now write the problems verified by $u_{0}^{FF}$ and $u_{1}^{FF}$.

%%***********
\subsubsection*{Construction of $u_0^{FF}$}
%%***********
\begin{definition}
By using \eqref{eq:cascade_champ_loin_alea} and the boundary condition \eqref{eq:condition_limite_champ_loin0_Dirichlet_alea} on $\Sigma_{\eps H}$, we define $u_0^{FF}$ as the unique deterministic solution in $H^1\big(\mathcal{B}_{\eps H,L}\big)$ of  
\begin{equation}\label{eq:pb_champ_loin_u0_Dirichlet_Alea}
\left\{
\begin{aligned}
-\Delta u_0^{FF} - k^2 u_0^{FF} &= f \quad \text{ in } \quad \mathcal{B}_{\eps H,L},\\
u_0^{FF} &= 0 \quad\text{ on }\quad \Sigma_{\eps H},\\
-\partial_{x_d}  u_0^{FF} + \Lambda^k  u_0^{FF} &= 0 \quad \text{ on } \quad \Sigma_{L} .
\end{aligned}
\right.
\end{equation}
\end{definition}

 Since $u_0^{FF}$ is solution of the homogeneous Helmholtz equation in a half-space with a homogeneous Dirichlet condition on the hyperplane and a source term in $L^2\big(\mathcal{B}_{\eps H,L}\big)$ with compact support, 
 completed by the outgoing wave condition \eqref{eq:UPRC}, we can show that the problem \eqref{eq:pb_champ_loin_u0_Dirichlet_Alea} is well posed in $H^1\big(\mathcal{B}_{\eps H,L}\big)$ \cite{chandler2005existence}. 
{Actually, $u_0^{FF}$ can be computed explicitly using a Fourier transform in the ${\bsypar}$-direction and one can deduce classical elliptic regularity results such as $u_0^{FF}\in H^{s}\big(\mathcal{B}_{\eps H,L'}\big)$ for any $s>0$ with $L'<L$ since the support of $f$ lies far away from $\Sigma_{\eps H}$ by hypothesis (see \eqref{eq:source}).}

%%***********
\subsubsection*{Construction of $u_1^{FF}$}
%%***********
\begin{definition}
By using \eqref{eq:cascade_champ_loin_alea} and the boundary condition \eqref{eq:condition_limite_champ_loin1_Dirichlet_alea} on $\Sigma_{\eps H}$, we define $u_1^{FF}$ as the unique deterministic solution in $H^1\big(\mathcal{B}_{\eps H,L}\big)$ of 
\begin{equation}\label{eq:pb_champ_loin_u1_Dirichlet_Alea}
\left\{
\begin{aligned}
  -\Delta u_1^{FF} - k^2 u_1^{FF} &= 0 \quad \text{ in } \quad \mathcal{B}_{\eps H,L} , \\
  u_1^{FF}&= c_1 \partial_{x_d} u_0^{FF}  \quad \text{ on }\quad \Sigma_{\eps H} ,  \\
-\partial_{x_d}  u_1^{FF} + \Lambda^k  u_1^{FF} &= 0 \quad \text{ on } \quad \Sigma_{L} .
\end{aligned}
\right.
\end{equation}
where $c_1$ is given by \eqref{eq:constante_c01_Dirichlet_alea}.
\end{definition}
{The far field $u_1^{FF}$ is solution of the homogeneous Helmholtz equation in a half-space with a Dirichlet datum that is in $H^{s}\big(\Sigma_{\eps H}\big)$ for any $s>0$.  Using the Fourier transform, one can compute 
$u_1^{FF}$ explicitly and deduce in particular that $u_1^{FF}\in H^{s}(\mathcal{B}_{\eps H,L})$ for all $s>0$.}

%%%%%%%%%%%%%%%%%%%%%%%%%%%%%%%
\subsection{Effective model} \label{subsection:effective}
%%%%%%%%%%%%%%%%%%%%%%%%%%%%%%%
Either we build the far field terms incrementally, or we use the model that will approximate the truncated far field series. This can be useful for studying more general geometries. By using an approximated model, we can have a direct approximation up to a certain order. \vsd 

The first model that we propose is built from $u_0^{FF}$ satisfying  \eqref{eq:pb_champ_loin_u0_Dirichlet_Alea}. The particles and the impedance condition on the object are replaced by a conducting plane on $\Sigma_{\eps H}$. This effective impedance condition does not take into account the particles. Therefore we propose a second model. This model is built from $u_0^{FF}$ and $u_1^{FF}$. We want to have an approximation of $u_0^{FF}+\eps u_1^{FF}$ up to the second order in $\eps$.
Since $u_0^{FF}$ and $u_1^{FF}$ satisfy the Helmholtz equation, we also have
\begin{equation}
-\Delta \big(u_0^{FF} + \eps u_1^{FF} \big) \ - \ k^2 \ \big(u_0^{FF} + \eps u_1^{FF} \big) = f \quad \text{ in } \quad \mathbb{R}^{d-1}\times(\eps H, +\infty) . 
\end{equation}
Given the boundary conditions \eqref{eq:condition_limite_champ_loin0_Dirichlet_alea} and \eqref{eq:condition_limite_champ_loin1_Dirichlet_alea}
\begin{equation}
\big(u_0^{FF} + \eps u_1^{FF} \big) = \eps \ c_1 \ \partial_{x_d}u_0^{FF} \quad \text{ on }\quad \Sigma_{\eps H} . 
\end{equation}
It is then natural to introduce the following effective model.
\begin{definition}
Let $v_\eps$ satisfy 
\begin{equation}\label{eq:modele_eff_Dirichlet_Alea}
\left\{
\begin{aligned}
-\Delta v_{\eps} - k^2 v_\eps &= f \quad \text{ in }\,\,\mathcal{B}_{\eps H,L} , \\
-\eps c_1 \partial_{x_d} v_\eps  + v_\eps &= 0 \quad \text{ on }\,\, \Sigma_{\eps H} , \\
-\partial_{x_d}  v_\eps + \Lambda^k  v_\eps &= 0 \quad \text{ on } \quad \Sigma_{L} ,
\end{aligned}
\right.
\end{equation}
where $c_1$ is given by \eqref{eq:constante_c01_Dirichlet_alea}.
\end{definition}

%\begin{remark}
%We have gone from a Dirichlet condition for the model at order $1$ to a Robin condition for the model at order $2$.
%\end{remark}

If $c_1>0$, then the problem \eqref{eq:modele_eff_Dirichlet_Alea} is well posed in 
$H^1\big(\mathcal{B}_{\eps H, L}\big)$ \cite{chandler1997impedance}. 
{Moreover, $v_\eps$ can be computed explicitly using the Fourier transform in the $\bsypar$-direction, and one can deduce easily that $v_\eps\in H^s(\mathcal{B}_{\eps H, L'})$ for any $s>0$ with $L'<L$, with norms independent of $\eps$, since the support of $f$ lies far away from $\Sigma_{\eps H}$. }
Let us prove that the parameter $H$ can be chosen such that $c_1>0$.

\begin{proposition}
There exists $H_0\geq h$ such that for all $H\geq H_0$ \begin{equation}\label{eq:c10pos}c_1(H)>0,\end{equation} where $c_1(H)\coloneqq \E\left[W_1(H)\right]$ and $W_1(H)$ denotes the solution of \eqref{eq:w10} with the normal derivative jump occurring at $\Sigma_H$.
\end{proposition}

\begin{proof}Let H'>H. One can easily establish thanks to proposition \ref{prop:Utnwellposed} that
\begin{equation*}W_1(H')=W_1(H)+(y_d-H)\chi_{(H,H')}(y_d)+(H'-H)\chi_{(H',+\infty)}(y_d).\end{equation*}Taking the limit as $y_d\to+\infty$, one gets
\begin{equation*}c_1(H')=c_1(H)+(H'-H).\end{equation*}
Choosing $H'$ large enough leads to \eqref{eq:c10pos}.
\end{proof}
Finally, let us show that $v_\eps$ is indeed an approximation of $u_0^{FF}+\eps u_1^{FF}$ up to an order 2 in $\eps$.
\begin{proposition}\label{prop:effmod}
For $H$ large enough and $L>\eps H$, the following estimate holds
\begin{equation}\label{eq:effmod}
\left\| u_0^{FF}+\eps u_1^{FF}-v_\eps\right\|_{H^2(\mathcal{B}_{\eps H,L})}\underset{\eps\to0}{=}\mathcal{O}(\eps^2).
\end{equation}
\end{proposition}
\begin{proof}
Let $e_\eps\coloneqq u_0^{FF}+\eps u_1^{FF}-v_\eps$. Applying the Fourier Transform in the $\bsypar$-direction to \eqref{eq:pb_champ_loin_u0_Dirichlet_Alea}, \eqref{eq:pb_champ_loin_u1_Dirichlet_Alea} and 
\eqref{eq:modele_eff_Dirichlet_Alea}, one obtains that 
\[ 
\widehat{e_\eps}(\zeta,x_d)=\eps^2\frac{c_1\partial_{x_d}\widehat{u}^{FF}_1(\zeta,\eps H)}{1-\ii\eps c_1\sqrt{k^2-|\boldsymbol{\zeta}|^2}}e^{\ii \sqrt{k^2-|\boldsymbol{\zeta}|^2}x_d}
\]
where $\widehat{e_\eps}$ (resp. $\partial_{x_d}\widehat{u}^{FF}_1$) denotes the Fourier Transform of $e_\eps$ (resp. $\partial_{x_d}{u_1}$) and by convention $\sqrt{k^2-|\boldsymbol{\zeta}|^2}=\ii\sqrt{|\boldsymbol{\zeta}|^2-k^2}$. Using the characterization of Sobolev spaces from the Fourier Transform (see \eqref{eq:def_Hs}), one easily shows that 
\[ 
\|e_\eps\|_{H^2(\mathcal{B}_{\eps H,L})}\lesssim \eps^2 \|\partial_{x_d}{u}_1^{FF}\|_{H^{3/2}(\Sigma_{\eps H})}.
\]

\end{proof}
\section{Error estimates} \label{sec:errest}
%%%%%%%%%%%%%%%%%%%%%%%%%%%%%%%

In this section we estimate
\begin{itemize}\item[-]  the error between $u_\eps$ and our two-scale asymptotic expansion in $H^1(\mathcal{B}^{\omega}_{L})$ (Recall that $\mathcal{B}^{\omega}_{L}\coloneqq \R^{d-1}\times(0,L)\setminus\Pa^\omega$ corresponds to the infinite strip (outside the particles) between $\Sigma_0$ and $\Sigma_L$ with $L>\eps H$); 
  \item[-] the error between $u_\eps$ and the far fields  $u^{FF}_0+\eps u^{FF}_1$ in $H^1(\mathcal{B}_{L',L})$. Recall that $\mathcal{B}_{L',L}$ is the infinite strip $\{(\bsxpar,x_d),\,L'<x_d<L\}$ with $L'>\eps H$. To estimate the latter it is necessary to consider a strip sufficiently far away from the layer (see Figure \ref{fig:error}) so that the contribution of the near-fields in the asymptotic expansion are negligible. This motivates our choice of considering a strip at distance of order 1 from the small particles. 
 \item[-] the same two errors replacing the far fields  $u^{FF}_0+\eps u^{FF}_1$ by the effective solution $v_\eps$.
\end{itemize}
 \begin{figure}[htbp]
 \begin{center}
    \begin{center}
  \begin{tikzpicture}[scale=0.5]
%    \draw[gray,->] (-4.5,0) -- (-3.85,0);
%    \draw [gray] (-4.15,-0.1) node[below] {$x_1$};
%    \draw [gray,->] (-4.5,0) -- (-4.5,0.6);
%    \draw [gray] (-4.5,0.3) node[left] {$x_3$}; 
%    %%%%%%
%    \draw[color3,thick,>=stealth,->] (-3,2.5) -- (-2.25,2.);
%    \draw[color3,thick] (-3.2,2.2) -- (-2.8,2.8);
%    \draw[color3,thick] (-3,2) -- (-2.6,2.6);
    %\draw[color3,thick] (-3.3,3.5) node[above] {\small Onde incidente};
    %\draw[color3,thick] (-3.5,2.5) node[above] {\small\it{$u_{inc}$}};
    %%%%%%
     \fill[color1!20!white] (-3.2,0)--(3,0)--(3,3.5)--(-3.2,3.5)--cycle;
    \draw[plum,thick] (-3.2,1.3) -- (3,1.3);  
    \draw [plum] (3,1.3) node[right] {$\Sigma_{\eps H}$} ; 
    
    \pattern[pattern=north  west lines, opacity =0.5, pattern color=color3, thick] (-3.2,2.8)--(3,2.8)--(3,3.5)--(-3.2,3.5)--cycle;
    \draw[color1,thick]  (-3.2,3.5) -- (3,3.5); 
        \draw [color1] (3,3.5) node[right] {$\Sigma_L$} ; 
          \draw [color3] (3,2.8) node[right] {$\Sigma_{L'}$} ; 
              \draw[color3,thick]  (-3.2,2.8) -- (3,2.8);

       \draw [color3] (-3.2,3.1) node[left] {$\mathcal{B}_{L',L}$} ; 
        \draw [color1] (-3.2,1.5) node[left] {$\mathcal{B}_{L}^\omega$} ; 
        \draw [color1] (-3.2,0.2) node[left] {$\eps r_0$} ; 
    
    \draw[green3,thick] (-3.2,-0.1) -- (3,-0.1);

    %\draw [green3] (1.3,0) node[below] {$\nabla u^{\varepsilon}.\vec n +\gamma u^{\varepsilon}=0  $};
    %\draw [green3] (6.2,0) node[right] {(CI)} ; 
       \draw [green3] (3,-0.1) node[right] {$\Sigma_0$} ; 
    
    %%%%%%
    %\draw [color3] (1.3,2.4) node[below] {$-\Delta u^{\varepsilon} - k^2 u^{\varepsilon} = 0 $};
    %%%%%%
    %\draw[>=stealth,<-]  (-3.15,0.65) -- (-3.5,0.65);
    %\draw (-3.4,0.65) node[left] {$\nabla u^{\varepsilon}.\vec n =0 $};
    %%%%%%
    \draw [thick,color1, fill=white] (-2.9,0.65) circle(0.25);
    \draw [thick,color1, fill=white](-2.5,0.25) circle(0.25);
    \draw [thick,color1, fill=white](-2.2,0.8) circle(0.25);
    \draw [thick,color1, fill=white](-1.7,0.4) circle(0.25);
    \draw [thick,color1, fill=white](-1.1,0.3) circle(0.25);
    \draw [thick,color1, fill=white] (-0.85,0.85) circle(0.25);
    \draw [thick,color1, fill=white](-0.3,0.65) circle(0.25);
    \draw [thick,color1, fill=white] (0.3,0.85) circle(0.25);
    \draw [thick,color1, fill=white] (0.5,0.30) circle(0.25);
    \draw [thick,color1, fill=white](1.1,0.40) circle(0.25);
    \draw [thick,color1, fill=white](1.9,0.65) circle(0.25);
    \draw [thick,color1, fill=white] (2.3,0.30) circle(0.25);
    \draw [thick,color1, fill=white](2.7,0.8) circle(0.25);
%    \draw [thick] (3.1,0.4) circle(0.25);
%    \draw [thick] (3.7,0.3) circle(0.25);
%    \draw [thick] (3.9,0.8) circle(0.25);
%    \draw [thick] (4.5,0.65) circle(0.25);
%    \draw [thick] (5.1,0.85) circle(0.25);
%    \draw [thick] (5.4,0.30) circle(0.25);
%    \draw [thick] (5.9,0.50) circle(0.25);
    \draw [color=color1,-,thick] (-2.5,0.30) -- (-2.25,0.30); 
 
    %%%%%%%
  %  \draw[color1] (-2.45,-0.7) node[above]{$1$} ;    
    %\draw[color1] (-3.6,0.3) node[above]{\small $h_L$} ;
     % \draw [color1] (-3.2,0.6) node[left] {$\partial \Pa^\omega$} ; 
    %%%%%%%
   % \draw (5.8,2.0) node[above]{\small $k=\dfrac{2\pi}{\lambda}$} ;
   % \draw[color1] (5.8,1.4) node[above]{\small$\tcbhighmath[top=0.1cm, bottom=0.1cm,left=0.1cm,right=0.1cm]{\varepsilon \ll \lambda}$} ;
    %%%%%%%
   % \draw [color3] (4.4,3) node[above]{\small$u^\varepsilon-u_{inc}$ \textit{sortant}} ;
  \end{tikzpicture}
\end{center}  
 \end{center}
 \caption{Illustration of the different domains where the errors are estimated in section \ref{sec:errest}}\label{fig:error}
 \end{figure}
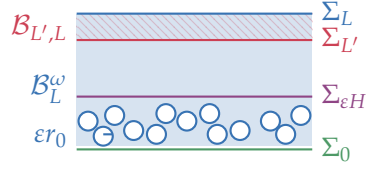
We establish those error estimates in the specific case where the distance between two particles is bounded a.s. and a.e. in the layer. We assume indeed  that

\begin{hypothesis}[$L^{\infty}$]\makeatletter\def\@currentlabelname{($L^{\infty}$)}\makeatother\label{hyp:Linfty}
The distance to the nearest particule $R$ defined in \eqref{eq:def_R} satisfies $$\sup_{y_d\in(0,h)}R(\cdot,y_d)\in L^{\infty}(\Omega, L^{\infty}(\R^{d-1})).$$ %$\dst R^\cdot(\cdot, h)\in L^{\infty}(\Omega, L^{\infty}(\R^{d-1}))$
\end{hypothesis}

\begin{remark}\label{rem:WP_Linfty}
  Hypothesis \nameref{hyp:Linfty} corresponds to a stronger assumption than \nameref{hyp:mixing} for the particle's distribution. A direct consequence is, since $\mu$ is bounded from below, that the near-fields solutions of problem of the form \eqref{eq:NFtGen} are in
\begin{multline*}
 \mathcal{H}_0(D)\coloneqq  \left\{\text{a.s. } V^\omega\in H^1_{loc}(D^\omega), V^\omega\big(\cdot,y_d\big) \text{ stationary for any }y_d>0, \right. \\\left. \text{a.s.  } V^\omega=0 \text{ on } \Pa^\omega,\,\E\left[ \int_{\square_1}\int_0^{+\infty}\mathbb{1}_{D^\omega}\Big( \frac{|V^\omega|^2}{1+y_d^2}+|\nabla V^\omega|^2\Big) \, \mathrm{d}\bsy\right] <+\infty \right\}.
\end{multline*}
 \end{remark}
We first prove in Section \eqref{subsection:firstest} estimates relying only on Hypothesis \nameref{hyp:Linfty}. We then improve those estimates in subsection \ref{subsection:quantest} by adding a quantitative mixing assumption (Hypothesis \nameref{hyp:mix}) on the particle's distribution.

\subsection{Technical lemmatas}

 Let us start with the following general lemma and its corollary that provide  
 $L^2(\mathcal{B}^\omega_L)$- and $H^1(\mathcal{B}_{L',L})$-estimates of near-field terms depending on their behaviour as $y_d\to\infty$. 

\begin{lemma}\label{lem:err1}
Let $u\in H^1(\R^{d-1})$ and $U\in\mathcal{H}_0(D)$ be such that 
\begin{equation}\label{eq:firstrate}
  \E[|U|^2(\cdot,y_d)]+y_d^{2}\,\E[|\nabla U|^2(\cdot,y_d)]\underset{y_d\to+\infty}{=}o(1).
\end{equation}
Then $\displaystyle\mathcal{U}_\eps:\bsx\mapsto u(\bsxpar)\,U(\frac{\bsx}{\eps})$ satisfies 
\begin{equation}\label{eq:energy_estimate}
\left\|\mathcal{U}_{\eps}\right\|_{L^2(\Omega, L^2(\mathcal{B}^{\omega}_{L}))}\leq \eta(\eps)\|u\|_{L^2(\R^{d-1})}\quad\text{and}\quad
+\left\|\mathcal{U}_{\eps}\right\|_{L^2(\Omega,H^1(\mathcal{B}_{L',L}))}\leq \eta(\eps)\|u\|_{H^1(\R^{d-1})},
\end{equation}
where $\eta(\eps)$ tends to $0$ as $\eps$ tends to $0$.
\end{lemma}

\begin{proof}
  By stationarity of $U$ (and a change of variable in the first integral), we get immediately
\begin{equation}\label{eq:inter1}
\E\left[\int_0^L\int_{\R^{d-1}}|u(\bsxpar)|^2\left|U\left(\frac{\bsx}{\eps}\right)\right|^2\dd \bsxpar\dd x_d\right]
=\eps\,\int_0^{L/\eps}\E\left[\left|U\left(\cdot,y_d\right)\right|^2\right]\dd y_d\int_{\R^{d-1}}|u(\bsxpar)|^2\dd \bsxpar.
\end{equation}
Let $\delta>0$. By \eqref{eq:firstrate}, there exists $L_0\geq 0$ such that for all $y_d\geq L_0$, 
\begin{equation}\label{eq:inter2}
  \E[|U|^2(\cdot,y_d)]+y_d^2\E[|\nabla U|^2(\cdot,y_d)]\leq \frac{\delta}{2L}.
\end{equation}
We can then write
\begin{equation*}
\int_0^{{L}/{\eps}} \E\left[\left|U\left(\cdot,y_d\right)\right|^2\right]\dd y_d\leq\int_0^{L_0}\E\left[\left|U\left(\cdot,y_d\right)\right|^2\right]\dd y_d+\frac{\delta}{2\eps}.
\end{equation*}
This and the previous relation prove that for $\eps$ small enough, $\left\|\mathcal{U}_\eps\right\|^2_{L^2(\Omega, L^2(\mathcal{B}^{\omega}_{L}))}\leq \delta\,\|u\|^2_{L^2(\R^{d-1})}$.
Moreover
\begin{equation}\label{eq:inter3bis}
  \|\nabla\mathcal{U}_\eps\|_{L^2(\Omega, L^2(\mathcal{B}_{L',L}))}\leq\|\nabla_{\bsxpar} u\,U\left(\frac{\cdot}{\eps}\right)\|_{L^2(\Omega, L^2(\mathcal{B}_{L',L}))}+\frac{1}{\eps}\|u\,\nabla U\left(\frac{\cdot}{\eps}\right)\|_{L^2(\Omega, L^2(\mathcal{B}_{L',L}))}.
\end{equation} 
The first term can be dealt with in the same way as $\left\|u(\cdot)U\left(\frac{\cdot}{\eps}\right)\right\|^2_{L^2(\Omega, L^2(\mathcal{B}^{\omega}_{L}))}$ replacing $u$ by $\nabla_{\bsxpar}u$.
For the second term, we use \eqref{eq:inter2} to obtain
\begin{equation}\label{eq:inter3}
  \begin{array}{ll}
  \dst \frac{1}{\eps^2}\|u(\cdot)\nabla U\left(\frac{\cdot}{\eps}\right)\|^2_{L^2(\Omega, L^2(\mathcal{B}_{L',L}))}
  &=\dst\frac{1}{\eps}\|u\|^2_{L^2(\R^{d-1})}\int_{\frac{L'}{\eps}}^{\frac{L}{\eps}} \E\left[\left|\nabla U\left(\cdot,y_d\right)\right|^2\right]\dd y_d
  \\&\dst\leq\frac{\delta}{2L\eps}\|u\|^2_{L^2(\R^{d-1})}\int_{\frac{L'}{\eps}}^{\frac{L}{\eps}}\frac{1}{y_d^{2}}\dd y_d\lesssim \delta\|u\|^2_{L^2(\R^{d-1})}.
\end{array}\end{equation}
\end{proof}

\begin{corollary}\label{cor:betterNFest}
Let $u\in L^2(\R^{d-1})$. Consider $U\in\mathcal{H}_0(D)$ such that  
\begin{equation}\label{eq:convu}
  y_d^{d-1}\E[|U|^2(\cdot,y_d)]+y_d^{d+1}\E[|\nabla U|^2(\cdot,y_d)]\underset{y_d\to+\infty}{=}\mathcal{O}(1).\end{equation} 
  Then $\displaystyle\mathcal{U}_\eps:\bsx\mapsto u(\bsxpar)\,U(\frac{\bsx}{\eps})$ satisfies 
  \eqref{eq:energy_estimate} with 
\begin{equation*}
\eta(\eps)\lesssim\left\{\begin{array}{ll}\dst\mathcal{O}\left(\eps^{1/2}|\log \eps|^{1/2}\right)&\text{for }\; d=2,
  \\\dst\mathcal{O}(\eps^{1/2})&\text{for }\; d=3.\end{array}\right.
\end{equation*}
\end{corollary}
\begin{proof} We follow the same ideas as in the proof of Lemma \ref{lem:err1}. Thanks to \eqref{eq:convu} we can improve our initial estimate of $\dst\int_0^{\frac{L}{\eps}} \E\left[\left|U\left(\cdot,y_d\right)\right|^2\right]\dd y_d$ as such
\begin{align*}
\int_0^{\frac{L}{\eps}} \E\left[\left|U\left(\cdot,y_d\right)\right|^2\right]\dd y_d&\leq\int_0^{L'}\E\left[\left|U\left(\cdot,y_d\right)\right|^2\right]\dd y_d+C'\int_{L'}^{\frac{L}{\eps}}y_d^{-(d-1)}\dd y_d\\&=\int_0^{L'}\E\left[\left|U\left(\cdot,y_d\right)\right|^2\right]\dd y_d+ C'\left\{\begin{array}{ll}\dst\log \frac{L}{\eps}-\log L'&\textrm{for}\, d=2,\\\dst(L'^{-1}-\frac{\eps}{L})&\textrm{for}\, d=3.\end{array}\right.
\end{align*}
This inequality combined with \eqref{eq:inter1} gives us the first result. Moreover, one has to adapt \eqref{eq:inter3} using that
\begin{equation*}
  \int_{\frac{L'}{\eps}}^{\frac{L}{\eps}}y_d^{-(d+1)}=\mathcal{O}(\eps^d).
\end{equation*}
\end{proof}
\subsection{Estimates under Hypothesis \nameref{hyp:Linfty}}\label{subsection:firstest}
Suppose that $u_0^{FF}$ and $u_1^{FF}$ satisfy respectively \eqref{eq:pb_champ_loin_u0_Dirichlet_Alea} and \eqref{eq:pb_champ_loin_u1_Dirichlet_Alea} 
then $U_{1}^{NF}$ is given by \eqref{eq:U_1_expr} that rewrites for $\bsxpar\in \R^{d-1}, \bsy\in D^\omega,$ 
\begin{equation}\label{eq:U1v}
U_1^{\omega, NF}(\bsxpar;\bsy)= \partial_{x_d} u_{0}^{FF}\big|_{\Sigma_{\eps H}}(\bsx_\shortparallel) V_1^\omega(\bsy) 
\quad\text{where}\;  V_1^\omega(\bsy) \coloneqq -c_1\mathbb{1}_{\mathcal{B}^\infty_H}(\bsy)+ W_1^\omega(\bsy).
\end{equation}
Consider now $W_2$ the unique solution in $\mathcal{H}_0(D)$ of the following near-field problem
\begin{equation}\label{eq:w2}
\begin{array}{|l}
   -\Delta_{\bsy} W_2^\omega=0 \quad \;\text{ in } \quad\mathcal{B}_{H}^{\omega} \cup \mathcal{B}_{ H}^{\infty}, \\
  -\partial_{y_d} W_2^\omega=-ik\gamma V_1^\omega\; \text{ on } \;\Sigma_0, \quad\text{and}\quad
 W_2^\omega = \ 0 \;\text{ on } \; \partial \mathcal{P}^\omega, \\
 \Big[W_2^\omega  \Big]_H = \ 0 \quad\text{ and }\quad   \Big[-\partial_{y_d}W_2^\omega  \Big]_H =0.
\end{array}
\end{equation}
Since $W_1|_{\Sigma_0}\in\mathcal{L}^2(\Sigma_0)$ under Hypothesis \nameref{hyp:Linfty}, then $V_1|_{\Sigma_0}\in\mathcal{L}^2(\Sigma_0)$ and \eqref{eq:w2} is well-posed in $\mathcal{H}_0(D)$ (see Remark \ref{rem:WP_Linfty}). Moreover we know that the results of Proposition \ref{prop:limNF} can be applied to $W_2$. Let us denote $c_2\coloneqq \E[W_2|_{\Sigma_L}]$, $V_2^\omega\coloneqq -c_2\mathbb{1}_{\mathcal{B}^\infty_H}+ W_2^\omega$ and 
\begin{equation}\label{eq:U2t}
\forall \bsxpar\in \R^{d-1}, \bsy\in D^\omega,\,\, \widetilde{U}_2^{\omega, NF}(\bsxpar;\bsy)\coloneqq \partial_{x_d} u_{1}^{ FF}\big|_{\Sigma_{\eps H}} (\bsxpar)V_1^\omega(\bsy)+\partial_{x_d} u_{0}^{FF}\big|_{\Sigma_{\eps H}} (\bsxpar)V_2^\omega(\bsy).
\end{equation}

 \begin{remark}Note that we cannot prove that Problem \eqref{eq:NF} with $n=2$ is well posed since 
  it is not clear in general that ${\bf G}_1$, which depends on $\nabla_{\bsxpar} U_1^{NF}$ (see \eqref{eq:NF}), is in $\mathcal{L}^2(D)$. The function $\widetilde{U}_2^{\omega, NF}$ is somehow the part of ${U}_2^{\omega, NF}$ that is well-defined and enables us to perform the error analysis.
%\begin{equation*}U_2^{NF}(\bsxpar;\bsy)\coloneqq \widehat{U}_2^{NF}(\bsxpar;\bsy)-u_{2}^{FF}\big|_{\Sigma_{\eps H}}(\bsx_\shortparallel) \mathbb{1}_{\mathcal{B}^\infty_H}(\bsy)+\nabla_{\bsxpar} \partial_{x_d} u_{0}^{FF}\big|_{\Sigma_{\eps H}}(\bsx_\shortparallel) \cdot \boldsymbol{W_2^{\omega'}}(\bsy),\end{equation*}where $\boldsymbol{W_2'}$ verifies\begin{equation}\label{eq:w2'}
%\begin{array}{|ll}
%   -\Delta_{\bsy} \boldsymbol{W_2^{\omega'}}=2 \nabla_{\bsypar}V_1^\omega\quad &\text{ in } \quad\mathcal{B}_{H}^{\omega} \cup \mathcal{B}_{ H}^{\infty}, \\
%  -\partial_{y_d} \boldsymbol{W_2^{\omega'}}=0& \text{ on } \quad \Sigma_0, \\
%\boldsymbol{W_2^{\omega'}}= \ 0 &\text{ on } \quad \partial \mathcal{P}^\omega, \\
% \Big[\boldsymbol{W_2^{\omega'}} \Big]_H = \ 0 \quad &\text{ and }\quad   \Big[-\partial_{y_d}\boldsymbol{W_2^{\omega'}}  \Big]_H =0.
%\end{array}
%\end{equation}
%Since $\nabla V_1(\cdot, y_d)=\mathcal{O}(y_d^{-1})$ (proposition \ref{prop:limNF}), \eqref{eq:w2'} is not well-posed in $\mathcal{H}_0(D)$. 

\end{remark}

We estimate in the theorem below 
\begin{itemize}
\item[-]the error between $u_\eps$ and the second-order multi-scale asymptotic expansion 
\begin{equation*}
  w_\eps^\omega:\bsx\mapsto (u_0^{FF}(\bsx)+\eps u_1^{FF}(\bsx))\mathbb{1}_{x_d\geq\eps H}(\bsx)+\left[\eps U_1^{\omega, NF}+\eps^2 \widetilde{U}_2^{\omega, NF}\right]\left(\bsxpar;\frac{\bsx}{\eps}\right) \in  L^2(\Omega,H^1(\mathcal{B}^\omega_L));
\end{equation*}
\item[-] the error between $u_\eps$ and $u_0^{FF}+\eps u_1^{FF}$ in $H^1(\mathcal{B}_{L',L})$.\end{itemize}

\begin{theorem} \label{thm:error} 
  Suppose that $\Pa$ verifies Hypothesis \nameref{hyp:Linfty}.
For all $L>L'>\eps H$, the following estimates hold
\begin{equation}\label{eq:err1a}
\left\| u_\eps-w_\eps\right\|_{L^2\big(\Omega,H^1(\mathcal{B}^{\omega}_{L})\big)} \underset{\eps\to0}{=} o(\eps),
\end{equation}
and
\begin{equation}\label{eq:err1b}
\left\| u_\eps-(u_0^{FF}+\eps u_1^{FF})\right\|_{L^2\big(\Omega,H^1(\mathcal{B}_{L',L})\big)} \underset{\eps\to0}{=} o(\eps).
\end{equation}
\end{theorem}
Thanks to Proposition \ref{prop:effmod}, similar error estimates can be proven replacing the far field expansion $u_0^{FF}+\eps u_1$ by the effective solution $v_\eps$. More precisely 
let us introduce for $\bsx\in D^\omega_\eps$
\begin{equation*}
  \widetilde{w}_\eps(\bsx):=v_\eps(\bsx)\mathbb{1}_{x_d\geq\eps H}(\bsx)+\partial_{x_d} v_\eps\big|_{\Sigma_{\eps H}}(\bsxpar) V^\omega_1\left(\bsxpar;\frac{\bsx}{\eps}\right)+\eps^2\partial_{x_d} v_\eps\big|_{\Sigma_{\eps H}}(\bsxpar) V^\omega_2\left(\bsxpar;\frac{\bsx}{\eps}\right).
\end{equation*} 
In the following corollary we estimate $u_\eps-\widetilde{w}_\eps$ in $H^1(\mathcal{B}^{\omega}_{L})$ and $u_\eps-v_\eps$ in $H^1(\mathcal{B}_{L',L})$.

\begin{corollary} \label{cor:error1} 
For all $L>L'>\eps H$, under Hypothesis \nameref{hyp:Linfty}, the following estimates hold
\begin{equation}\label{eq:cor1a}
\left\| u_\eps-\widetilde{w}_\eps\right\|_{L^2\big(\Omega,H^1(\mathcal{B}^{\omega}_{L})\big)} \underset{\eps\to0}{=} o(\eps),
\end{equation}
and
\begin{equation}\label{eq:cor1b}
\left\| u_\eps-v_\eps\right\|_{L^2\big(\Omega,H^1(\mathcal{B}_{L',L})\big)} \underset{\eps\to0}{=} o(\eps).
\end{equation}
\end{corollary}

\begin{proof}[Proof of Theorem \ref{thm:error}]
Let us write the problem verified by $w^\omega_\eps$ in $\mathcal{B}^{\omega}_{L}$.
$w^\omega_\eps$ satisfies on one hand for $\bsx\in\mathcal{B}^\omega_L$
\begin{equation*}
\begin{array}{rl}
\vsd\dst-(\Delta+k^2)w^\omega_\eps (\bsx)
=&\dst f(\bsx)-\eps(\Delta +k^2 )\Big[(U_{1}^{\omega,NF}+ \eps\widehat{U}_2^{\omega,NF})\left( \bsx_\shortparallel;\frac{\bsx}{\eps}\right) \Big] \\ \vsd
=&\dst f(\bsx)+ \left[-\frac{1}{\eps} \Delta_{\bsy}- \nabla_{\bsx_\shortparallel} \cdot\nabla_{\bsy} - \eps \Big( \Delta_{\bsx_\shortparallel}+k^2\Big)\right](U_{1}^{\omega,NF}+ \eps\widetilde{U}_2^{\omega,NF})\left( \bsx_\shortparallel;\frac{\bsx}{\eps}\right) \\
=&\dst f(\bsx) - \eps \nabla\cdot\left[\nabla_{\bsxpar}( U_{1}^{\omega,NF}+ \eps\widetilde{U}_2^{\omega,NF})\left( \bsx_\shortparallel;\frac{\bsx}{\eps}\right)\right]-\eps k^2( U_{1}^{\omega,NF}+ \eps\widetilde{U}_2^{\omega,NF})\left( \bsx_\shortparallel;\frac{\bsx}{\eps}\right),
\end{array}
\end{equation*} 
where we used that $U_1^{NF}$ and $\widetilde{U}_2^{NF}$ are solutions of near-field problems with no volume source terms and that $\eps\nabla [V(\bsxpar;\bsx/\eps)]=\eps[\nabla_{\bsxpar}V](\bsxpar;\bsx/\eps)+[\nabla_y V](\bsxpar;\bsx/\eps)$. On the other hand $w^\omega_\eps$ verifies the following jumps and boundary conditions
\begin{equation*}
\begin{array}{|l}
\dst \Big[ w^\omega_\eps\Big]_{\eps H}=0,\,\,\text{and}\,\,\Big[ \partial_{x_d}w^\omega_\eps\Big]_{\eps H}=0,\\ 
\dst -\partial_{x_d}w^\omega_\eps+ik\gamma w^\omega_\eps=\eps^2\ ik\gamma \widetilde{U}_2^{\omega,NF}|_{\Sigma_0}\quad\text{on\,}\; \Sigma_0,\quad\text{and}\quad \dst w^\omega_\eps=0\;\text{on\,}\; \partial \mathcal{P}^\omega,\\
 \partial_{x_d} w^\omega_\eps + \Lambda^k w^\omega_\eps=\dst [\frac{1}{\eps}\partial_{y_d} + \Lambda^k](\eps U_{1}^{\omega,NF}+ \eps^2\widetilde{U}_2^{\omega,NF})\quad\text{on\,} \Sigma_L,
 \end{array}
\end{equation*}where $\Lambda^K$ denotes the DtN operator introduced in the definition \ref{definition_DtN_sol_ref_alea}.

We can now write the variational formulation verified by $e_\eps\coloneqq  u_\eps-w_\eps$ in $L^2\big(\Omega,H^1(\mathcal{B}^{\omega}_{L})\big)$. 
For all $v \in L^2\big(\Omega,H^1(\mathcal{B}^{\omega}_{L})\big)$, 
\begin{equation}\label{eq:FVerror}
   \mathbb{E}\left[a_\eps^\omega(e_\eps^\omega,v^\omega)\right] =\eps \ell_\eps^1(v)+\eps^2\ell_\eps^2(v)+\eps\ell_\eps^3(v),
\end{equation}
where $a_\eps^\omega$ is defined in \eqref{eq:a_epsomega} and 
\begin{equation}\label{eq:FVerror_ell}
  \begin{array}{ll}
\ell_\eps^1(v)&= -\dst\mathbb{E}\left[ \int_{\mathcal{B}^{\omega}_{L}}\mathbb{1}_{D^\omega}\nabla_{\bsx_\shortparallel} \left(U_{1}^{\omega,NF} +\eps \widetilde{U}_{2}^{\omega,NF}\right)\big(\bsx_\shortparallel;\frac{\bsx}{\eps}\big)\cdot \overline{ \nabla v^\omega}(\bsx)\dd\bsx\right.\\
&\vsd\dst\hspace{2cm}+ k^2\int_{\mathcal{B}^{\omega}_{L}} \mathbb{1}_{D^\omega} \left(U_{1}^{\omega,NF} +\eps \widetilde{U}_{2}^{\omega,NF}\right)\big(\bsx_\shortparallel;\frac{\bsx}{\eps} \big) \overline{v^\omega}(\bsx)\dd\bsx\left.\right]\\
\ell_\eps^2(v)&= \dst ik\gamma\,\mathbb{E}\left[\int_{\Sigma_0} \widetilde{U}_{2}^{\omega,NF}\Big|_{\Sigma_0}\big(\bsx_\shortparallel;\frac{\bsx_\shortparallel}{\eps}\big) \overline{v^\omega}\big|_{\Sigma_0}(\bsx_\shortparallel)\dd\bsx_\shortparallel\right]\\
\ell_\eps^3(v)&= \dst\mathbb{E}\left[\int_{\Sigma_{L}}\overline{v^\omega}|_{\Sigma_L}(\bsxpar) \left[\frac{1}{\eps}\partial_{y_d} + \Lambda^k \right] \left(U_{1}^{\omega,NF} +\eps \widetilde{U}_{2}^{\omega,NF}\right)\Big|_{\Sigma_L}\big(\bsx_\shortparallel;\frac{\bsx_\shortparallel}{\eps}\big)  \dd\bsx_\shortparallel \right].
\end{array}
\end{equation}
Let us now bound the norm of each anti-linear form appearing in \eqref{eq:FVerror}. For $\ell_\eps^1$, both volume integrals 
can be bounded applying Lemma \ref{lem:err1} to $U_1^{NF}, \nabla_{\bsxpar} U_1^{NF}, \widetilde{U}_2^{NF}$ 
or $\nabla_{\bsxpar}\widetilde{U}_2^{NF}$ with $U$ being $V_i, i=1,2$ and $u$ being $\partial_{x_d} u_i^{FF}\big|_{\Sigma_{\eps H}}$ or $\nabla_{\bsxpar}\partial_{x_d} u_i^{FF}\big|_{\Sigma_{\eps H}}$. Note that $V_i, i=1,2$ verify \eqref{eq:firstrate} by Proposition \ref{prop:limNF} and $\partial_{x_d} u_i^{FF}\big|_{\Sigma_{\eps H}}, \nabla_{\bsxpar}\partial_{x_d} u_i^{FF}\in H^{1}(\Sigma_{\eps H}), i=1,2$ ($\Sigma_{\eps H}$ being identified to $\R^{d-1}$). By Cauchy-Schwarz inequality, we get 
\begin{equation*}
|\ell_\eps^1(v)|\lesssim \eta(\eps) \|v\|_{L^2(\Omega,H^1(\mathcal{B}^{\omega}_{L}))},
\end{equation*}
where $\eta(\eps)$ tends to $0$ as $\eps$ tends to $0$.

For $\ell_\eps^2$, we use that under Hypothesis \nameref{hyp:Linfty} the following trace theorem (analogue to 
Lemma \ref{lem:weightedtrace}) holds for all $U\in\mathcal{H}_0(D)$
\begin{equation*}
  \E\left[\left|U\big|_{\Sigma_0}(\cdot)\right|^2\right]\lesssim\E\left[\int_{0}^{+\infty}\int_{\square_1}\left|\nabla U\right|^2(\bsy)\dd\bsy\right].
\end{equation*}
We obtain for $u\in L^2(\R^{d-1})$ and  $U\in\mathcal{H}_0(D)$
\begin{equation*}
  \E\left[\int_{\Sigma_0}\left|u(\bsxpar)U\big|_{\Sigma_0}\left(\frac{\bsxpar}{\eps}\right)\right|^2\dd\bsxpar\right]
  \lesssim\|u\big|_{\Sigma_0}\|^2_{L^2(\Sigma_0)}\E\left[\int_{0}^{+\infty}\int_{\square_1}\left|\nabla U\right|^2(\bsy)\dd\bsy\right].
\end{equation*}
Applying this result to $\widetilde{U}_2^{NF}$ we get by Cauchy-Schwarz inequality
\begin{equation*}
|\ell_\eps^2(v)|\lesssim \|v\big|_{\Sigma_0}\|_{L^2(\Omega,L^2(\Sigma_0))}.
\end{equation*}
We conclude with a trace inequality in $L^2(\Omega, H^1(\mathcal{B}^\omega_L))$ 
\begin{equation}\label{eq:tracewoutmu}
  \forall \alpha\in[0,L],\quad \|\mathbb{1}_D v|_{\Sigma_\alpha}\|_{L^2(\Omega,L^2(\Sigma_\alpha))}\lesssim \|\mathbb{1}_Dv\|_{L^2(\Omega,H^1(\mathcal{B}_L^\omega))}.
\end{equation}
%To prove \ref{eq:tracewoutmu}, one applies the following trace inequality 
%in boxes $\square_{R_+}(n)\times(0,L)$ \begin{equation}\label{eq:Dtracen}\|\mathbb{1}_Dv|_{\Sigma_\alpha}\|_{\mathcal{L}^2(\Omega,L^2(\Sigma_\alpha\cap\square_{R_+(n)}))}\lesssim \|\mathbb{1}_Dv\|_{\mathcal{L}^2(\Omega,H^1_0(\square_{R_+}(n)\times(0,L)))},
%\end{equation}where $n\in\mathbb{Z}^{d-1}$ and \begin{equation*}R_+\coloneqq \text{esssup}_{\omega\in\Omega}\sup_{\bsypar\in\R^{d-1}} R^\omega(\bsypar,h).\end{equation*}\eqref{eq:Dtracen} holds since at least one particle lies in each $\square_{R_+}(n)\times(0,L)$. By summing \eqref{eq:Dtracen} over $n\in\mathbb{Z}^{d-1}$ one gets \eqref{eq:tracewoutmu}.
Finally let us focus on $\ell_\eps^3$.
Since by Proposition \ref{prop:limNF} $\|V_i(\cdot,y_d)\|_{L^2(\Omega)}+y_d\|\nabla V_i(\cdot,y_d)\|_{L^2(\Omega)}\underset{y_d\to+\infty}{=}o(1)$  for $i=1$ and $2$, we deduce that
\begin{equation*}
\|V_i\big|_{\Sigma_{{L}/{\eps}}}\|_{L^2(\Omega)}\underset{\eps\to0}{=}o(1),\quad\text{and}\quad \|\partial_{y_d}V_i|_{\Sigma_{{L}/{\eps}}}\|_{L^2(\Omega)}\underset{\eps\to0}{=}o(\eps).
\end{equation*}
By Definition \eqref{eq:U1v} and \eqref{eq:U2t} of $U_1^{NF}$ and $\widetilde{U}_2^{NF}$ and thanks to the continuity of $\Lambda^k$ and \eqref{eq:tracewoutmu} we deduce that 
\begin{equation*}
|\ell_\eps^3(v)|\lesssim \eta(\eps)\,\|v\|_{L^2(\Omega,H^1(\mathcal{B}_L^\omega))},
\end{equation*}
Therefore, we have shown that, for any $ v\in L^2\Big( \Omega,H^1\big(\mathcal{B}^{\omega}_{L}\big)\Big)$
\begin{equation}
\big| \eps \ell_\eps^1(v)+\eps^2\ell_\eps^2(v)+\eps\ell_\eps^3(v)\big| \leq  \eps \,\eta(\eps)\, \| v\|_{L^2(\Omega,H^1(\mathcal{B}_L^\omega))}.
\end{equation}
Thanks to the well-posedness result of Corollary \ref{cor:wpue} we deduce that
\begin{equation*}
\| e_\eps\|_{L^2(\Omega,H^1_0(\mathcal{B}_L^\omega))}\underset{\eps\to0}{=} \eps \,\eta(\eps). 
\end{equation*} 
This ends the proof of \eqref{eq:err1a}.

To prove \eqref{eq:err1b} we decompose the error in two parts
\begin{equation*}
\left\| u_\eps-(u_0^{FF}+\eps u_1^{FF})\right\|_{L^2\big(\Omega,H^1(\mathcal{B}_{L',L})\big)}\leq \left\| u_\eps-w_\eps\right\|_{L^2\big(\Omega,H^1(\mathcal{B}_{L',L})\big)}+\left\| \eps U_1^{NF}+\eps^2 \widetilde{U}_2^{NF}\right\|_{L^2\big(\Omega,H^1(\mathcal{B}_{L',L})\big)}
\end{equation*}
From \eqref{eq:err1a}, we know that the first term of the right-hand side is a $o(\eps)$. For the second term we use the result of Lemma \ref{lem:err1}.
\end{proof}
\begin{proof}[Proof of Corollary \ref{cor:error1}] 
  To prove \eqref{eq:cor1a} we need to estimate the norm of $w_\eps-\widetilde{w}_\eps$ in $V_L\coloneqq L^2(\Omega, H_1(\mathcal{B}_L^\omega))$. To do so we decompose the error in three parts
  \begin{multline}\label{eq:decomperr}
\left\| w_\eps-\widetilde{w}_\eps\right\|_{V_L}\leq \left\| u_0^{FF}+\eps u_1^{FF}-v_\eps\right\|_{V_L}
+\eps \left\| \partial_{x_d} (u_0^{FF}+\eps  u_1^{FF}- v_\eps)\big|_{\Sigma_{\eps H}}V_1\left(\frac{\cdot}{\eps}\right)\right\|_{V_L}\\ +\eps^2\left\|  \partial_{x_d} (u_0^{FF}- v_\eps)\big|_{\Sigma_{\eps H}}V_2\left(\frac{\cdot}{\eps}\right)\right\|_{V_L}.
\end{multline}
Proposition \ref{prop:effmod} allows us to estimate the first term of the right-hand-side. 
Let us now deal with the two remaining terms. Since $V_i, i=1,2$ verify \eqref{eq:firstrate} by Proposition \ref{prop:limNF} 
and $\partial_{x_d} u_i^{FF}\big|_{\Sigma_{\eps H}},\partial_{x_d} v_\eps^{FF}\big|_{\Sigma_{\eps H}}\in H^{s}(\Sigma_{\eps H})$ for any $s$, we get deriving inequalities similar to \eqref{eq:inter3bis} and \eqref{eq:inter3}, and from Lemma \ref{lem:err1}
\begin{multline}\label{eq:V1e}
\left\| \partial_{x_d} (u_0^{FF}+\eps  u_1^{FF}- v_\eps)\big|_{\Sigma_{\eps H}}V_1\left(\frac{\cdot}{\eps}\right)\right\|_{V_L}\lesssim \frac{1}{\eps}\|\partial_{x_d} (u_0^{FF}+\eps  u_1^{FF}- v_\eps)\big|_{\Sigma_{\eps H}}\|_{L^2(\Sigma_{\eps H})}\\
+\eta(\eps)\|\partial_{x_d} (u_0^{FF}+\eps  u_1^{FF}- v_\eps)\big|_{\Sigma_{\eps H}}\|_{H^1(\Sigma_{\eps H})},
\end{multline} 
and  
\begin{multline}\label{eq:V2e}
\left\| \partial_{x_d} (u_0^{FF}- v_\eps)\big|_{\Sigma_{\eps H}}V_2\left(\frac{\cdot}{\eps}\right)\right\|_{V_L}\lesssim \frac{1}{\eps}\|\partial_{x_d} (u_0^{FF}- v_\eps)\big|_{\Sigma_{\eps H}}\|_{L^2(\Sigma_{\eps H})}+\eta(\eps)\|\partial_{x_d} (u_0^{FF}- v_\eps)\big|_{\Sigma_{\eps H}}\|_{H^1(\Sigma_{\eps H})}
\end{multline}
where $\eta(\eps)$ tends to $0$ as $\eps$ tends to $0$. From \eqref{eq:effmod} we know moreover that
\begin{equation}\label{eq:difftrace}
\|\partial_{x_d} (u_0^{FF}+\eps  u_1^{FF}- v_\eps)\big|_{\Sigma_{\eps H}}\|_{H^1(\Sigma_{\eps H})}\lesssim \eps^2, \quad\text{and}\quad\|\partial_{x_d} (u_0^{FF}- v_\eps)\big|_{\Sigma_{\eps H}}\|_{H^1(\Sigma_{\eps H})}\lesssim \eps.
\end{equation} 
Inserting \eqref{eq:difftrace} into \eqref{eq:V1e} and \eqref{eq:V2e} allows us to conclude.
Finally, to prove \eqref{eq:cor1b} we combine \eqref{eq:err1b} and \eqref{eq:effmod}.
\end{proof}

\subsection{Improved estimates under a quantitative mixing assumption on $\Pa$}\label{subsection:quantest}
The limiting factor in the two estimates of Theorem \ref{thm:error} is that contrarily to the periodic case, we cannot without an additional quantitative mixing assumption on $\Pa$ quantify the rate of convergence of $V_1$ and $y_d^{2}\big| \nabla V_1\big|$ to $0$. Under such an assumption (Hypothesis \nameref{hyp:mix}) we prove in this subsection that for $R\geq1$
 \begin{equation}\label{eq:quanttrace}\E\left[\left|\fint_{\square_R} V_1^\omega\Big|_{\Sigma_L}\right|^2\right]=\E\left[\left|\E[W_1^{\omega}|_{\Sigma_L}]-\fint_{\square_R}W_1^{\omega}|_{\Sigma_L}\right|^2\right]\lesssim R^{-(d-1)},\end{equation}
 so that the results of proposition \ref{prop:U_comportement_inf_alea_quant} hold true.  More specifically we establish the following result
  \begin{theorem}[Fluctuations of $W_1^{\omega}|_{\Sigma_L}$] \label{thm:fluct} Let $L\geq \eps H$. Providing that $\Pa$ verifies \nameref{hyp:Linfty} and the mixing hypothesis \nameref{hyp:mix}, $W_1^{\omega}|_{\Sigma_L}$ verifies for all $f\in L^\infty(\Sigma_L)$ with compact support\begin{equation}\label{eq:fluct}\Var\left[\int_{\Sigma_L}f(\bsy_\shortparallel)W_1^\omega|_{\Sigma_L}(\bsy_\shortparallel)\dd \bsy_\shortparallel\right]\lesssim\|f\|^2_{L^2(\Sigma_L)}.\end{equation} 
 \end{theorem}
For $f\coloneqq R^{-\frac{d-1}{2}}\mathbb{1}_{\square_R}$ with $R\geq 1$, we directly retrieve \eqref{eq:quanttrace}. Consequently  the $L^2(\mathcal{B}^\omega_L)$- and $H^1(\mathcal{B}_{L',L})$- estimates of corollary \ref{cor:betterNFest} can be applied to the near-field $U^{NF}_1$ and we can then prove the following improved error estimates 
 \begin{theorem} \label{thm:errorimproved} Suppose that $\Pa$ verifies  \nameref{hyp:Linfty} and \nameref{hyp:mix} . For all $R>0$ and $L>L'>\eps H$,  the following estimates hold
\begin{equation}\label{eq:err1improved}
\left\| u_\eps-w_\eps\right\|^2_{L^2\big(\Omega,H^1(\mathcal{B}^{\omega}_{L})\big)}+\left\| u_\eps-(u_0^{FF}+\eps u_1^{FF})\right\|^2_{L^2\big(\Omega,H^1(\mathcal{B}_{L',L})\big)} \underset{\eps\to0}{=}\left\{\begin{array}{ll}\dst\mathcal{O}\left(\eps^{3}|\log\eps|\right)&\text{for } d=2,\\\dst\mathcal{O}(\eps^3)&\text{for } d=3,\end{array}\right.
\end{equation}
% and
% \begin{equation}\label{eq:err2improved}
%  \underset{\eps\to0}{=}\left\{\begin{array}{ll}\dst\mathcal{O}\left(\eps^{3}|\log\eps|\right)&\text{for } d=2,\\\dst\mathcal{O}(\eps^3)&\text{for } d=3,\end{array}\right.
% \end{equation}
\end{theorem}

\begin{corollary} \label{cor:error2improved}Suppose that $\Pa$ verifies  \nameref{hyp:Linfty} and \nameref{hyp:mix}. For all $R>0$ and $L>L'>\eps H$, the following estimates hold
\begin{equation*}
\left\| u_\eps-\widetilde{w}_\eps\right\|^2_{L^2\big(\Omega,H^1(\mathcal{B}^{\omega}_{L})\big)}+\left\| u_\eps-v_\eps\right\|^2_{L^2\big(\Omega,H^1(\mathcal{B}_{L',L})\big)} \underset{\eps\to0}{=}\left\{\begin{array}{ll}\dst\mathcal{O}\left(\eps^{3}|\log\eps|\right)&\text{for } d=2,\\\dst\mathcal{O}(\eps^3)&\text{for } d=3,\end{array}\right.
\end{equation*}
% and
% \begin{equation*}
%  \underset{\eps\to0}{=}\left\{\begin{array}{ll}\dst\mathcal{O}\left(\eps^{3}|\log\eps|\right)&\text{for } d=2,\\\dst\mathcal{O}(\eps^3)&\text{for } d=3,\end{array}\right.
% \end{equation*}
\end{corollary}

\begin{proof}[Proof of Theorem  \ref{thm:errorimproved}]
By Theorem \ref{thm:fluct}, under Hypotheses \nameref{hyp:mix} and \nameref{hyp:Linfty}, for all $R\geq1$
 \begin{equation*}
  \E\left[\left|\E[W_1^{\omega}|_{\Sigma_L}]-\fint_{\square_R}W_1^{\omega}|_{\Sigma_L}\right|^2\right]\lesssim R^{-(d-1)}.
\end{equation*} 
Consequently by Proposition \ref{prop:U_comportement_inf_alea_quant} we know that $V_1$, defined in \eqref{eq:U1v}, verifies 
\begin{equation*}
  y_d^{d-1}\E[|V_1|^2(\cdot,y_d)]+y_d^{d+1}\E[|\nabla V_1|^2(\cdot,y_d)]\underset{y_d\to+\infty}{=}\mathcal{O}(1).
\end{equation*} 
We then obtain for $L>\eps H$
\begin{equation}\label{eq:ingr1}
\|V_1|_{\Sigma_{{L}/{\eps}}}\|_{L^2(\Omega)}\underset{\eps\to0}{=}\mathcal{O}\left(\eps^{\frac{d-1}{2}}\right),\quad\text{and}\quad \|\partial_{y_d}V_1|_{\Sigma_{{L}/{\eps}}}\|_{L^2(\Omega)}\underset{\eps\to0}{=}\mathcal{O}\left(\eps^{\frac{d+1}{2}}\right).
\end{equation}
Inserting \eqref{eq:ingr1} and the result of Corollary \ref{cor:betterNFest} in the proof of Theorem \ref{thm:error} yields \eqref{eq:err1improved}.

% Combining the following estimate
% \begin{equation*}
% \left\|u(\cdot)U\left(\frac{\cdot}{\eps}\right)\right\|_{L^2(\Omega, H^1(\mathcal{B}_{L',L}))}\underset{\eps\to0}{=}\left\{\begin{array}{ll}\dst\mathcal{O}\left(\eps^{\frac{1}{2}}\log \left(\frac{1}{\eps}\right)^{\frac{1}{2}}\right)&\text{if } d=2,\\\dst\mathcal{O}(\eps)&\text{if } d=3,\end{array}\right.
% \end{equation*} of corollary \ref{cor:betterNFest} and \eqref{eq:err1improved} in the proof of Theorem \ref{thm:error} leads to \eqref{eq:err2improved}. 

\end{proof}

\begin{proof}[Proof of Corollary  \ref{cor:error2improved}]
 The proof is identical as the proof of Corollary \ref{cor:error1} but we take advantage here of the improved estimate of Corollary  \ref{cor:betterNFest}. %so that \eqref{eq:V1e} and \eqref{eq:V2e} are replaced by
% \begin{equation*}\begin{multlined}
% \left\| \partial_{x_d} (u_0^{FF}+\eps  u_1^{FF}- v_\eps)\big|_{\Sigma_{\eps H}}V_1\left(\frac{\cdot}{\eps}\right)\right\|_{L^2\big(\Omega,H^1(\mathcal{B}_{L})\big)}\lesssim \frac{1}{\eps}\|\partial_{x_d} (u_0^{FF}+\eps  u_1^{FF}- v_\eps)\big|_{\Sigma_{\eps H}}\|_{L^2(\Sigma_{\eps H})}\\+\eps^{\frac{3}{2}}\left(\|\nabla_\shortparallel\partial_{x_d} (u_0^{FF}+\eps  u_1^{FF}- v_\eps)\big|_{\Sigma_{\eps H}}\|_{L^2(\Sigma_{\eps H})}+\|\partial_{x_d} (u_0^{FF}+\eps  u_1^{FF}- v_\eps)\big|_{\Sigma_{\eps H}}\|_{L^2(\Sigma_{\eps H})}\right),\end{multlined}
% \end{equation*} and  \begin{equation*}\begin{multlined}
% \left\| \partial_{x_d} (u_0^{FF}- v_\eps)\big|_{\Sigma_{\eps H}}V_2\left(\frac{\cdot}{\eps}\right)\right\|_{L^2\big(\Omega,H^1(\mathcal{B}_{L})\big)}\lesssim \frac{1}{\eps}\|\partial_{x_d} (u_0^{FF}- v_\eps)\big|_{\Sigma_{\eps H}}\|_{L^2(\Sigma_{\eps H})}\\+\eps^{\frac{3}{2}}\left(\|\nabla_\shortparallel\partial_{x_d} (u_0^{FF}- v_\eps)\big|_{\Sigma_{\eps H}}\|_{L^2(\Sigma_{\eps H})}+\|\partial_{x_d} (u_0^{FF}- v_\eps)\big|_{\Sigma_{\eps H}}\|_{L^2(\Sigma_{\eps H})}\right).\end{multlined}
% \end{equation*}

 \end{proof}

 The rest of this section is dedicated to the proof of Theorem \ref{thm:fluct}.

% \begin{remark}
%Since $\nabla V_1(\cdot, y_d)=\mathcal{O}(y_d^{-\frac{d+1}{2}})$ under hypotheses \nameref{hyp:Linfty} and \nameref{hyp:mix}, \eqref{eq:w2'} is well-posed in $\mathcal{H}_0(D)$. \end{remark}

\subsubsection{The quantitative mixing assumption}

To prove Theorem \ref{thm:fluct} we need to assume that the point process associated to the distribution 
of particles $\Pa$ verifies a quantitative mixing condition. Since the dependency of $W_1^{\omega}$ in $\Pa$ 
is highly non-linear, we choose to write the hypothesis as a variance inequality on any function $F$ of $\Pa$. 
As proven in \cite{DuerinckxGloria2017}, most common hard-core point processes such as random parking or hard-core 
Poisson point processes verify the following variance inequality.

\begin{hypothesis}[Mix] \makeatletter\def\@currentlabelname{(Mix)}\makeatother\label{hyp:mix}
There exists a non-increasing weight function \mbox{ $\pi:\R^+\to\R^+$} with super-algebraic decay such that $\Pa$ verifies for all $\sigma(\Pa)$- measurable random variable $F(\Pa)$,
\begin{equation} \label{eq:decorr}
\Var\left[F(\Pa)\right]\leq\E\left[\int_1^{+\infty}\int_{\R^{d-1}} \left(\partial^{osc}_{\Pa, \mathcal{L}_{\ell}(\bsxpar)} F(\Pa)\right)^2 \dd\bsxpar \ell^{-(d-1)} \pi(\ell-1) d\ell\right],
\end{equation}
where $\mathcal{L}_\ell(\bsxpar)\coloneqq \dst\square_\ell(\bsxpar)\times(0,h)$ is the portion of the layer with width 
$\ell\geq 0$ centered at $\dst\left(\bsxpar, {h}/{2}\right)$ and the oscillation $\partial^{osc}_{\Pa, \mathcal{L}_\ell(\bsxpar)} F(\Pa)$
 of $F(\Pa)$ with respect to $\Pa$ on $\mathcal{L}_\ell(\bsxpar)$ is defined by: 
 \begin{multline}\label{eq:der_osc}
  \partial^{osc}_{\Pa, \mathcal{L}_{\ell}(\bsxpar)}F(\Pa)\coloneqq \textrm{sup ess}\left\{F(\Pa'), \Pa'\cap\left(\R^d\setminus \overline{\mathcal{L}_\ell(\bsxpar)}\right)=\Pa\cap\left(\R^d\setminus \overline{\mathcal{L}_\ell(\bsxpar)}\right)\right\}\\-\textrm{inf ess}\left\{F(\Pa'), \Pa'\cap\left(\R^d\setminus \overline{\mathcal{L}_\ell(\bsxpar)}\right)=\Pa\cap\left(\R^d\setminus \overline{\mathcal{L}_\ell(\bsxpar)}\right)\right\}.
\end{multline}
\end{hypothesis}
The structure of the proof follows the classical structure of proofs of quantitative stochastic homogenization estimates \cite{Gloria2019}. For a particularly pedagogical proof of the same nature but in a Gaussian setting we refer to \cite{josien2022annealed}. More precisely to prove Theorem \ref{thm:fluct} we adapted \cite{DuerinckxGloria2021} that study random suspensions of rigid particles in a steady Stokes flow to our setting of a thin layer of randomly distributed particles.

The starting point is to apply \eqref{eq:decorr} to 
$\int_{\Sigma_L}f(\bsy_\shortparallel)W_1^\omega|_{\Sigma_L}(\bsy_\shortparallel)\dd \bsy_\shortparallel$ 
for $f\in L^\infty(\Sigma_L)$ with compact support. A preliminary step of the proof 
consists then in bounding a.s. the oscillation of the integral. This step is done in 
Section \ref{subsubsec:osc}. The analysis requires  the introduction of a (well-posed) 
auxiliary problem and the control of high stochastic moments of $W_1$. These constitute 
the respective subjects of Sections \ref{subsubsec:aux} and \ref{subsubsec:moments}. 
Theorem \ref{thm:fluct} is then finally proven in Section \ref{subsubsec:proofthmfluct}. 

\subsubsection{Auxiliary problem}\label{subsubsec:aux}

Let 
\begin{equation}\label{eq:W0D}
W_0(D^\omega)\coloneqq \left\{v\in H_{loc}^1\big(D^\omega\big), v=0\,\text{on }\partial \Pa^\omega, \int_{\R^{d-1}}\int_0^{+\infty} \mathbb{1}_{D} \ (\mu|v|^2+|\nabla v|^2) \dd\bsy<+\infty\right\}.
\end{equation}
Equipped with the norm
\begin{equation*}
\forall v\in W_0(D^\omega), \quad \|v\|^2_{W_0(D^\omega)}\coloneqq  \int_{\R^{d-1}}\int_0^{+\infty} \mathbb{1}_{D} \ (\mu|v|^2+|\nabla v|^2) \dd\bsy,
\end{equation*}
the weighted Sobolev space $W_0(D^\omega)$ is a Hilbert space.
Recall that we proved in Proposition \ref{prop:wHardy} (see Remark \ref{rem:wHardyW0D}) that if $v\in W_0(D^\omega)$ for $\omega\in\Omega$, then
\begin{equation}\label{eq:wPoincareCC}
  \int_{D^\omega}\mu^\omega|v^\omega|^2\lesssim\int_{D^\omega}|\nabla v^\omega|^2,
\end{equation}
so that the $W_0(D^\omega)$-norm and the $H^1$ semi-norm are equivalent in $W_0(D^\omega)$. 

Consider now the following adjoint auxiliary problem
\begin{equation}\label{eq:uf}
\left\{
\begin{array}{l}
-\Delta u_{f,\bsg}^\omega =\nabla\cdot\bsg\; \text{in}\, D^\omega\setminus(\Sigma_H\cup\Sigma_L),  \\
  -\partial_{y_d} u_{f,\bsg}^\omega  = 0 \;\text{on} \,\Sigma_0, \quad\text{and}\quad 
  u_{f,\bsg}^\omega =0\;\text{on} \,\partial\Pa^\omega,\\
 \Big[u_{f,\bsg}^\omega \Big]_H  = 0 \quad \text{and}\quad\Big[ -\partial_{y_d} u_{f,\bsg}^\omega \Big]_H=0, \\ 
 \Big[u_{f,\bsg}^\omega \Big]_L  = 0 \quad \text{and}\quad \Big[ -\partial_{y_d} u_{f,\bsg}^\omega \Big]_L=f,
\end{array}
\right.
\end{equation} 
for $f\in L^\infty(\Sigma_L)$ and $\bsg\in L^\infty(\R^{d-1}\times\R^+)^d$ both with compact support. 
Problem \eqref{eq:uf} is well-posed a.s. in $W_0(D^\omega)$ as stated by the following proposition.

\begin{proposition}\label{prop:aux} 
  Let $f\in L^2(\Sigma_L)$ and $\bsg\in L^2(\R^{d-1}\times\R^+)^d$ with compact support. There exists a unique process $u_{f,\bsg}$ such that a.s. $u_{f,\bsg}^\omega\in W_0(D^\omega)$ is a weak solution of \eqref{eq:uf} and 
\begin{equation}\label{eq:energufg}
  \|\nabla u_{f,\bsg}\|_{L^2(D^\omega)}\lesssim\|\mu^{-\frac{1}{2}}f\|_{ L^2(\Sigma_L)}+\|\bsg\|_{L^2(D^\omega)^d}.
\end{equation} 
\end{proposition}

\begin{proof}[Proof of Proposition \ref{prop:aux}]
It suffices to use Lax-Milgram Theorem.
The variational formulation reads for
\begin{equation}\label{eq:WFug}
  \forall v^\omega\in W_0(D^\omega),\quad \int_{D^\omega}\nabla u_{f,\bsg}^\omega\cdot\nabla\overline{v}^\omega=\int_{D^\omega}\bsg\cdot\nabla\overline{v}^\omega+\int_{\Sigma_L}f\overline{v^\omega}.
\end{equation}
The coercivity of the sesquilinear form is ensured by \eqref{eq:wPoincareCC}.
Since $f\in L^2(\Sigma_L)$ has compact support, $\mu^{-\frac{1}{2}}f\in L^2(\Sigma_L)$ and
\begin{equation*}
  \left|\int_{\Sigma_L}f\overline{v}^\omega\right|\leq \|\mu^{-\frac{1}{2}}f\|_{L^{2}(\Sigma_L)}\| v^\omega\|_{W_0(D^\omega)}.
\end{equation*} 
% There exists then a unique weak solution in $W_0(D^\omega)$ to \eqref{eq:uf}. 
% that verifies the following energy estimate
% \begin{equation}\label{eq:energufg}
%   \|\nabla u_{f, \bsg}\|_{L^2(D^\omega)}\leq\|u_{f,\bsg}\|_{W_0(D^\omega)}\lesssim\|\mu^{-\frac{1}{2}}f\|_{L^{2}(\Sigma_L)}+\|\bsg\|_{L^2(D^\omega)^d}.\end{equation}
\end{proof}

\begin{theorem}[$L^p$-regularity of $u_{f,\bsg}$]\label{thm:aux} 
  Suppose that $\Pa$ verifies  \nameref{hyp:Linfty} and \nameref{hyp:mix}. There exists $C_0\simeq 1$ such that for all $f\in L^\infty(\Sigma_L)$ and $\bsg\in L^\infty(\R^{d-1}\times\R^+)^d$ with compact support, the solution $u_{f,\bsg}$ of the auxiliary problem \eqref{eq:uf} statisfies 
\begin{equation}\label{eq:momentuf}
\|\mathbb{1}_{D}\nabla u_{f,\bsg}\|_{L^p(\R^{d-1}\times\R^+, L^q(\Omega))}\lesssim\|f\|_{ L^p(\Sigma_L)}+\|\bsg\|_{ L^p(\R^{d-1}\times\R^+)^d},
\end{equation}
for all $|p-2|,|q-2|\leq \frac{1}{C_0}$.
 \end{theorem}

The proof of the theorem is postponed to Appendix \ref{sec:appaux}. 

\subsubsection{Higher stochastic moments of $\nabla W_1$}\label{subsubsec:moments}

The last ingredient we need is a control in $L^{q}(\Omega)$-norm of $\|\mathbb{1}_{D}\nabla W_1\|_{L^2(\square_1(\bsx_\shortparallel)\times(0,L))}$ for $L$ large enough.

\begin{theorem}[Stochastic moments of $\nabla W_1$]\label{thm:momentsnw1} 
Assume that $\Pa$ verifies hypotheses  \nameref{hyp:Linfty} and \nameref{hyp:mix}. 
For all $q\geq 2$, $R_1\geq1$, the following estimate holds for all $L$ large enough
\begin{equation}\label{eq:momentnw1}
\E\left[\left(\fint_{\square_{R_1}}\int_0^L\mathbb{1}_{D}|\nabla W_1|^2\right)^{\frac{q}{2}}\right]^{\frac{1}{q}}\lesssim 1,
\end{equation}
where the constant depends on $R_1$ and $L$.
\end{theorem}

To prove Theorem \ref{thm:momentsnw1} we adapt the proof of \cite[Theorem 4.2]{DuerinckxGloria2021} to our problem. The proof relies on two arguments that we state in the two propositions below. The first result is a control in the $L^q(\Omega)$- norm of the fluctuations of $\nabla W_1$.

\begin{proposition}[Fluctuations of $\nabla W_1$]\label{prop:fluctw1}
Providing that $\Pa$ verifies hypotheses \nameref{hyp:Linfty} and \nameref{hyp:mix}, $\nabla W_1^{\omega}$ verifies for $\bsg\in L^\infty(\R^{d-1}\times\R^+)^d$ with compact support, $q\gg 2$ and $R_1\geq1$
 \begin{equation}\label{eq:fluctw1}
  \E\left[\left|\int_{0}^{+\infty}\int_{\R^{d-1}}\mathbb{1}_{D^\omega}\bsg\cdot\nabla W_1^\omega\dd \bsy\right|^q\right]^{\frac{1}{q}}
  \lesssim\|\bsg\|_{L^2(\R^{d-1}\times\R^+)^d}\E\left[\left(\int_{\square_{R_1}}\int_0^L\mathbb{1}_{D^\omega}|\nabla W_1^\omega|^2\right)^{\frac{q}{2}}\right]^{\frac{1}{q}}.
\end{equation} 
\end{proposition}

 The second argument consists in bounding local norms of $\nabla W_1$ (such as the one appearing in the right-hand side of \eqref{eq:fluctw1}) with large scale averages of $\nabla W_1$.
\begin{proposition}\label{prop:localnorm} 
  Suppose that $\Pa$ verifies Hypotheses \nameref{hyp:Linfty} and \nameref{hyp:mix}. Consider ${\Lchi}\in C_c^\infty(\square_1\times(0,1))$ 
  such that $\int_{\square_1\times (0,1)}\Lchi=1$ and define for $r>0$, $\Lchi_r(\bsx)\coloneqq r^{-d}\Lchi\left(\frac{\bsx}{r}\right)$. 
For all $R\gg 1$, $L\gg 1$, $\alpha\in(0,1)$ and $q\geq 2$  
\begin{equation*}
\E\left[\left(\fint_{\square_R}\int_0^L\mathbb{1}_{D^\omega}|\nabla W_1^\omega|^2\right)^{\frac{q}{2}}\right]^{\frac{1}{q}}
\lesssim 1+\left(\int_0^{2L}\E\left[\left|\int_{\R^{d-1}}\int_0^{+\infty}\mathbb{1}_{D^\omega}\Lchi_{R^\alpha}(\bsypar, x_d-y_d)\nabla W_1^\omega(\bsy)\dd\bsy\right|^q\right]^{\frac{2}{q}}\dd x_d\right)^{\frac{1}{2}},
\end{equation*}
where the constant only depends on $\Lchi$.
\end{proposition}

We prove below Theorem \ref{thm:momentsnw1}. The proof of Propositions \ref{prop:fluctw1} and \ref{prop:localnorm} is postponed to Appendix \ref{sec:appnw1}.

\begin{proof}[Proof of Theorem \ref{thm:momentsnw1}]
We prove here \eqref{eq:momentnw1} for $q\gg 2$. By Jensen's inequality the result can then be extended to all $q\geq 2$. 
By stationarity of $W_1$ we have on one hand for all $R_1, R_2\geq 1$, $L\geq H$
\begin{align}\label{eq:buck1}
  &\E\left[\left(\fint_{\square_{R_1}}\int_0^L \mathbb{1}_D|\nabla W_1|^2\right)^{\frac{q}{2}}\right]^{\frac{1}{q}}
  %&=\E\left[\fint_{\square_{R_2}}\left(\fint_{\square_{R_1}(\bsxpar)}\int_0^L  \mathbb{1}_D|\nabla W_1|^2\right)^{\frac{q}{2}}\dd\bsxpar\right]
  =\E\left[\fint_{\square_{R_2}}\left(\fint_{\square_{R_1}(\bsxpar)}\int_0^L \mathbb{1}_D |\nabla W_1|^2\right)^{\frac{q-1}{2}}\left(\fint_{\square_{R_1}(\bsxpar)} \int_0^L \mathbb{1}_D|\nabla W_1|^2\right)^{\frac{1}{2}}\dd\bsxpar\right]^{\frac{1}{q}}
  \nonumber \\&\qquad\qquad\leq\left(\frac{R_1+R_2}{R_1}\right)^{\frac{(d-1)(q-1)}{2q}}\E\left[\left(\fint_{\square_{R_1+R_2}}\int_0^L  \mathbb{1}_D|\nabla W_1|^2\right)^{\frac{q-1}{2}}
    \fint_{\square_{R_2}}\left(\fint_{\square_{R_1}(\bsxpar)}\int_0^L  \mathbb{1}_D |\nabla W_1|^2\right)^{\frac{1}{2}}\dd\bsxpar\right]^{\frac{1}{q}}
   \nonumber \\&\qquad\qquad\leq\left(\frac{R_1+R_2}{R_1}\right)^{\frac{(d-1)(q-1)}{2q}}\frac{1}{R_2^{\frac{d-1}q}}\E\left[\left(\fint_{\square_{R_1+R_2}}\int_0^L  \mathbb{1}_D|\nabla W_1|^2\right)^{\frac{q-1}{2}}
    \left(\int_{\square_{R_2}}\fint_{\square_{R_1}(\bsxpar)}\int_0^L  \mathbb{1}_D |\nabla W_1|^2\right)^{\frac{1}{2}}\dd\bsxpar\right]^{\frac{1}{q}}
   \nonumber \\&\qquad\qquad\leq\left(\frac{R_1+R_2}{R_1}\right)^{\frac{(d-1)}{2}}\frac{1}{R_2^{\frac{d-1}q}}\E\left[\left(\fint_{\square_{R_1+R_2}} \int_0^L \mathbb{1}_D|\nabla W_1|^2\right)^{\frac{q}{2}}\right]^{\frac{1}{q}},
  \end{align}
  where we used Jensen's inequality since $x\mapsto \sqrt{x}$ is concave.
  On the other hand according first to Proposition \ref{prop:localnorm} for $R\gg 1$, $L\gg1$ and $q\gg2$ and second to Proposition \ref{prop:fluctw1} 
\begin{equation}\label{eq:buck2}
\begin{aligned}
    \E\left[\left(\fint_{\square_{R}}\int_0^{L} |\nabla W_1|^2\right)^{\frac{q}{2}}\right]^{\frac{1}{q}}
    &\lesssim1+\left(\int_0^{2L}\E\left[\left|\int_{\R^{d-1}}\int_0^{+\infty}\Lchi_{{R}^\alpha}(\bsy,x_d-y_d)\nabla W_1(\bsy)\dd\bsy\right|^q\right]^{\frac{2}{q}}\dd x_d\right)^{\frac{1}{2}}
    \\&\lesssim 1+\left(\int_0^{2L}\|\Lchi_{{R}^\alpha}(x_d-\cdot)\|^2_{L^2(\R^{d-1}\times\R^+)}\dd x_d\right)^{\frac{1}{2}}\E\left[\left(\int_{\square_{R_1}}\int_0^L |\nabla W_1|^2\right)^{\frac{q}{2}}\right]^{\frac{1}{q}}
    \\&\lesssim1+R^{-\frac{\alpha d}{2}}R_1^{\frac{d-1}2}\E\left[\left(\fint_{\square_{R_1}}\int_0^L |\nabla W_1|^2\right)^{\frac{q}{2}}\right]^{\frac{1}{q}}
    %\\&\leq 1+r^{-\frac{d}{2}}\E\left[\left(\int_{\square_{R_1}}\int_0^L |\nabla W_1|^2\right)^{\frac{q}{2}}\right]^{\frac{1}{q}}.  
  \end{aligned}
\end{equation}
Injecting \eqref{eq:buck2} with $R=R_1+R_2$ into the final inequality of \eqref{eq:buck1} yields
\begin{equation*}
  \E\left[\left(\fint_{\square_{R_1}} \int_0^L\mathbb{1}_D |\nabla W_1|^2\right)^{\frac{q}{2}}\right]^{\frac{1}{q}}\lesssim\left(\frac{R_1+R_2}{R_1}\right)^{\frac{d-1}{2}}\frac{1}{R_2^{\frac{d-1}{q}}}
  +\frac{(R_1+R_2)^{\frac{(1-\alpha)d-1}{2}}}{R_2^{\frac{d-1}{q}}}\E\left[\left(\fint_{\square_{R_1}}\int_0^L |\nabla W_1|^2\right)^{\frac{q}{2}}\right]^{\frac{1}{q}}
\end{equation*}
With $\alpha$ chosen so that $(1-\alpha)d<1$ and $R_2$ large enough, the second term of the right-hand side can be absorbed into the left hand-side and we get the result.
\end{proof}

\subsubsection{A.s. estimate of the oscillation} \label{subsubsec:osc}

We prove that the  oscillation of $\dst\int_{\Sigma_L}f(\bsy_\shortparallel)W_1^{\omega}|_{\Sigma_L}(\bsy_\shortparallel)\dd \bsy_\shortparallel$ 
in $\mathcal{L}_\ell(\bsx_\shortparallel)$ can be controlled by the $L^2$-norms of $\nabla W_1^\omega$ and $\nabla u_{f,\boldsymbol{0}}^\omega$ in a 
layer slightly larger (represented in orange on Figure \ref{fig:do}) than $\mathcal{L}_\ell(\bsx_\shortparallel)$ (represented in red on Figure \ref{fig:do}).
\begin{figure}\centering
\includegraphics[scale=1.2]{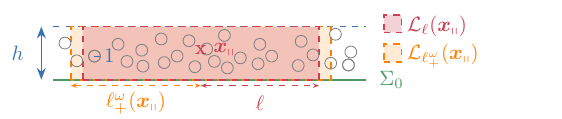}
\caption{Notations used in Section \ref{subsubsec:osc}}\label{fig:do}
\end{figure}
\begin{proposition}\label{prop:osc}
  Let $\bsxpar\in \R^{d-1}$ and $\ell\geq 1$. 
  For $f\in L^\infty(\Sigma_L)$ with compact support
\begin{equation*}
  \left|\partial^{osc}_{\Pa, \mathcal{L}_{\ell}(\bsxpar)}\int_{\Sigma_L}f(\bsy_\shortparallel)W_1^{\omega}|_{\Sigma_L}(\bsy_\shortparallel)\dd \bsy_\shortparallel\right|
  \lesssim(\ell^{\omega}_+(\bsxpar))^2\|\mathbb{1}_{D^\omega}\nabla u_{f}^\omega \|_{L^2(\mathcal{L}_{\ell^{\omega}_+}(\bsx_\shortparallel))}\|\mathbb{1}_{D^\omega}\nabla W_1^\omega\|_{L^2(\mathcal{L}_{\ell^{\omega}_+}(\bsx_\shortparallel))},
\end{equation*} 
where  $u_f^\omega\coloneqq u_{f,\boldsymbol{0}}^\omega$ and $\ell^{\omega}_+(\bsx_\shortparallel)$ is the smallest distance such that $\mathcal{L}_{\ell^{\omega}_+}(\bsx_\shortparallel)\setminus\overline{\mathcal{L}_\ell(\bsx_\shortparallel)}$ contains at least one particle (see Figure \ref{fig:do}), 
\textit{i.e.}
\begin{equation}\label{eq:dp}
\ell^{\omega}_+(\bsx_\shortparallel)\coloneqq \text{argmin}\left\{d>\ell, \quad\exists j\in\mathbb{N}, B(\bsx_j^\omega)\subset\mathcal{L}_{d}(\bsx_\shortparallel)\setminus\overline{\mathcal{L}_\ell(\bsx_\shortparallel)}\right\}.
\end{equation} 
\end{proposition}
%\begin{remark}By stationarity of $\Pa$ note that $d^+$ is also stationary.\end{remark}

\begin{proof}[Proof of proposition \ref{prop:osc}]
Let $\bsxpar\in \R^{d-1}$ and $\ell\geq 1$.
Consider $\Pa^\omega$ a realization of the particle distribution in the layer and let $\Pa'$ be a repartition of particles satisfying the hard-core assumption and such that
\begin{equation*}
  \Pa'\cap\left(\R^d\setminus \overline{\mathcal{L}_\ell(\bsxpar)}\right)=\Pa^\omega\cap\left(\R^d\setminus \overline{\mathcal{L}_\ell(\bsxpar)}\right).\end{equation*}
We let $W_1'$ denote the solution of problem \eqref{eq:w10} where $\Pa^\omega$ (resp. $D^\omega$) is replaced by $\Pa'$ (resp. $D'\coloneqq \R^{d-1}\times\R^+\setminus\overline{\Pa'}$). 

%By definition, estimating 
%\begin{equation*}\left|\partial^{osc}_{\Pa, \mathcal{L}_{\ell}(\bsxpar)}\int_{\Sigma_L}f(\bsy_\shortparallel)W_1^{\omega}|_{\Sigma_L}(\bsy_\shortparallel)\dd \bsy_\shortparallel\right|
%\end{equation*}for a given $f\in L^\infty(\Sigma_L)$ with compact support corresponds to bounding
%\begin{equation}\label{eq:termZ}\left|\int_{\Sigma_L}f(\bsy_\shortparallel)Z^\omega|_{\Sigma_L}(\bsy_\shortparallel)\dd \bsy_\shortparallel\right|.
%\end{equation}

\noindent\textit{Step 1:} We first prove that $W_1^{'}$ is well-defined and verifies
\begin{equation}\label{eq:claim1}
  \|\mathbb{1}_{D'}\nabla W_1^{'}\|_{L^2(\mathcal{L}_{\ell^{\omega}_+}(\bsx_\shortparallel))}\lesssim\ell^{\omega}_+(\bsx_\shortparallel)\,\|\mathbb{1}_{D^\omega}\nabla W_1^\omega\|_{L^2(\mathcal{L}_{\ell^{\omega}_+}(\bsx_\shortparallel))}.\end{equation}
Consider the set of particles of $\Pa^\omega$ (resp. $\Pa'$) intersecting with $\mathcal{L}_\ell(\bsx_\shortparallel)$, \textit{i.e.}
 \begin{gather*}
  \Pa_\ell^\omega\coloneqq\bigcup_{i\in\mathcal{I}_{\ell}^\omega(\bsx_\shortparallel)}B(\bsx_i^\omega),\,\, \text{with }\quad\mathcal{I}_\ell^\omega(\bsx_\shortparallel)\coloneqq \left\{i\in\mathbb{N}, \quad \bsx_i^\omega\in\mathcal{L}_\ell(\bsxpar)\right\},
  \\\left(\text{resp. }\,\,\Pa_\ell'\coloneqq\bigcup_{i\in\mathcal{I}_{\ell}'(\bsx_\shortparallel)}B(\bsx_i'),\,\, \text{with }\quad\mathcal{I}'_\ell(\bsx_\shortparallel)\coloneqq \left\{i\in\mathbb{N}, \quad \bsx'_i\in\mathcal{L}_\ell(\bsxpar)\right\}\right).
\end{gather*} 
Let $\mathcal{S}\coloneqq\text{Hull}\left(\Pa_\ell^\omega\right)$ where $\text{Hull}$ denotes the convex hull. Let $\Lchi$ be a smooth cut-off function such that $\text{supp} \Lchi \subset \mathcal{S}$
%  \begin{equation*}
%  \chi\equiv1 \,\,\text{in }\, \mathcal{S}_{\frac{\delta}{4}} \quad\text{and}\quad\text{supp} \chi \subset \mathcal{S}_{\frac{3\delta}{4}} ,
%  \end{equation*} where $\mathcal{S}_\alpha\coloneqq\left\{\bsx\in \R^d, \text{dist}(\bsx, \mathcal{S})=\alpha\right\}, \alpha>0$. 
then $X^\omega\coloneqq W_1'-\mathbb{1}_{\Pa_\ell^\omega} W_1^\omega(1-\chi)$ verifies the following problem in $D'$
\begin{equation}\label{eq:x}
\left\{
\begin{array}{l}
-\Delta X^\omega= \Delta (\chi W_1^\omega)  \;\text{ in } \quad D'\setminus\Sigma_H,  \\
  -\partial_{y_d} X^\omega = 0 \; \text{ on } \quad \Sigma_0, \quad\text{and}\quad 
 X^\omega = 0\; \text{ on } \quad \Pa', \\
    \Big[X^\omega\Big]_H  = 0 \quad  \text{ and }\quad \Big[ -\partial_{y_d}X^\omega\Big]_H = 0.
\end{array}
\right.
\end{equation}
Since $\nabla (\chi W_1)\in [L^2(D')]^d$ with compact support, there exists a unique $X^\omega\in W_0(D')$ by proposition \ref{prop:aux}. Hence $W_1'$ is well-defined. Let us now show that it verifies estimate \eqref{eq:claim1}.

Let $\widetilde{D}^\omega\coloneqq D^\omega\cup\overline{\Pa_\ell^\omega}.$
$Z^\omega\coloneqq \mathbb{1}_{\Pa_\ell'}W_1'- \mathbb{1}_{\Pa_\ell^\omega}W_1^\omega$  verifies
% \begin{equation}\label{eq:z}
% \left\{
% \begin{array}{cl}
% -\Delta Z^\omega= 0 &\text{ in } \quad \widetilde{D}^\omega\setminus \left(\Sigma_H\cup \partial\Pa_\ell^\omega\cup \partial\Pa_\ell'\right),  \\
%   -\partial_{y_d} Z^\omega = 0 & \text{ on } \quad \Sigma_0, \\
%  Z^\omega = 0& \text{ on } \quad \partial\Pa^\omega\setminus\partial\Pa^\omega_\ell, \\
%   \left[ Z^\omega\right]_{\partial\Pa_\ell^\omega}=0 \quad&\text{and }\left[-\nabla Z^\omega\cdot \bsn\right]_{\partial\Pa_\ell^\omega}=-\nabla W_1^\omega \cdot \bsn,\\
%   \left[ Z^\omega\right]_{\partial\Pa_\ell'}=0\quad&\text{and }  \left[-\nabla Z^\omega\cdot \bsn\right]_{\partial\Pa_\ell'}=-\nabla W_1^{'} \cdot \bsn,\\
%  \Big[Z^\omega\Big]_H  = 0 \quad  &\text{ and }\Big[ -\partial_{y_d}Z^\omega\Big]_H = 0.
% \end{array}
% \right.
% \end{equation}
% The weak formulation associated to \eqref{eq:z} reads for all $$
\begin{equation}\label{eq:weakz}
  \forall v\in W_0(\widetilde{D}^\omega),\quad \int_{\widetilde{D}^\omega} \nabla Z^\omega\cdot\nabla\overline{ v}=\int_{\partial\Pa_\ell^\omega}\nabla W_1^\omega\cdot\bsn \,\overline{v}+\int_{\partial\Pa_\ell'}\nabla W_1^{'}\cdot\bsn \,\overline{v}.
\end{equation} 
To simplify notations we use here integrals on the boundary of $\Pa_\ell^\omega$ and $\Pa_\ell'$ to denote the duality product $H^{-{1}/{2}}, H^{{1}/{2}}$. 
%Note that if a particle of $\Pa_\ell'$ intersects a particle of $\Pa_\ell^\omega$ then the integrals on the boundary of those particles stand for the duality product $\left(H^{\frac{1}{2}}_{00}\right)', H^{\frac{1}{2}}_{00}$.

The coercivity of the associated sesquilinear form is a direct consequence of the weighted Poincar\'e inequality \eqref{eq:wPoincareCC}. 
The continuity of the linear form is ensured by the trace theorem and a Poincare's inequality that holds with a constant proportional to $\ell^{\omega}_+(\bsxpar)$ since at least one particle lies in $\mathcal{L}_{\ell^{\omega}_+(\bsxpar)}(\bsx_\shortparallel)$. We obtain for all $i\in\mathcal{I}'_\ell(\bsx_\shortparallel), j\in\mathcal{I}_\ell^\omega(\bsx_\shortparallel)$, $v\in W_0(\widetilde{D}^\omega)$
\begin{equation}\label{eq:tracelayer}
  \|v\|_{H^{{1}/{2}}(\partial B(\bsx_j^\omega))}+\|v\|_{H^{{1}/{2}}(\partial B(\bsx_i'))}\lesssim \ell^{\omega}_+(\bsxpar)\,\|\nabla v\|_{L^2(\mathcal{L}_{\ell^{\omega}_{+}}(\bsx_\shortparallel))}.
\end{equation} 

Let us now estimate $\|\nabla W_1^\omega\cdot\bsn\|_{H^{-{1}/{2}}(\partial B(\bsx_i^\omega))}$ for $i\in\mathcal{I}_\ell^\omega(\bsx_\shortparallel)$. 
We proceed by duality. Let $\psi_i\in H^{{1}/{2}}(\partial B(\bsx_i^\omega))$ and $\Psi_i\in H^1(\dst\mathcal{L}_{\ell^{\omega}_+}(\bsx_\shortparallel))$ 
a lifting of $\psi$ such that $\textrm{supp} \Psi_i\subset \dst\mathcal{L}_{\ell^{\omega}_+}(\bsx_\shortparallel)$, $\Psi_i|_{\partial  B(\bsx_i^\omega)}=\psi_i\delta_{i,j}$ for all $j\in\mathcal{I}_\ell^\omega(\bsx_\shortparallel)$ and 
\begin{equation*}
  \|\Psi_i\|_{H^1(\dst\mathcal{L}_{\ell^{\omega}_+}(\bsx_\shortparallel))}\leq\|\psi_i\|_{H^{{1}/{2}}(\partial B(\bsx_i^\omega))}.
\end{equation*} 
We write the weak formulation associated to \eqref{eq:w10} in $\mathcal{L}_{\ell^{\omega}_{+}}(\bsx_\shortparallel)$
\begin{equation*}
  \int_{\mathcal{L}_{\ell^{\omega}_{+}}(\bsx_\shortparallel)}\mathbb{1}_{D^\omega}\nabla W_1^\omega\cdot\nabla \overline{\Psi_i}=\int_{\partial B(\bsx_i^\omega)}\nabla W_1^\omega\cdot\bsn\, \overline{\psi_i}.
\end{equation*}
By definition of the $H^{-1/2}$-norm we deduce that
\begin{equation*}
  \|\nabla W_1^\omega\cdot\bsn\|_{H^{-{1}/{2}}(\partial B(\bsx_i^\omega))}\leq\|\mathbb{1}_{D^\omega}\nabla W_1^\omega\|_{L^2(\mathcal{L}_{\ell^{\omega}_{+}}(\bsx_\shortparallel))}.
\end{equation*}
We can prove similarly for $i\in\mathcal{I}'_\ell(\bsx_\shortparallel)$, $\|\nabla W_1^{'}\cdot\bsn\|_{H^{-{1}/{2}}(\partial B(\bsx_i'))}\leq\|\mathbb{1}_{D'}\nabla W_1'\|_{L^2(\mathcal{L}_{\ell^{\omega}_{+}}(\bsx_\shortparallel))}.$

Testing \eqref{eq:weakz} with $v=Z^\omega$, we get
\begin{equation*}
  \int_{\widetilde{D}^\omega} |\nabla Z^\omega|^2=-\int_{\partial\Pa_\ell^\omega}\nabla W_1^\omega\cdot\bsn\, \mathbb{1}_{D'}\overline{W_1^{'}}
  -\int_{\partial\Pa_\ell'}\nabla W_1^{'}\cdot\bsn\,  \mathbb{1}_{D^\omega}\overline{W_1^\omega},
\end{equation*} 
and therefore
\begin{equation*}
  \int_{\widetilde{D}^\omega}|\nabla Z^\omega|^2\lesssim \ell^{\omega}_{+}(\bsxpar)\| \mathbb{1}_{D^\omega}\nabla W_1^\omega\|_{L^2(\mathcal{L}_{\ell^{\omega}_{+}}(\bsx_\shortparallel))}\| \mathbb{1}_{D'}\nabla W_1^{'}\|_{L^2(\mathcal{L}_{\ell^{\omega}_{+}}(\bsx_\shortparallel))}.
\end{equation*}
Claim \eqref{eq:claim1} follows from Young's inequality.\vsd

\noindent\textit{Step 2: Sensitivity of $W_1|_{\Sigma_L}$} \vsd

Recall that $u_f\coloneqq u_{f,\boldsymbol{0}}$ denotes the unique weak solution in $W_0(D^\omega)$ of \eqref{eq:uf} with $\bsg=\boldsymbol{0}$. The weak formulation associated to \eqref{eq:uf} reads for all $v\in W_0(\widetilde{D}^\omega)$
\begin{equation}\label{eq:weakug}
  \int_{\widetilde{D}^\omega} \mathbb{1}_{D^\omega}\nabla u_{f}^\omega\cdot\nabla\overline{ v}=\int_{\Sigma_L}f\overline{v}+\int_{\partial\Pa_\ell^\omega}\nabla u_{f}^\omega\cdot\bsn\, \overline{v}.
\end{equation} 
We evaluate \eqref{eq:weakz} for $v=\mathbb{1}_{D^\omega}\overline{u_{f}^\omega}$ and \eqref{eq:weakug} for $v=\overline{Z^\omega}.$ After subtracting both expressions and since $u_f^\omega=W^\omega_1=0$ on $\partial\Pa^\omega$, we obtain 
\begin{equation*}
  \int_{\Sigma_L}f Z^\omega=-\int_{\partial\Pa_\ell^\omega}\nabla u_{f}^\omega\cdot\bsn \,\mathbb{1}_{D'}W_1^{'}-\int_{\partial\Pa_\ell'}\nabla W_1^{'}\cdot\bsn\, \mathbb{1}_{D^\omega}u_{f}^\omega.
\end{equation*}
Next we can estimate $\|\nabla u_{f}^\omega\cdot\bsn\|_{H^{-{1}/{2}}(\partial B(\bsx_i^\omega))}$ for $i\in\mathcal{I}_{\ell}(\bsx_\shortparallel)$ using similar steps as in the estimate of $\|\nabla W_1^\omega\cdot\bsn\|_{H^{-{1}/{2}}(\partial B(\bsx_j^\omega))}$. We deduce that

%\noindent Let $\psi\in H^{\frac{1}{2}}(\partial B(\bsx_i^\omega))$. We introduce $\Psi\in W_0(\dst\mathcal{L}_{\ell+d^{\omega,+}(\bsx_\shortparallel)}(\bsx_\shortparallel))$ such that $\textrm{supp} \Psi\subset \dst\mathcal{L}_{\ell+d^{\omega,+}(\bsx_\shortparallel)}(\bsx_\shortparallel)$ and $\Psi|_{\partial  B(\bsx_i^\omega)}=\psi$. By definition, $\Psi$ verifies \begin{equation*}\|\Psi\|_{W_0(\mathcal{L}_{\ell+d^{\omega,+}(\bsx_\shortparallel)}(\bsx_\shortparallel))}\leq\|\psi\|_{H^{\frac{1}{2}}(\partial B(\bsx_i^\omega))}.\end{equation*}We write the weak formulation associated to \eqref{eq:uf} in $\mathcal{L}_{\ell+d^{\omega,+}(\bsx_\shortparallel)}(\bsx_\shortparallel)$
%\begin{equation*}\int_{\mathcal{L}_{\ell+d^{\omega,+}(\bsx_\shortparallel)}(\bsx_\shortparallel)}\mathbb{1}_{D^\omega}\nabla u_{f}\cdot\nabla \overline{\Psi}=
%\int_{\partial B(\bsx_i^\omega)}\nabla u_{f}\cdot\bsn \overline{\psi}\end{equation*}
%By duality we deduce that
\begin{equation*}
  \|\nabla u_{f}\cdot\bsn\|_{H^{-{1}/{2}}(\partial B(\bsx_i^\omega))}\leq\|\mathbb{1}_{D^\omega}\nabla u_{f}^\omega\|_{L^2(\mathcal{L}_{\ell^{\omega}_{+}}(\bsx_\shortparallel))}.
\end{equation*}
Therefore
\begin{equation*}
  \left|\int_{\Sigma_L}f (\mathbb{1}_{\Pa_\ell'}W_1'- \mathbb{1}_{\Pa_\ell^\omega}W_1^\omega)\right|\lesssim \ell^{\omega}_{+}(\bsxpar)\,\|\mathbb{1}_{D^\omega}\nabla u_{f}^\omega \|_{L^2(\mathcal{L}_{\ell^{\omega}_{+}}(\bsx_\shortparallel))}\,\|\mathbb{1}_{D'}\nabla W_1^{'}\|_{L^2(\mathcal{L}_{\ell^{\omega}_{+}}(\bsx_\shortparallel))}.
\end{equation*}
Using \eqref{eq:claim1} we deduce that
\begin{equation*}
  \left|\int_{\Sigma_L}f (\mathbb{1}_{\Pa_\ell'}W_1'- \mathbb{1}_{\Pa_\ell^{''}}W_1^{''})\right|\lesssim \ell^{\omega}_{+}(\bsxpar)\,\|\mathbb{1}_{D^\omega}\nabla u_{f}^\omega \|_{L^2(\mathcal{L}_{\ell^{\omega}_{+}}(\bsx_\shortparallel))}\,\|\mathbb{1}_{D^\omega}\nabla W_1^{\omega}\|_{L^2(\mathcal{L}_{\ell^{\omega}_{+}}(\bsx_\shortparallel))}.
\end{equation*}
where $\Pa_\ell''$ satisfies the same assumption than $\Pa_\ell'$ and $W_1''$ is the associated solution. By definition of \eqref{eq:der_osc}, this allows us to conclude.
\end{proof}

\subsubsection{Proof of Theorem \ref{thm:fluct}}\label{subsubsec:proofthmfluct}

Let $f\in L^\infty(\Sigma_L)$ with compact support.
The first step consists in applying the mixing hypothesis (Hypothesis \nameref{hyp:mix}) to $\dst\int_{\Sigma_L}f(\bsypar)W_1^\omega|_{\Sigma_L}(\bsypar)\dd \bsypar$ and use Proposition \ref{prop:osc}.
\begin{align*}
  &\Var\left[\int_{\Sigma_L}f(\bsypar)W_1^\omega|_{\Sigma_L}(\bsypar)\dd \bsypar\right]
  \leq \E\left[\int_1^{+\infty}\int_{\R^{d-1}} \left|\partial^{osc}_{\Pa, \mathcal{L}_{\ell}(\bsxpar)} \int_{\Sigma_L}f(\bsypar)W_1^\omega|_{\Sigma_L}(\bsypar)\dd \bsypar\right|^2 \dd\bsxpar \ell^{-(d-1)} \pi(\ell-1) \dd\ell\right]
  \\
  &\qquad\qquad\lesssim\E\left[\int_1^{+\infty}\int_{\R^{d-1}}(\ell^{\omega}_+(\bsxpar))^4 \|\mathbb{1}_{D^\omega}\nabla u_{f}^\omega \|_{L^2(\mathcal{L}_{\ell^{\omega}_{+}}(\bsx_\shortparallel))}^2\|\mathbb{1}_{D^\omega}\nabla W_1^\omega\|_{L^2(\mathcal{L}_{\ell^{\omega}_{+}}(\bsx_\shortparallel))}^2\dd\bsxpar 
    \ell^{-(d-1)} \pi(\ell-1) \dd\ell\right]
  \end{align*}
Next we use Hypothesis \nameref{hyp:Linfty} to bound a.s. and for a.e. $\bsxpar\in R^{d-1}$, $\ell_{+}^{\omega}(\bsxpar)$ by $\ell+R^+$.
\begin{multline*}
  \Var\left[\int_{\Sigma_L}f(\bsy_\shortparallel)W_1^\omega|_{\Sigma_L}(\bsy_\shortparallel)\dd \bsy_\shortparallel\right]\lesssim\int_1^{+\infty}\int_{\R^{d-1}}(\ell+R^+)^{2(d+1)}\E\left[ \left(\fint_{\square_{\ell+R^+}(\bsxpar)}\int_0^h\mathbb{1}_{D^\omega}|\nabla u_{f}^\omega|^2\right)\right.
  \\\left.\left(\fint_{\square_{\ell+R^+}(\bsxpar)}\int_0^h\mathbb{1}_{D^\omega}|\nabla W_1^\omega|^2\right)\right]\dd\bsxpar \ell^{-(d-1)} \pi(\ell-1) \dd\ell
\end{multline*}
 By H\"older's inequality, we get for $q$ such that $|{q}-1|\leq \frac{1}{2C_0}$ where $C_0$ is defined in Theorem \ref{thm:aux} and $q'\coloneqq q/(q-1)$
\begin{multline*}
  \Var\left[\int_{\Sigma_L}f(\bsy_\shortparallel)W_1^\omega|_{\Sigma_L}(\bsy_\shortparallel)\dd \bsy_\shortparallel\right]\lesssim\int_1^{+\infty}\int_{\R^{d-1}}(\ell+R^+)^{2(d+1)}\E\left[ \left(\fint_{\square_{\ell+R^+}(\bsxpar)}\int_0^h\mathbb{1}_{D^\omega}|\nabla u_{f}^\omega|^2\right)^{q}\right]^{{1}/{q}}
  \\\E\left[\left(\fint_{\square_{\ell+R^+}(\bsxpar)}\int_0^h\mathbb{1}_{D^\omega}|\nabla W_1^\omega|^2\right)^{q'}\right]^{{1}/{q'}}\dd\bsxpar \ell^{-(d-1)} \pi(\ell-1) \dd\ell.
\end{multline*}
We use  the stationarity of $W_1$ and Theorem \ref{thm:momentsnw1} with $R_1=\ell+R^+$ to deal with the near-field term
\begin{equation}\label{eq:1f}
\E\left[\left(\fint_{\square_{\ell+R^+}(\bsxpar)}\int_{0}^{h}\mathbb{1}_{D^\omega}|\nabla W_1^{\omega}|^2\right)^{q'}\right]^{{1}/{q'}}=\E\left[\left(\fint_{\square_{\ell+R^+}(0)}\int_{0}^{h}\mathbb{1}_{D^\omega}|\nabla W_1^{\omega}|^2\right)^{q'}\right]^{{1}/{q'}}
\lesssim 1.
\end{equation} 
Since $X\mapsto\E[|X|^{{q}}]^{{1}/{{q}}}$ is a norm we get by Jensen's inequality
\begin{equation}\begin{aligned}\label{eq:2f}
  \int_{\R^{d-1}}\E\left[\left( \fint_{\square_{\ell+R^+}(\bsxpar)}\int_{0}^{h}\mathbb{1}_{D^\omega}|\nabla u_{f}^\omega|^2\right)^{q}\right]^{{1}/{q}}\dd\bsxpar
  &\leq\int_{\R^{d-1}} \fint_{\square_{\ell+R^+}(\bsxpar)}\int_{0}^{h}\E\left[\mathbb{1}_{D^\omega}|\nabla u_{f}^\omega|^{2q}\right]^{{1}/{q}}\\
  &=\int_{\R^{d-1}} \int_{0}^{h}\E\left[\mathbb{1}_{D^\omega}|\nabla u_{f}^\omega|^{2q}\right]^{{1}/{q}}\\&=\|\mathbb{1}_D\nabla u_f^\omega\|_{L^2(\R^{d-1}\times \R^+,L^{2 q}(\Omega))}^2. \end{aligned}
\end{equation}
On the penultimate row we simplified the spatial average in the second term of the right-hand side by noticing that by Fubini-Tonelli's Theorem for $h\in L^1(\R^{d-1})$
\begin{equation*}
  \int_{\R^{d-1}}\fint_{\square_{\ell+R^+}(\bsxpar)}h(\bsypar)\dd\bsypar\dd\bsxpar=\int_{\R^{d-1}}h(\bsxpar)\dd\bsxpar.
\end{equation*} 
We conclude with the regularity result on the adjoint problem (see Theorem \ref{thm:aux}). Finally since $\pi$ is super-algebraic the integral over $\ell$ is finite.
\begin{equation}\label{eq:4f}
 \int_1^{+\infty}(\ell+R^{+})^{2(d+1)}\ell^{-(d-1)}\pi(\ell-1) \dd\ell<+\infty.
\end{equation}
Combining \eqref{eq:1f}, \eqref{eq:2f}, \eqref{eq:4f} yields the desired result.

\section{Numerical results}

In this section we present numerical simulations that illustrate our theoretical results and more specifically the error estimates derived in Corollary \ref{cor:error2improved}. All simulations were conducted using XLife$++$ an open source FEM-BEM solver \cite{kielbasiewicz2017xlife++}.

Instead of considering a compactly supported source we study the scattering by an incident plane wave $u_{inc}:=e^{i(k_1x_1+k_2x_2)}$. We solve then the following problem

\begin{equation*}
\left\{
\begin{aligned}
-\Delta u_{\eps} - k^2 u_\eps &= 0 \quad \text{ in }\,\,\mathcal{B}_{\eps H,L} , \\
-\partial_{x_d} u_\eps  +ik\gamma u_\eps &= 0 \quad \text{ on }\,\, \Sigma_{\eps H} , \\
-\partial_{x_d}  u_\eps + \Lambda^k  u_\eps &= -\partial_{x_d}  u_{inc} + \Lambda^k  u_{inc} \quad \text{ on } \quad \Sigma_{L} ,
\end{aligned}
\right.
\end{equation*}

Since the problem satisfied by $u_\eps$ is unbounded in the longitudinal directions, we approximate it by truncating the layer
$L_T:=\square_T\times(0,\eps h)$ with $T=400$ and impose that the field, {still denoted $u_\eps$}, is $k_1 T$-quasi-periodic, 
\emph{i.e.} $x_1\mapsto e^{-k_1x_1} u_\eps(x_1,x_2)$ is $T$-periodic. 
Thanks to the quasi-periodicity the Dirichlet-to-Neumann operator on $\Sigma_L^T:=\square_T\times\{y_d=L\}\}$ admits the following modal decomposition
\begin{equation*}
\forall v\in H^{{1}/{2}}(\Sigma_L^T),\quad \Lambda^k v(\cdot)=i\sum_{m\in\mathbb{Z}}\beta_m\langle v,\varphi_m\rangle_{L^2(\Sigma_L^T)}\varphi_m(\cdot),
\end{equation*}
where $\beta_m^2\coloneqq k^2 - ({2m\pi}/{T}+k_1)^2$ and  $\varphi_m(x_1)\coloneqq{1}/{\sqrt{T}}e^{i\left({2m\pi}/{T}+k_1\right)x_1}$ for $x_1\in\R$. 
To implement this operator we truncate the series to ${-N}\leq m\leq N$ where $N>0$ is chosen such that
\begin{equation*}
  e^{-\sqrt{\left({2N\pi}/{T}+k_1\right)^2-k^2}(L-\eps H)}<\eta
\end{equation*} 
with $\eta$ being a small threshold parameter ($\eta=10^{-6}$ in the simulations).
The centers of the particles are sampled according to a Mat\`ern point process \cite[Section 6.5.2]{illian2008statistical} restricted to the truncated layer $L_T$.
 The algorithm follows the following steps : 
 \begin{enumerate}
 \item for a given initial density $\rho  \in (0,1)$ we sample uniformly in $L_T$, $N_p$ centers where $N_p$ 
 is drawn according to a Poisson distribution of parameter $\nu^\eps\coloneqq\rho {T \eps h}/{\pi\eps^2};$
 \item we assign to each center independently and randomly a score between 0 and 1;
 \item we remove the centers violating the hard-core assumption with the centers with a lower score.
 \end{enumerate}    
  On Figure \ref{fig:ue},  the real part of the total field $u_\eps^\omega$ for a given realization $\omega$ is displayed.
 \begin{figure}
 \begin{center}\begin{tabular}[t]{cc}
\includegraphics [width=0.85\textwidth] {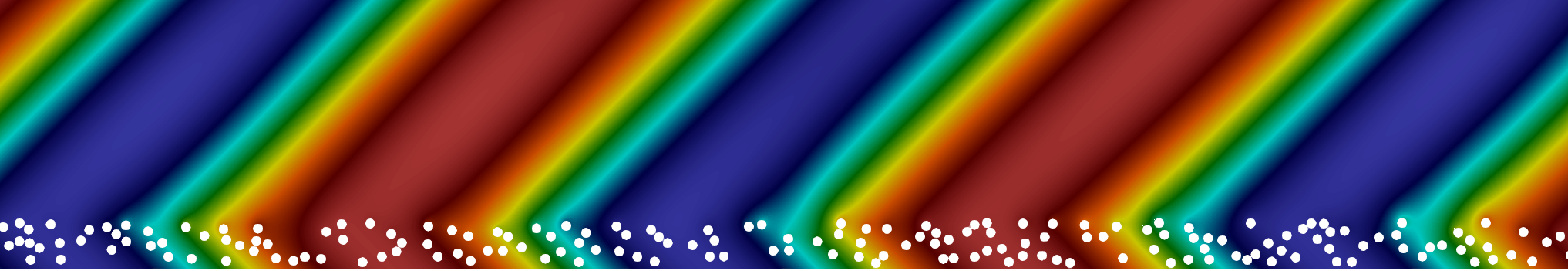}&\includegraphics [width=0.05\textwidth] {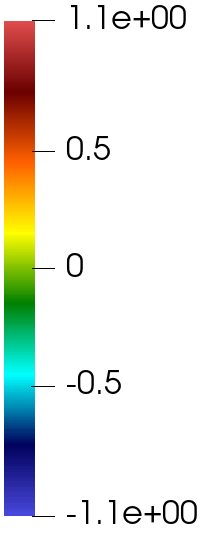}\end{tabular}
 \end{center}
 \caption{Real part of the total field $u_\eps$ for $\rho=0.4$, frequency$=2$GHz, $\theta=\pi/4$, $k_2\varepsilon=10^{-1}$m, $\gamma=1+i$}\label{fig:ue}
 \end{figure}

To implement the effective model \eqref{eq:modele_eff_Dirichlet_Alea} the first step is to compute $c_1$. 
By stationarity of $W_ 1$ for all $R>0$, it holds that 
\begin{equation*}
  c_1= \mathbb{E}[W_1\Big|_{\Sigma_L}]=\mathbb{E}\left[\fint_{\square_R}W_1\Big|_{\Sigma_L}(\bsypar) \dd\bsy_\shortparallel \right].
\end{equation*}
 To compute the expectation we use a Monte-Carlo algorithm so that
 \begin{equation}\label{eq:c1num}c_1\sim \frac{1}{N}\sum_{j=1}^N\left[\displaystyle\fint_{\square_R}W_1^{\omega_j}\Big|_{\Sigma_L}(\bsy_\shortparallel)\, \dd\bsy_\shortparallel \right].
 \end{equation}
$W_1^\omega$ is solution of a Laplace-type problem in $\mathcal{B}_L^\omega$, an unbounded random domain. To compute it numerically,
we restrict the domain to $\vsu\square_R\times [0,L]\setminus\Pa^\omega$ and impose periodic boundary conditions at $x_1=-{R}/{2}, {R}/{2}$ as it is customary in stochastic homogenization 
(see \cite{clozeau2024bias}). Since $y_1\to W_1^\omega(y_1, y_2)$ is $R$-periodic, the DtN operator on $\Sigma_L^\sharp\coloneqq\{(y_1,L), y_1\in(-{R}/{2},{R}/{2})\}$ 
admits the following modal decomposition
\begin{equation*}
\forall v\in H_\sharp^{{1}/{2}}(\Sigma_L^\sharp),\quad \Lambda v(\cdot)=\sum_{m\in\mathbb{Z}}\frac{2|m|\pi}{R}\langle v,\phi_m\rangle_{L^2(\Sigma_L^\sharp)}\phi_m(\cdot),
\end{equation*}
where $\phi_m(y_1)\coloneqq{1}/{\sqrt{R}}e^{\imath {2m\pi y_1}/{R}}$ for $y_1\in\R$.To implement this operator we truncate the series to ${-N}\leq m\leq N$ where $N>0$ is chosen such that
\begin{equation*}
  e^{-{2N\pi (L-H)}/{R}}<\eta.
\end{equation*}The centers of the particles are sampled according to a Mat\`ern point process restricted to the truncated layer $\square_R\times (0,h)$.
Keep in mind that contrarily to the simulation of $u_\eps$, the particles are now of size $1$ so that the parameter of the Poisson process for a given $\rho$ is $\nu\coloneqq\rho{R h}/{\pi}.$ On Figure \ref{fig:w1} two plots of $W_1^\omega$ for two different realizations and two different $\rho$ are displayed. We see that the density of particles plays a role in the convergence rate of $W_1^\omega$ when $y_2\to +\infty.$

\begin{figure}
\begin{center}\begin{tabular}{cccc}\includegraphics [width=0.4\textwidth] {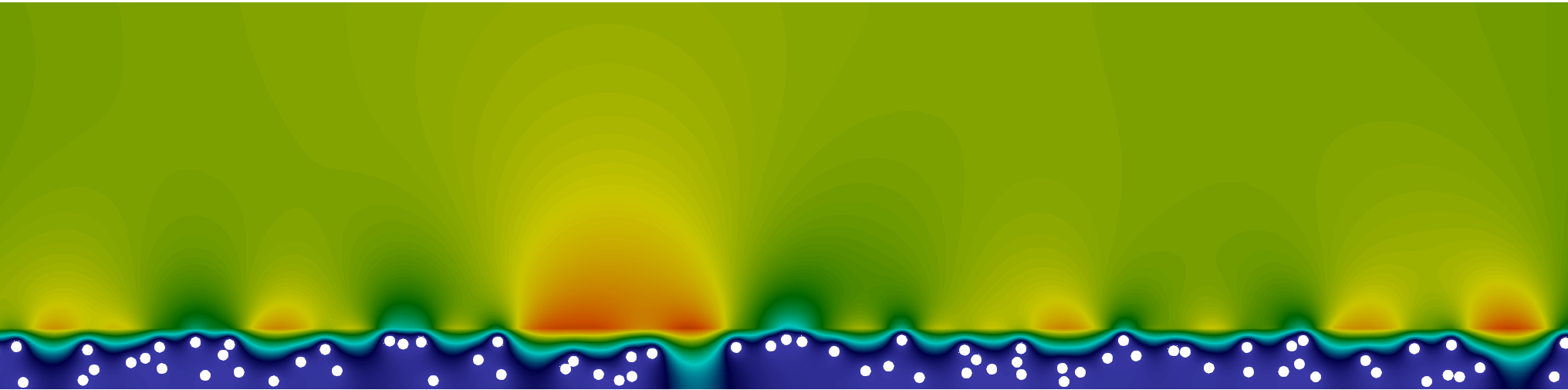}&\includegraphics [width=0.04\textwidth] {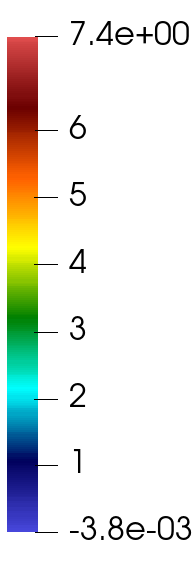}&\includegraphics [width=0.4\textwidth] {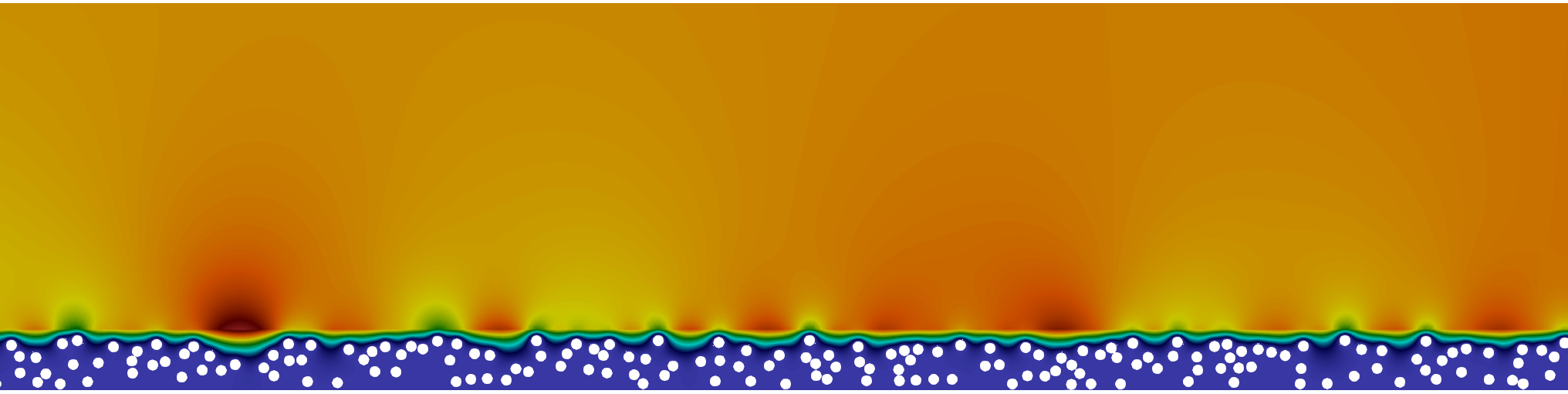}&\includegraphics [width=0.04\textwidth] {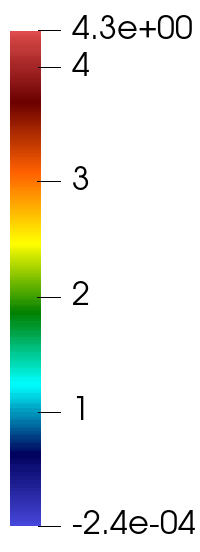}\end{tabular}\\
\caption{Profile function $W_1$ for $\rho=0.1$ (left) and $\rho=0.4$ (right)}\label{fig:w1}\end{center}\end{figure}

\begin{figure}
\begin{center}\begin{tabular}{cccc}\includegraphics [width=0.3\textwidth] {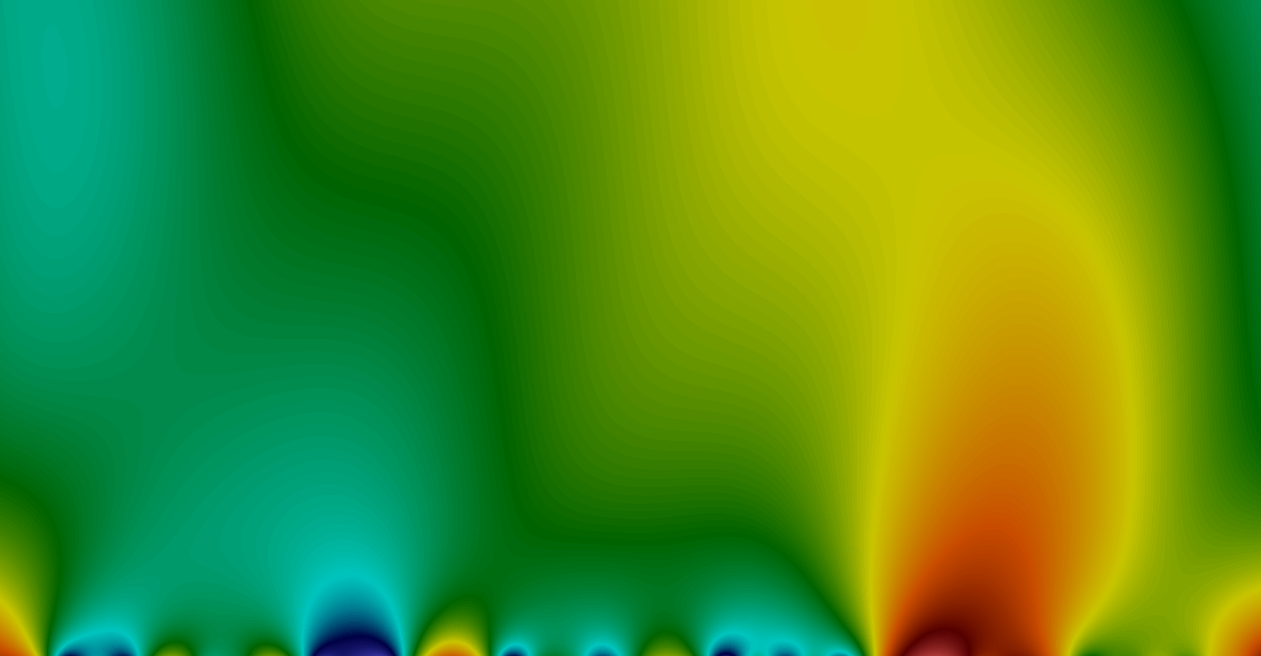}&\includegraphics [width=0.04\textwidth] {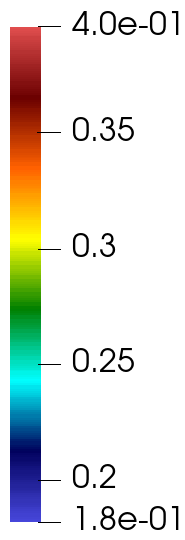}&\hspace{1cm}\includegraphics [width=0.3\textwidth] {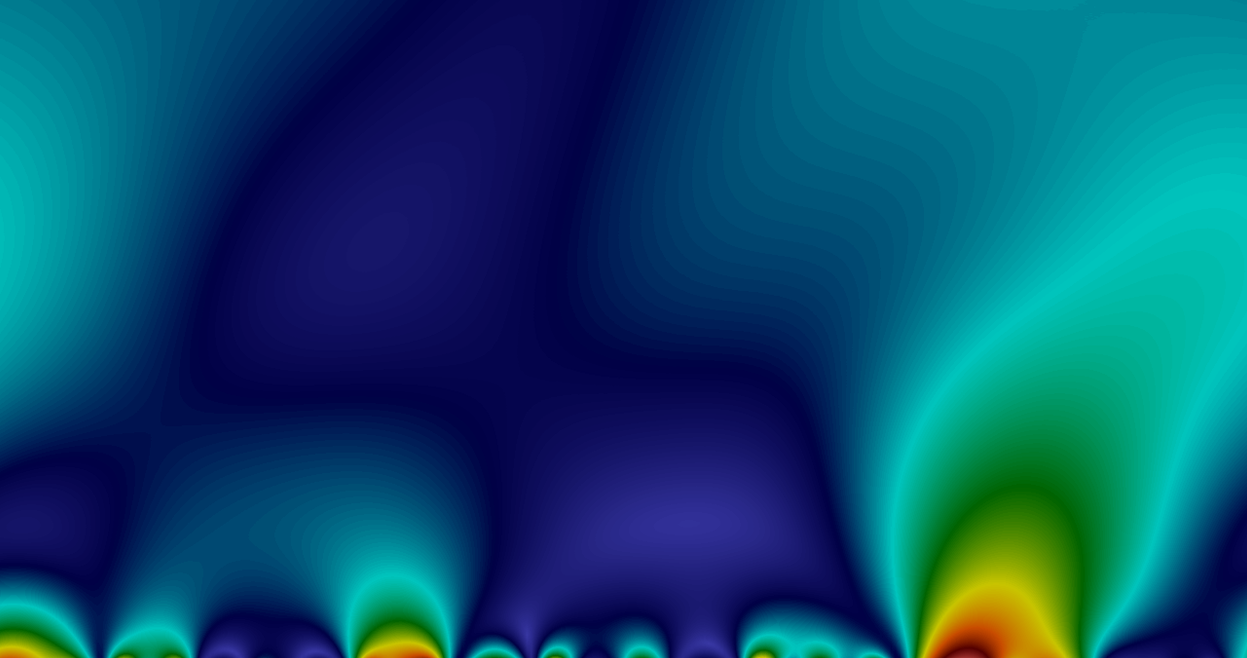}&\includegraphics [width=0.04\textwidth] {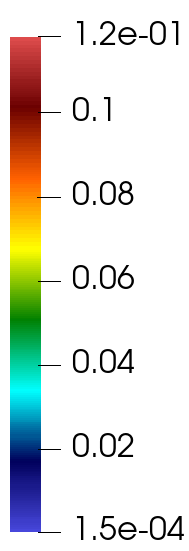}\end{tabular}\\
\caption{$|u_\eps^\omega-u_0^{FF}|$ (left) and $|u_\eps^\omega-v_\eps|$ (right) on $(\eps H, \eps L)$ for a given $\omega$ and $L=70$}\label{fig:err}\end{center}\end{figure}

In \eqref{eq:c1num} we can adjust the size of the domain $R$ and the number of realizations $N$ to improve the rate of convergence. 
On Figure \ref{fig:c1} on the left, we have plotted the value of the computed constant for different sizes $R$ of the computational domain. 
In green, the average is computed over 100 realizations and in red over 500 realizations. The two averages are displayed as the dotted lines 
and the colored zones correspond to the $95\%$-confidence interval. On Figure \ref{fig:c1} on the right, the constant is computed with one realization for three different realizations thanks to the ergodicity. We notice that the computed coefficient does converge to the same limit as the ones computed with Monte-Carlo algorithm but as expected it requires a much larger domain to achieve convergence ($3200$ vs $1700$).

 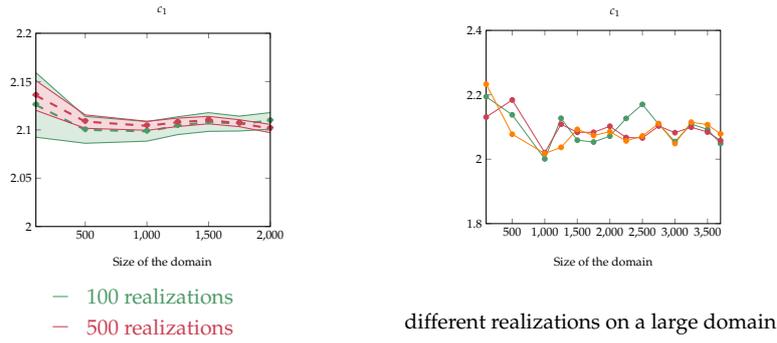
\begin{figure}
 \begin{center}
\begin{center}\begin{tabular}{cc}

   \begin{tikzpicture}[scale=0.45,transform shape]
   \begin{axis}[axis on top, xlabel={Size of the domain }, ymin=2.0,ymax=2.2, yticklabel style={/pgf/number format/fixed, /pgf/number format/precision=5}, scaled y ticks=false, title={}, enlarge x limits=false,title={$c_1$}, title style={at={(1.2,2.5)},anchor=north},  xlabel style={at={(1.2,-0.3)},anchor=south}]]  
  
   %% M 100 %%
   \addplot[mark=,green3,name path=B,forget plot]table[x=T,y=ETM]{graphes/M100_TE.txt};
   \addplot [mark=,green3,name path=A,forget plot]table[x=T,y=ETP]{graphes/M100_TE.txt};
   \addplot [mark=*,green3,dashed,thick]table[x=T,y=MOYENNE]{graphes/M100_TE.txt};

   \addplot[green3!20] fill between[of=A and B];

   %% M 500  %%
    \addplot[mark=,color3,name path=B,forget plot]table[x=T,y=ETM]{graphes/M500_TE.txt};
   \addplot [mark=,color3,name path=A,forget plot]table[x=T,y=ETP]{graphes/M500_TE.txt};
   \addplot [mark=*,color3,dashed,thick]table[x=T,y=MOYENNE]{graphes/M500_TE.txt};

   \addplot[color3!20] fill between[of=A and B];
\end{axis}
 
  \end{tikzpicture}
&\hspace{1cm}
   \begin{tikzpicture}[scale=0.45]
   \begin{axis}[xlabel={Size of the domain },ylabel={},ylabel={}, ymin=1.8,ymax=2.4,yticklabel style={ /pgf/number format/fixed, /pgf/number format/precision=5 }, scaled y ticks=false, enlarge x limits=false,title={$c_1$}, title style={at={(0.55,1.1)},anchor=north},  xlabel style={at={(0.5,-0.15)},anchor=south}]]
   %,ymin=-0.4,ymax=0.4
   
   \addplot [mark=*,green3,thick]table[x=T,y=VALEUR1]{evol_coef_TE.txt};
   \addplot [mark=*,color3,thick]table[x=T,y=VALEUR2]{evol_coef_TE.txt};
   \addplot [mark=*,orange,thick]table[x=T,y=VALEUR3]{evol_coef_TE.txt};
   \end{axis}
  \end{tikzpicture}\\\begin{tikzpicture}\draw[-, color=green3](0,0)--(0.2,0) node[at={(2.5,0)}, anchor=east]{\footnotesize100 realizations};\draw[-, color=color3](0,-0.4)--(0.2,-0.4) node[at={(2.5,-0.4)}, anchor=east]{\footnotesize500 realizations};\end{tikzpicture}&\hspace{1cm}\begin{tikzpicture}\draw node[at={(2,0.3)}, anchor=east]{\footnotesize different realizations on a large domain};\end{tikzpicture}\end{tabular}

\end{center}
 \end{center}
 \caption{Computation of the effective model's coefficient}\label{fig:c1}
 \end{figure}

Finally we verify the convergence rate of the error between the effective model and the reference solution. In the case of an incident plane wave $u_{inc}:=e^{i(k_1x_1+k_2x_2)}$ the solution $v_\eps$ to the effective model problem
\begin{equation*}
\left\{
\begin{aligned}
-\Delta v_{\eps} - k^2 v_\eps &= 0 \quad \text{ in }\,\,\mathcal{B}_{\eps H,L} , \\
-\eps c_1 \partial_{x_d} v_\eps  + v_\eps &= 0 \quad \text{ on }\,\, \Sigma_{\eps H} , \\
-\partial_{x_d}  v_\eps + \Lambda^k  v_\eps &= -\partial_{x_d}  u_{inc} + \Lambda^k  u_{inc} \quad \text{ on } \quad \Sigma_{L} ,
\end{aligned}
\right.
\end{equation*} is explicit:
\begin{equation*}
v_\eps(x_1, x_2)=e^{i(k_1x_1+k_2x_2)}+r_\eps^2e^{i(k_1x_1-k_2x_2)}
\end{equation*}
for $x_1\in \mathbb{R}, x_2\geq \eps H$ where the reflection coefficient $r_\eps^2$ is given by\begin{equation*}
r_\eps^2\coloneqq\frac{ik_2\eps c_1-1}{ik_2\eps c_1+1}e^{2ik_2\eps H}.
\end{equation*}
Similarly $u_0^{FF}=e^{i(k_1x_1+k_2x_2)}+r_\eps^1e^{i(k_1x_1-k_2x_2)}$ with $r^1_\eps\coloneqq-e^{2ik_2\eps H}$. The reflection coefficient of the computed reference solution $u_\eps^\omega$ for a given realization $\omega$ can be computed as
\begin{equation*}
r_{ref}^\omega\coloneqq\frac{e^{ik_2L}}{R}\int_{\Sigma_{L}^\sharp}(u_\eps^\omega-u_{inc})(x_1,L)e^{-k_1x_1}\,\dd x_1.
\end{equation*}On Figure \ref{fig:reflectioncoef} we plot the errors $\E[|r_{ref}^\omega-r_\eps^i|]$ for $i=1,2$ with respect to $k_2\eps$ for two different densities of particles $\rho=0.1$ and $0.4$. The expectation is computed over $50$ samples. The shadowed areas correspond to the confidence intervals using one empirical standard deviation. 
For each $\rho$ we recover the expected rate of convergence of $1$ for $i=1$. For $i=2$, the convergence rate is $1.57$ for $\rho=0.1$ and $1.3$ for $\rho=0.4$ (recall that the theoretical rate is $1.5$). Note first that as expected the larger the density the closer the solution is to the solution of the Dirichlet problem. 
For $k_2 \eps=0.025$ the first order model gives a precision of $0.22$ for $\rho=0.1$ and $0.14$ for $\rho=0.4$. For this value of $k_2\eps$, it is imperative to use the second order model that is more accurate. At order 2 for $k_2\eps =0.025$ the error is indeed about $0.01$ for $\rho=0.1$ and $0.003$ for $\rho=0.4$. Again, the error at order 2 becomes smaller as $\rho$ increases.  This explains in our point of view why the convergence rate for the order 2 degrades as $\rho$ increases. Indeed, since the error is smaller it is more sensible to the approximation of the reference solution. 
As a result to witness the theoretical convergence rates one needs to solve for the reference solution with higher accuracy (\textit{i.e.} smaller $\eps$, larger $T$) which makes the computation more costly. Note finally that this last observation justifies directly the need for effective boundary conditions.
 \begin{figure}
 \begin{center}\begin{tabular}{cc}
\includegraphics[scale=0.7]{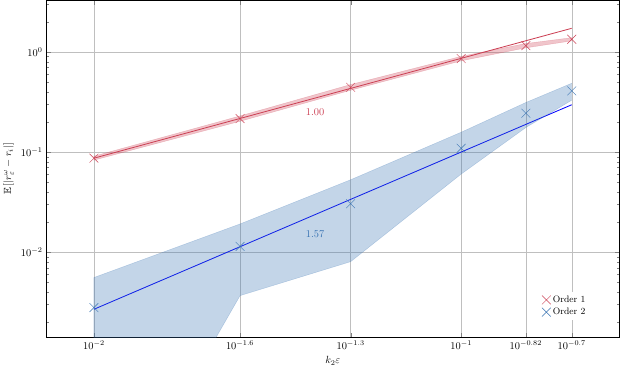}&\includegraphics[scale=0.7]{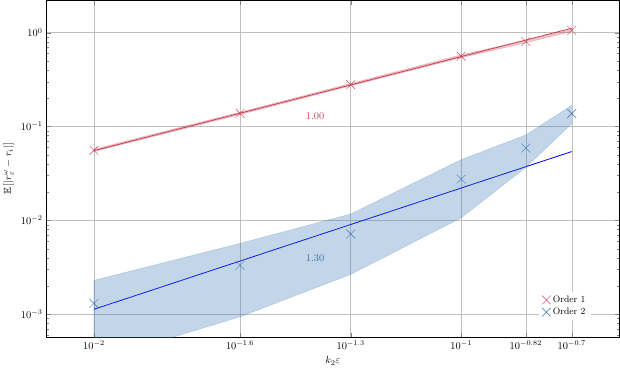}\end{tabular}
 \end{center}
 \caption{Error between the reflection coefficients for the reference solution and the effective model for $\rho=0.1$ (left) and $\rho=0.4$ (right)}\label{fig:reflectioncoef}
 \end{figure}

%%%%% Biblio
\bibliographystyle{abbrv}
\bibliography{references}
%%%%%

\appendix

\section{Integral representation of the near-field}\label{sec:proofir}

Let us introduce the space of stationary traces 
\begin{equation*}
  \mathcal{W}^{\frac{1}{2}}(\Sigma_L)\coloneqq  \Big\{\psi\in L^2\big(\Omega,H^{\frac{1}{2}}_{loc}(\Sigma_L)\big) ,  \psi \text{ stationary},\quad \mathbb{E}\Big[\| h_{|_{\Sigma_L}}^{\frac{1}{2}}\psi\|^2_{{L^2}(\square_1)} \Big]<+\infty\Big\}  
\end{equation*}
where we have identified $\Sigma_L$ with $\R^{d-1}$ for the definition of the stationarity. 

For any $\varphi$ such that $\varphi\in \mathcal{W}^{\frac{1}{2}}(\Sigma_L)$, we consider the following half-space problem
\begin{equation}\label{eq:halfspace}
\left\{
\begin{aligned}
-\Delta U^\omega &= 0 \quad \quad \text{ in } \quad  \mathcal{B}^\infty_L \ ,  \\
U^\omega &= \varphi^\omega\quad \text{ on } \quad \Sigma_L. 
\end{aligned}
\right.
\end{equation}The appropriate functional framework is defined as follows \begin{multline*}
    \mathcal{W}(\mathcal{B}^\infty_L)\coloneqq \Big\{V\in L^2(\Omega,H^1_{loc}\big(\mathcal{B}^\infty_L\big), V\big(\cdot,y_d\big) \text{ stationary for any } y_d>L, \\ \mathbb{E}\Big[\int_{\square_1}\int_{L}^{+\infty} | \nabla V^\omega|^2 \mathrm{d}\bsy \Big]<+\infty \Big\}. 
    \end{multline*}
Let $\widehat{U}^\omega$ be  the unique solution of \eqref{eq:NFtGen} in $\mathcal{W}_0(D)$ for $F=0$, $G=0$,$\mu^{-\frac{1}{2}}\Psi\in\mathcal{L}^2(\Sigma_0)$ and $\mu^{-\frac{1}{2}}\alpha^N\in \mathcal{L}^2(\Sigma_H)$. Note that $\widehat{U}^\omega|_{\mathcal{B}^\infty_L}\in\mathcal{W}(\mathcal{B}^\infty_L)$ verifies \eqref{eq:halfspace} with $\varphi=\widehat{U}^\omega|_{\Sigma_L}$.
The outline of the proof is the following : we first show that there exists a unique solution $V\in \mathcal{W}(\mathcal{B}^\infty_L)$ to \eqref{eq:halfspace}. We then prove that the integral representation \eqref{eq:repint_W10} is in $\mathcal{W}(\mathcal{B}^\infty_L)$ and verifies problem \eqref{eq:halfspace}.

Let us first prove the well-posedness result.
    \begin{proposition}\label{prop:pbdemiespa_bienpose}
For all $\varphi$ such that $\varphi\in \mathcal{W}^{\frac{1}{2}}(\Sigma_L)$, there exists a unique $V\in \mathcal{W}(\mathcal{B}^\infty_L)$ that is a.s. solution of \eqref{eq:halfspace}.
    \end{proposition}
To prove this result that is similar to the analysis of the half-space problem for the Stokes equation with a trace $H^{\frac{1}{2}}_{uloc}$ by D. Gerard-Varet and N.Masmoudi in \cite{gerard2010relevance}, we follow the following steps: for all  $\varphi$ such that $\varphi\in \mathcal{W}^{\frac{1}{2}}(\Sigma_L)$,
\begin{enumerate}
    \item[Step 1] we construct a regular stationary lift of $\varphi$;
    \item[Step 2] we reduce problem \eqref{eq:halfspace} to the problem with a vanishing Dirichlet boundary condition and prove well-posedness using an appropriate Hardy inequality.
\end{enumerate}
In the rest of the proof $\varphi$ is fixed.
\noindent\textit{Step 1: construction of a stationary and regular lifting of $\varphi$.}
In order to derive a lifting, we introduce the Green's function associated to the regularized operator $- \Delta + T^{-1}$ in $\R^d$ for $T>0$, namely
\begin{equation}\label{eq:fonction_Laplace_Green_T}
    G_T(\bsz) \coloneqq  
    -\frac{i}{4} H_0^{(1)}\big(i\sqrt{T}^{-1}|\bsz|\big)\; \text{ for } \; d=2,\quad  
    \frac{e^{-\sqrt{T}^{-1}|\bsz|}}{4\pi|\bsz|} \;\text{ for } \; d=3, \quad \bsz\in \mathbb{R}^{d}.
    \end{equation}
    Note that, using \cite[Equation (9.2.30)]{abramowitz1968handbook}, one can show that for $d=2$ there exists a constant (whose value is not relevant for the sequel) such that
    \begin{equation}\label{eq:GreenT2d}
    G_T(\bsz) = C \frac{e^{-\sqrt{T}^{-1}|\bsz|}}{|\bsz|^{\frac{1}{2}}}\left( 1+ \mathcal{O}(|\bsz|^{-1})\right)\quad \text{as}\;|\bsz|\rightarrow +\infty.
    \end{equation}
As we shall prove in the next proposition, a natural lifting of $\varphi$ is given by the integral formula
\begin{equation}\label{eq:VT_intrep}
    \text{a.s. }  \quad V_{T}^{\omega}\big(\bsy_\shortparallel,y_d\big) = -2 \int_{\mathbb{R}^{d-1}} \partial_{y_d}G_T\big(\bsy_\shortparallel-\bsz_\shortparallel,y_d-L \big) \ \varphi^\omega(\bsz_\shortparallel) \, \mathrm{d}\bsz_\shortparallel , \quad \bsy_\shortparallel\in \mathbb{R}^{d-1}, \; y_d>L
    \end{equation}
 that is in the functional space 
    \begin{multline*}
        \mathcal{H}^1(\mathcal{B}_L^\infty)\ \coloneqq  \ \Bigg\{V\in L^2\big(\Omega,H_{loc}^1(\mathcal{B}^\infty_L) \big), V\big(\cdot,y_d\big)\text{ stationary for any } y_d>L, \\ \mathbb{E}\Big[\int_{\square_1}\int_{L}^{+\infty} |V^\omega|^2+ | \nabla V^\omega|^2 \,\mathrm{d}\bsy \Big]<+\infty \Bigg\} \ ;
      \end{multline*}

\begin{lemma}\label{Lemme_decompo_VT_alea}
The function $V_T$ given by \eqref{eq:VT_intrep} is in $\mathcal{H}^1(\mathcal{B}_L^\infty)$ and satisfies a.s.
\begin{equation}\label{eq:halfspaceT}
    \left\{ 
    \begin{aligned}
    - \Delta V_T^\omega + T^{-1}V_T^\omega &= 0 \quad \quad \text{ in  } \quad \mathcal{B}^\infty_L \ , \\
    V_T^\omega &= \varphi^\omega \quad \text{ on } \quad \Sigma_L \ . 
    \end{aligned}
    \right.
    \end{equation}
\end{lemma}
\begin{remark}\label{rem:comportement_infini_reg}
Even if it is not fundamental for the sequel, note that the function $V_T$ given by \eqref{eq:VT_intrep} is even the unique solution of \eqref{eq:halfspaceT} in  $\mathcal{H}^1(\mathcal{B}_L^\infty)$

Let us also mention that if $\varphi^\omega\in H^{\frac{1}{2}}(\Sigma_L)$ then it is easy to show using Lax Milgram's Theorem that $V^\omega\in H^1(B_L)$. The difficulty in our setting comes from the fact that $\varphi\in\mathcal{W}^{\frac{1}{2}}(\Sigma_L)$ but as we shall see the stationarity plays a fundamental role. 
\end{remark}
\begin{proof}[Proof of Lemma \ref{Lemme_decompo_VT_alea}]

Let ${V}_T$ be given by the integral expression of \eqref{eq:VT_intrep}. It is easy to show (using similar arguments than in \cite{colton2013integral} or in \cite[Theorem 3.2]{chandler1997impedance} for instance) that a.s. for any $\bsy_\shortparallel\in \mathbb{R}^{d-1}$, $y_d>L$, $\widetilde{V}_T^\omega$ is in $C^2(\mathcal{B}_L^\infty)$, it satisfies the first equation of (\ref{eq:halfspaceT}) and the boundary condition of (\ref{eq:halfspaceT}) in the trace sense. This implies in particular that a.s. $V_T^\omega\in H_{loc}^1(\mathcal{B}^\infty_L)$ .
By stationarity of $\varphi$, we have for $ \bsy_\shortparallel$, $\bsx_\shortparallel \in \mathbb{R}^{d-1}$ and $ y_d>L$, 
\begin{equation*}
\begin{aligned}
\widetilde{V}_T^\omega\big(\bsy_\shortparallel+\bsx_\shortparallel,y_d \big) 
% &= -2 \int_{\mathbb{R}^{d-1}}\partial_{y_d} G_T^\omega\big(\bsy_\shortparallel+\bsx_\shortparallel-\mathbf{z},y_d-L\big) \ \varphi^\omega(\mathbf{z}) \, \mathrm{d}\mathbf{z}, \\ 
&= -2 \int_{\mathbb{R}^{d-1}}\partial_{y_d} G_T^\omega\big(\bsy_\shortparallel-\bsz_\shortparallel,y_d-L\big) \ \varphi^\omega(\bsz_\shortparallel+\bsx_\shortparallel) \, \mathrm{d}\bsz_\shortparallel, \\ 
&= -2 \int_{\mathbb{R}^{d-1}}\partial_{y_d} G_T^\omega\big(\bsy_\shortparallel-\bsz_\shortparallel,y_d-L\big) \ \varphi^{\tau_{\bsx_\shortparallel}\omega}(\bsz_\shortparallel) \, \mathrm{d}\bsz_\shortparallel= \widetilde{V}_T^{\tau_{\bsx_\shortparallel}\omega}\big(\bsy_\shortparallel,y_d \big). 
\end{aligned}
\end{equation*}
which shows that $\big(\omega,\bsy_\shortparallel\big)\mapsto V^\omega\big(\bsy_\shortparallel,y_d\big)\text{is stationary for any } y_d>L$ .

Moreover,  we have on one hand for $L'>L$,

 \begin{equation*}
 \begin{array}{ll}
 \vsd\dst\mathbb{E}\left[ \int_{\square_1}\int_{L'}^{+\infty} \big|{V}_T^\omega \big(\bsy_\shortparallel,y_d\big)\big|^2 \, \mathrm{d}\bsy\right]&\dst= \ 4 \  \mathbb{E}\left[\int_{\square_1} \int_{L'}^{+\infty} \left|\int_{\mathbb{R}^{d-1}} \partial_{y_d}G_T(\bsz_\shortparallel,y_d-L) \  \varphi^\omega\big(\bsy_\shortparallel-\bsz_\shortparallel\big) \, \mathrm{d}\bsz_\shortparallel \right|^2  \, \mathrm{d}\bsy\right], \\ 
\vsd&\dst \lesssim    \ \int_{\square_1}\int_{L'}^{+\infty} \mathbb{E}\left[\int_{\mathbb{R}^{d-1}}(\mu^\omega)^{-1}(\bsypar-\bsz_\shortparallel,L) \left|\partial_{y_d}G_T(\bsz_\shortparallel,y_d-L)\right|  \, \mathrm{d}\bsz_\shortparallel \right] \\\vsd&\dst\E \left[\int_{\mathbb{R}^{d-1}} \left|\partial_{y_d}G_T(\bsz_\shortparallel,y_d-L)\right|  \big|(\mu^\omega_{|_{\Sigma_L}})^{\frac{1}{2}}\varphi^\omega\big(\bsy_\shortparallel-\bsz_\shortparallel\big)\big|^2 \, \mathrm{d}\bsz_\shortparallel \right] \, \mathrm{d}\bsy,  \\ 
\vsd&\dst \lesssim   \ \int_{L'}^{+\infty} \left(\int_{\mathbb{R}^{d-1}} \left|\partial_{y_d}G_T(\bsz_\shortparallel,y_d-L)\right|  \, \mathrm{d}\bsz_\shortparallel \right)^2dy_d\, \mathbb{E}\Big[ h_{|_{\Sigma_L}}|\varphi|^2\Big] \\&\dst\lesssim \ \big\| \mu^{\frac{1}{2}}\varphi\big\|^{2}_{\mathcal{L}^2(\Sigma_L)},
 \end{array}
 \end{equation*}
where we used that $\mu(\cdot,L)$ and $\varphi$ are stationary, $\E[\mu^{-1}(\cdot,L)]\lesssim L^2$ and that (thanks to \eqref{eq:fonction_Laplace_Green_T} and \eqref{eq:GreenT2d}) for $L'$ large enough
\[
\int_{\mathbb{R}^{d-1}} \left|\partial_{y_d}G_T(\bsz_\shortparallel,y_d-L)\right|  \, \mathrm{d}\bsz_\shortparallel
 \lesssim \frac{1}{y_d-L}\quad \text{for}\;y_d>L'.\]
On the other hand, for $\chi\in \mathcal{C}_c^\infty\big(\mathbb{R}^{d-1}\big)$ such that $0\leq \chi\leq 1$ and $\chi=1$ on $\left(-\frac{3}{2},\frac{3}{2}\right)^{d-1}$, we can write
\begin{multline*}
  \mathbb{E}\left[ \int_{\square_1}\int_{L}^{L'} \big|{V}_T^\omega \big(\bsy_\shortparallel,y_d\big)\big|^2 \, \mathrm{d}y_d\mathrm{d}\bsypar\right] \lesssim \mathbb{E}\left[ \int_{\square_1}\int_{L}^{L'} \left| \int_{\mathbb{R}^{d-1}}  \partial_{y_d} G_T \big(\bsy_\shortparallel -\bsz_\shortparallel, y_d-L\big) \ \chi  \varphi^\omega(\bsz_\shortparallel)\, \mathrm{d}\bsz_\shortparallel \right|^2\, \mathrm{d}\bsy \right] \\   + \mathbb{E}\left[ \int_{\square_1}\int_L^{L'}  \left| \int_{\mathbb{R}^{d-1}}\partial_{y_d} G_T\big(\bsy_\shortparallel-\bsz_\shortparallel,y_d-L\big) \ (1-\chi)  \varphi^\omega\big(\bsz_\shortparallel\big)\, \mathrm{d}\bsz_\shortparallel \right|^2 \,\mathrm{d}\bsy\right].
\end{multline*}
Since a.s.  $\chi \varphi^\omega \in {H}^{\frac{1}{2}}(\Sigma_L)$ it is easy to show (see Remark \ref{rem:comportement_infini_reg}) that a.s.
\[
   \int_{\square_1}\int_{L}^{L'} \left| \int_{\mathbb{R}^{d-1}}  \partial_{y_d} G_T \big(\bsy_\shortparallel -\bsz_\shortparallel, y_d-L\big) \ \chi \ \varphi^\omega(\bsz_\shortparallel)\, \mathrm{d}\bsz_\shortparallel \right|^2\, \mathrm{d}\bsy <+\infty
\]
Moreover, using that when $\bsy_\shortparallel\in\square_1$, $(1-\chi)(\bsy_\shortparallel-\bsz_\shortparallel)= 0$ when $\bsz_\shortparallel\in\square_1$, we deduce that
\begin{align*}
 &\mathbb{E}\left[ \int_{\square_1} \int_L^{L'} \left| \int_{\mathbb{R}^{d-1}} \partial_{y_d}G_T\big(\bsz_\shortparallel,y_d-L\big) \ (1-\chi)\varphi^\omega(\bsy_\shortparallel-\bsz_\shortparallel) \, \mathrm{d}\bsz_\shortparallel\right|^2 \, \mathrm{d}\bsy \right],\\
&\quad \qquad\qquad\lesssim \int_{\square_1}\int_{L}^{L'} \mathbb{E}\left[ \int_{\mathbb{R}^{d-1}\setminus\square_1} (\mu^\omega)^{-1}(\bsypar-\bsz_\shortparallel,L)\left|\partial_{y_d}G_T(\bsz_\shortparallel,y_d-L)\right|  \, \mathrm{d}\bsz_\shortparallel \right]  \\ &\qquad\qquad\qquad\qquad \mathbb{E}\left[\int_{\mathbb{R}^{d-1}\setminus\square_1} \left|\partial_{y_d}G_T(\bsz_\shortparallel,y_d-L)\right| \  \big|(\mu^\omega_{|_{\Sigma_L}})^{\frac{1}{2}}\varphi^\omega\big(\bsy_\shortparallel-\bsz_\shortparallel\big)\big|^2 \, \mathrm{d}\bsz_\shortparallel \right] \,\mathrm{d}\bsy, \\ 
&\quad \qquad\qquad\lesssim   \ \int_{L}^{L'} \left(\int_{\mathbb{R}^{d-1}\setminus\square_1} \left|\partial_{y_d}G_T(\bsz_\shortparallel,y_d-L)\right|  \, \mathrm{d}\bsz_\shortparallel \right)^2\,\mathrm{d}y_d\,\mathbb{E}\Big[ h_{|_{\Sigma_L}}|\varphi|^2\Big], \\&\quad \qquad\qquad\dst\lesssim \ \big\| \mu^{\frac{1}{2}}\varphi\big\|^{2}_{\mathcal{L}^2(\Sigma_L)}.\end{align*}
Note that $\partial_{y_d}G_T\in L^1(\R^{d-1}\setminus\square_1\times(L,L'))$ is a regular function with an exponential decay at infinity (see \eqref{eq:fonction_Laplace_Green_T}-\eqref{eq:GreenT2d}).

Similar results can be derived for $\nabla V_T^\omega$.

\end{proof}

\noindent\textit{Step  2 : Laplace Problem with a vanishing Dirichlet boundary condition.}
Now consider $\widetilde{V}=V-V_T$. By Lemma \ref{Lemme_decompo_VT_alea}, $\widetilde{V}$ verifies a.s. 
\begin{equation}\label{eq:halfspaceD}
\left\{ 
\begin{aligned}
-\Delta \widetilde{V}^\omega &= - T^{-1} V_T^\omega\quad \text{ in } \quad \mathcal{B}^\infty_L \ , \\
\widetilde{V}^\omega \ &= \ 0 \quad \text{ in } \quad \Sigma_L \ . 
\end{aligned}
\right.
\end{equation} 

We introduce the following functional space
\begin{multline*}
 \mathcal{W}_0(\mathcal{B}^\infty_L)\coloneqq  \Big\{\text{a.s. } V^\omega\in H^1_{loc}\big(\mathcal{B}^\infty_L\big), V\big(\cdot,y_d\big) \text{ stationary for any }y_d>L,  \\ \text{a.s.  } V^\omega=0 \text{ on } \Sigma_L,\,\E\left[ \int_{\square_1}\int_L^{+\infty}\Big( \frac{|V^\omega|^2}{1+y_d^2}+|\nabla V^\omega|^2\Big) \, \mathrm{d}\bsy\right] <+\infty \Big\} \ .
\end{multline*}
equipped with the norm
\begin{equation*}
 \|V\|^2_{\mathcal{W}^0}\coloneqq \mathbb{E}\left[\int_{\square_1}\int_{L}^{+\infty} \Big( \frac{|V^\omega|^2}{1+y_d^2}+|\nabla V^\omega|^2\Big) \, \mathrm{d}\bsy\right]. 
\end{equation*}
A classical Hardy inequality gives that for all $V\in \mathcal{W}_0(\mathcal{B}^\infty_L)$ and since $V^\omega=0$ on $\Sigma_L$ a.s., we have a.s.
\begin{equation*}
\int_{\square_1} \int_L^{+\infty} \frac{|V^\omega|^2}{1+y_d^2} \, \mathrm{d}\bsy\leq \int_{\square_1}\int_L^{+\infty} |\nabla V^\omega|^2  \, \mathrm{d}\bsy.  
\end{equation*}
By taking the expectation of the last inequality, we obtain that the following Hardy inequality is verified in $\mathcal{W}_0(\mathcal{B}^\infty_L)$ 
\begin{equation}\label{inegalite_Hardy_alea}
\forall V\in \mathcal{W}_0(\mathcal{B}^\infty_L) \ , \quad \mathbb{E}\left[\int_{\square_1}\int_L^{+\infty} \frac{|V^\omega|^2}{1+y_d^2} \, \mathrm{d}\bsy \right] \lesssim\mathbb{E}\left[\int_{\square_1}\int_L^{+\infty} |\nabla V^\omega|^2 \, \mathrm{d}\bsy \right]. 
\end{equation}
\begin{lemma}
Problem \eqref{eq:halfspaceD} is well posed in $\mathcal{W}_0(\mathcal{B}^\infty_L)$.
\end{lemma}
\begin{proof}[Proof of Lemma \ref{inegalite_Hardy_alea}]
The variational formulation associated with \eqref{eq:halfspaceD} is given by 
\begin{equation}\label{eq:pbDir_FV}
\forall W\in \mathcal{W}_0(\mathcal{B}^\infty_L), \quad \mathbb{E} \left[\int_{\square_1} \int_{L}^{+\infty} \nabla \widetilde{V} \cdot \nabla\overline{ W}\, \mathrm{d}\bsy  \right] \ = \ -T^{-1}\mathbb{E} \left[\int_{\square_1}\int_{L}^{+\infty}  V_T \overline{W}\mathrm{d}\bsy \right]. 
\end{equation}The coercivity of the sesquilinear form can be directly deducted from Hardy's inequality \eqref{inegalite_Hardy_alea}.
From the integral representation of $V_T$, we deduce that $y_d\mapsto V_T(\cdot,y_d)$ is exponentially decaying and in particular
\begin{equation*}
\mathbb{E}\left[\int_{\square_1}\int_L^{+\infty} \big( 1 + y_d^2\big) | V_T|^2 \, \mathrm{d}\bsy \right] \ < \ +\infty \ . 
\end{equation*}
We have then by Cauchy-Schwarz inequality
\begin{equation*}
 T^{-1} \left|\mathbb{E}\left[ \int_{\square_1}\int_{L}^{+\infty}   V_T \overline{W} \mathrm{d}\bsy \right] \right|\lesssim T^{-1}\|\sqrt{1+y_d^2}V_T\|_{\mathcal{L}^2(\mathcal{B}^\infty_L)}\| W\|_{\mathcal{W}^{0}}  \ . 
\end{equation*}
The linear form in the r.h.s of \eqref{eq:pbDir_FV} is then continuous in $\mathcal{W}^{0}$. We conclude by Lax-Milgram's Theorem.
\end{proof}

We deduce finally proposition \ref{prop:pbdemiespa_bienpose} : since there exists a unique $\widetilde{V}\in \mathcal{W}_0(\mathcal{B}^\infty_L)$ solution of \eqref{eq:halfspaceD}, $V\coloneqq \widetilde{V}+V_T\in\mathcal{W}(\mathcal{B}^\infty_L)$ satisfies \eqref{eq:halfspace}. The uniqueness result  for \eqref{eq:halfspace} is the consequence of the one for\eqref{eq:halfspaceD}.

Let us finish the proof of Proposition \ref{prop:representation_integrale_U_alea} by proving that the integral representation given by \eqref{eq:repint_W10} is in $\mathcal{W}(\mathcal{B}^\infty_L)$ and verifies problem \eqref{eq:halfspace}. By uniqueness of the solution this shows that $\widehat{U}{|_{\mathcal{B}^\infty_L}}$ can be represented by \eqref{eq:repint_W10} .

\begin{proposition}\label{prop:regir} Let $Z:\R^{d-1}\times(L,+\infty)$ be defined as 
\begin{equation*}
 Z(\bsy)=-2\int_{\R^{d-1}}\partial_{y_d}\Gamma\big(\bsy_\shortparallel-\bsz_\shortparallel,y_d-L\big) \ \varphi^\omega\big(\bsz_\shortparallel\big) \mathrm{d}\bsz_\shortparallel,\end{equation*}Then $Z\in \mathcal{W}(\mathcal{B}^\infty_L)$ and $Z^\omega$ is a.s. solution of \eqref{eq:halfspace}.
\end{proposition}
\begin{proof}[Proof of proposition \ref{prop:regir}]
We follow the same steps as in the proof of Lemma \ref{Lemme_decompo_VT_alea} to show that $Z$
 satisfies \eqref{eq:halfspace}, that a.s. it belongs to $H^1_\text{loc}( \mathcal{B}_L^\infty)$ and that $Z(\cdot,y_d)$ is stationary for $y_d>L$. To prove that 
 \begin{equation*}
    \mathbb{E}\Big[ \int_{\square_1}\int_{L}^{+\infty} |\nabla \widehat{U}^\omega |^2\,\dd\bsy \Big] <+\infty 
    \end{equation*} 
  the calculations are similar to the proof of Lemma \ref{Lemme_decompo_VT_alea}, provided one uses that
  \[
  \partial_{y_d}\nabla_{\bsz} \Gamma(\bsz)=\mathcal{O}\left(\frac{1}{\left((z_d)^2+|\bsz_\shortparallel|^2\right)^{d/2}}\right)\quad\text{as}\; |\bsz|\rightarrow+\infty
  \]
  and 
  \[
  \int_{\mathbb{R}^{d-1}} \frac{\mathrm{d}\bsz_\shortparallel}{\Big(\big(y_d-L\big)^2+|\bsz_\shortparallel|^2\Big)^{d/2} } \, \lesssim \frac{1}{y_d-L}\quad\text{for}\;y_d>L'\quad\text{and}\quad \int_{\mathbb{R}^{d-1}\setminus\square_1} \frac{\mathrm{d}\bsz_\shortparallel}{\Big(\big(y_d-L\big)^2+|\bsz_\shortparallel|^2\Big)^{d/2} } <+\infty.
  \]
 \end{proof}
 
 \section{$L^p$-regularity of the auxiliary problem}\label{sec:appaux}
 
 To prove Theorem \ref{thm:aux}, we adapt the proof of \cite[Theorem 3.1]{DuerinckxGloria2021} to our problem. To this effect we start by reproducing two technical lemmatas that are key in the proof of the theorem. The first one is Gehring's Lemma \cite{gehring1973p} in the form of \cite[proposition 5.1]{modica1979regularity}.
 \begin{lemma}[Gehring's Lemma]\label{lem:Gehring} Let $1<q<s$ and $Q_0$ be a reference cube in $\R^d$. Take $F\in L^s(2Q_0)$ and $G\in L^q(2Q_0)$ two non-negative functions. Suppose that for every cube $Q\subset Q_0$
 \begin{equation*}\left(\fint_Q G^q\right)^{\frac{1}{q}}\leq b\fint_{2Q} G+b\left(\fint_{2Q} F^q\right)^{\frac{1}{q}}+\theta \left(\fint_{2Q}G^q\right)^{\frac{1}{q}},\end{equation*}where $b>1$ and $\theta\geq 0$. There exists $\theta_0=\theta_0(q,s,d)$ and $\eta_0=\eta_0(b,d,q,s)>0$ such that if $\theta\leq\theta_0$, then for all $q\leq p\leq q+\eta_0$
  \begin{equation*}\left(\fint_{Q_0} G^p\right)^{\frac{1}{p}}\lesssim\left(\fint_{2Q_0} F^p\right)^{\frac{1}{p}}+ \fint_{2Q_0}G,\end{equation*}
 \end{lemma}
 
The second lemma corresponds to the dual version of Calder\'on-Zygmund Lemma due to Shen \cite[Theorem 3.2]{shen2007lp}.
 
  \begin{lemma}[Dual Calder\'on-Zygmund Lemma]\label{lem:DCZ}Let $1\leq p_0<p_1$, $C\in[\frac{1}{2}, 2]$ and $Q_0$ be a reference cube in $\R^d$. Take $F\in L^{p_1}(2Q_0)$ and $G\in L^{p_1}(2Q_0)$. Suppose that for every cube $Q\subset Q_0$, there exists two measurable functions $F_{Q,0}$ and $F_{Q,1}$ such that $|F|\leq |F_{Q,0}|+|F_{Q,1}|$ and $|F_{Q,1}|\leq |F|+|F_{Q,0}|$ on $2Q_0$
  \begin{align*}
  \left(\fint_{Q} |F_{Q,0}|^{p_0}\right)^{\frac{1}{p_0}}&\leq C_1 \left(\fint_{CQ} |G|^{p_0}\right)^{\frac{1}{p_0}},\\
  \left(\fint_{\frac{1}{C}Q} |F_{Q,1}|^{p_1}\right)^{\frac{1}{p_1}}&\leq C_2 \left(\fint_{Q} |F_{Q,1}|^{p_0}\right)^{\frac{1}{p_0}},
  \end{align*}with $C_1,C_2>0$.
 Then for all $p_0<p<p_1$\begin{equation*}
  \left(\fint_{Q_0} |F|^{p}\right)^{\frac{1}{p}}\lesssim \fint_{2Q_0} |F|+\left(\fint_{2Q_0} |G|^{p}\right)^{\frac{1}{p}}.\end{equation*}
 \end{lemma}

\begin{proof}[Proof of Theorem \ref{thm:aux}]
 
\noindent\textit{Step 1.} We prove that a.s. for all $2\leq p\leq 2+\frac{1}{C_0}$
\begin{equation}\label{eq:upperpf}
  \left(\int_{\R^{d-1}}\int_{\R^+} \mathbb{1}_{D^\omega}|\nabla  u_{f,\bsg}^\omega|^p\right)^{\frac{1}{p}}
  \lesssim\left(\int_{\R^{d-1}}\int_{\R^+} |\bsg|^p\right)^{\frac{1}{p}}+\left(\int_{\R^{d-1}}\mu^{-\frac{p}{2}} |f|^p\right)^{\frac{1}{p}}.\end{equation}

Let $R>0$. Take $\chi$ a smooth cut-off function such that $\chi|_{\square_R\times(0,R)}\equiv1$, $\chi\geq 0$, $\textrm{supp}\chi\subset\square_{2R}\times(0,2R)$ and $\vsu\left|\nabla \chi\right|\lesssim \dst\frac{1}{R}$. We test \eqref{eq:WFug} with $v\coloneqq \chi(u_{f,\bsg}-c)$ where $c$ is a constant. We obtain 
\begin{multline*}
  \int_{\square_R\times(0,R)}\mathbb{1}_{D^\omega}|\nabla  u_{f,\bsg}^\omega|^2\leq\int_{\square_{2R}\times(0,2R)}\mathbb{1}_{D^\omega}\,\bsg\cdot(\chi\nabla u_{f,\bsg}^\omega+(u_{f,\bsg}^\omega-c)\nabla \chi)
  \\+\mathbb{1}_{R\geq L}\int_{\square_{2R}\cap\Sigma_L}f\chi(u_{f,\bsg}^\omega-c)-\int_{\square_{2R}\times(0,2R)}\mathbb{1}_{D^\omega}\,\nabla u_{f,\bsg}^\omega\cdot(u_{f,\bsg}^\omega-c)\nabla \chi.
\end{multline*}
Using Young's inequality with $K\geq 1$ we get
\begin{multline*}
  \int_{\square_R\times(0,R)}\mathbb{1}_{D^\omega}|\nabla  u_{f,\bsg}^\omega|^2\lesssim\frac{1}{R^2}\left(\frac{K^2}{2}+\frac{1}{2K^2}\right)\int_{\square_{2R}\times(0,2R)}\mathbb{1}_{D^\omega}|u_{f,\bsg}^\omega-c|^2+\frac{1}{K^2}\int_{\square_{2R}\times(0,2R)}\mathbb{1}_{D^\omega}|\nabla  u_{f,\bsg}^\omega|^2
  \\+\frac{1}{2K^2}\int_{\square_{2R}\cap\Sigma_L}\mu|u_{f,\bsg}^\omega-c|^2+\frac{K^2}{2}\int_{\square_{2R}\times(0,2R)}|\bsg|^2+\frac{K^2}{2}\int_{\square_{2R}\cap\Sigma_L}\mu^{-1}|f|^2.
\end{multline*}
Choosing $c\coloneqq \dst\fint_{\square_{2R}\times(0,2R)} \mathbb{1}_{D^\omega}u_{f,\bsg}^\omega$, we apply Poincar\'e-Sobolev inequality to the first term and the trace Theorem to the third term of the right-hand side
\begin{multline*}
  \left(\int_{\square_{R}\times(0,R)} \mathbb{1}_{D^\omega}|\nabla  u_{f,\bsg}^\omega|^2\right)^{\frac{1}{2}}
  \lesssim K\left(\int_{\square_{2R}\times(0,2R)} \mathbb{1}_{D^\omega}|\nabla  u_{f,\bsg}^\omega|^{p^\star}\right)^{\frac{1}{p^\star}}+\frac{1}{K}\left(\int_{\square_{2R}\times(0,2R)} \mathbb{1}_{D^\omega}|\nabla  u_{f,\bsg}^\omega|^2\right)^{\frac{1}{2}}
  \\+K\left(\int_{\square_{2R}\times(0,2R)} \left(|\bsg|^2+ (2R\mu)^{-1}|f|^2\right)\right)^{\frac{1}{2}},
\end{multline*}
with $p^\star={2d}/{(d+2)}$ and where we have identified the function $\mu^{-1}|f|^2\in L^1(\square_{2R}\cap\Sigma_L)$ to a function in $L^1(\square_{2R}\times(0,2R))$
that is independent of $x_d$. We conclude by applying Gehring's Lemma (Lemma \ref{lem:Gehring})
with $q={2}/{p^\star}$, $F\coloneqq \left(|\bsg|^2+(2R\mu)^{-1}|f|^2\right)^{{p^\star}/{2}}$ 
and $G\coloneqq |\nabla u_{f,\bsg}^\omega|^{p^\star}$ that for all $2\leq p\leq 2+ {1}/{C_0}$
\begin{equation}\label{eq:upperp}
    \left(\fint_{\square_{R}\times(0,R)} \mathbb{1}_{D^\omega}|\nabla  u_{f,\bsg}^\omega|^p\right)^{\frac{1}{p}}\lesssim 
    \left(\fint_{\square_{2R}\times(0,2R)} \mathbb{1}_{D^\omega}|\nabla  u_{f,\bsg}^\omega|^{p^\star}\right)^{\frac{1}{p^\star}}+\left(\fint_{\square_{2R}\times(0,2R)} |\bsg|^p+\fint_{\square_{2R}} (2R\mu)^{-\frac{p}{2}}|f|^p\right)^{\frac{1}{p}},
  \end{equation}
where $C_0\coloneqq \eta_0p^\star$.
By Jensen's inequality and since $2>p^*$, \eqref{eq:upperp} yields
\begin{multline*}
  \left(\int_{\square_{R}\times(0,R)} \mathbb{1}_{D^\omega}|\nabla  u_{f,\bsg}^\omega|^p\right)^{\frac{1}{p}}\lesssim R^{(\frac{1}{p}-\frac{1}{2})d}\left(\int_{\square_{2R}\times(0,2R)} \mathbb{1}_{D^\omega}|\nabla  u_{f,\bsg}^\omega|^{2}\right)^{\frac{1}{2}}
  \\+\left(\int_{\square_{2R}\times(0,2R)} |\bsg|^p+\int_{\square_{2R}}  \mu^{-\frac{p}{2}}|f|^p\right)^{\frac{1}{p}}.
\end{multline*}
\eqref{eq:energufg} implies 
\begin{multline*}
  \left(\int_{\square_{R}\times(0,R)} \mathbb{1}_{D^\omega}|\nabla  u_{f,\bsg}^\omega|^p\right)^{\frac{1}{p}}\lesssim R^{(\frac{1}{p}-\frac{1}{2})d}\left(\left(\int_{\R^{d-1}\times\R^+} |\bsg|^2\right)^{\frac{1}{2}}+\left(\int_{\R^{d-1}}\mu^{-1} |f|^2\right)^{\frac{1}{2}}\right)
  \\+\left(\int_{\square_{2R}\times(0,2R)} |\bsg|^p+\int_{\square_{2R}} \mu^{-\frac{p}{2}}|f|^p\right)^{\frac{1}{p}}.
\end{multline*}
Next we let $R\to+\infty$ and obtain \eqref{eq:upperpf} for $2\leq p\leq 2+ \frac{1}{C_0}$. 
% Moreover
% \footnote{\textcolor{red}{J'enleverai ca parce que ca fait bizarre: moi je me disais que $u_{0,0}=0$ si f et g sont nuls alors qu on ne l utilise pas comme ca dans la suite. Je pense que ce n est pas necessaire si on fait reference plus tard à \eqref{eq:upperp}} }, for $\bsg\equiv\boldsymbol{0}$ and $f\equiv0$, $u_{0, \boldsymbol{0}}$ verifies by \eqref{eq:upperp} 
% \begin{equation}\label{eq:36ii}
%   \left(\fint_{\square_{R}}\fint_{(0,R)} 
%   \mathbb{1}_{D^\omega}|\nabla  u_{0,\boldsymbol{0}}^\omega|^p\right)^{\frac{1}{p}}\lesssim  
%   \left(\fint_{\square_{2R}}\fint_{(0,2R)} \mathbb{1}_{D^\omega}|\nabla  u_{0, \boldsymbol{0}}^\omega|^{p^\star}\right)^{\frac{1}{p^\star}},\end{equation}for $2\leq p\leq 2+ \frac{1}{C_0}$.

\noindent\textit{Step 2: Proof of \eqref{eq:momentuf} for $2\leq q<p\leq 2+\frac{1}{C_0}$} Let $2\leq p_0\leq p_1\leq 2+\frac{1}{C_0}$ and $R>0$. Let $u^{\omega, 0}_{f,\bsg}$ be the unique solution in $W_0(D^\omega)$ of 
\begin{equation*}
\left\{
\begin{array}{l}
-\Delta u_{f,\bsg}^{\omega,0} = \nabla\cdot(\bsg\mathbb{1}_{\square_R\times(0,R)})\; \text{in} \;D^\omega\setminus(\Sigma_H\cup\Sigma_L),  \\
  -\partial_{y_d} u_{f,\bsg}^{\omega,0}  = 0 \;\text{on} \,\,\Sigma_0, \quad\text{and}\quad 
  u_{f,\bsg}^{\omega,0} =0\;\text{on}\,\,\partial\Pa^\omega,\\
 \Big[u_{f,\bsg}^{\omega,0} \Big]_H  = 0 \quad \text{and}\quad\Big[ -\partial_{y_d} u_{f,\bsg}^{\omega,0} \Big]_H=0, \\ 
 \Big[u_{f,\bsg}^{\omega,0}\Big]_L  = 0 \quad \text{and}\quad\Big[ -\partial_{y_d} u_{f,\bsg}^{\omega,0} \Big]_L=f\mathbb{1}_{\square_R}. 
\end{array}
\right.
\end{equation*}
We apply \eqref{eq:upperpf} to $u^{\omega, 0}_{f,\bsg}$ 
\begin{align*}
\int_{\square_R}\int_{(0,R)}\E\left[ \mathbb{1}_{D^\omega}|\nabla  u_{f,\bsg}^{\omega,0}|^{p_0}\right]&\leq\E\left[\int_{\R^{d-1}}\int_{\R^+} \mathbb{1}_{D^\omega}|\nabla  u_{f,\bsg}^{\omega,0}|^{p_0}\right]\\&\lesssim\E\left[\int_{\square_R}\int_{(0,R)} |\bsg|^{p_0}+ \int_{\square_R} \mu^{-\frac{p_0}{2}}|f|^{p_0}\right]\\&\lesssim\E\left[\int_{\square_{2R}}\int_{(0,2R)} |\bsg|^{p_0}+\int_{\square_{2R}} \mu^{-\frac{p_0}{2}}|f|^{p_0}\right].\end{align*}
On the other hand $u^{\omega, 1}_{f,\bsg}\coloneqq u^{\omega}_{f,\bsg}-u^{\omega, 0}_{f,\bsg}$ verifies in $D^\omega_R\coloneqq \square_R\times(0,R)\setminus\overline{\Pa^\omega}$
\begin{equation*}
\left\{
\begin{array}{l}
-\Delta u_{f,\bsg}^{\omega,1} =0\; \text{in}\;D^\omega_{R,R}\setminus(\Sigma_H\cup\Sigma_L),  \\
  -\partial_{y_d} u_{f,\bsg}^{\omega,1} = 0 \;\text{on} \,\,\Sigma_0\cap\square_R,\quad\text{and}\quad 
  u_{f,\bsg}^{\omega,1}=0\;\text{on} \,\,\partial\Pa^\omega,\\
 \Big[u_{f,\bsg}^{\omega,1} \Big]_{\Sigma_H\cap\square_R}  = 0 \quad \text{and}\quad\Big[ -\partial_{y_d} u_{f,\bsg}^{\omega,1} \Big]_{\Sigma_H\cap\square_R}=0,\\ 
 \Big[u_{f,\bsg}^{\omega,1}\Big]_{\Sigma_L\cap\square_R}  = 0 \quad \text{and}\quad \Big[ -\partial_{y_d} u_{f,\bsg}^{\omega,1} \Big]_{\Sigma_L\cap\square_R}=0.
\end{array}
\right.
\end{equation*}
We first deduce from the triangle inequality in $L^{\frac{p_1}{p_0}}(\square_{\frac{R}{2}}\times(0,\frac{R}{2}))$ the following
\begin{equation*}
\left(\fint_{\square_{\frac{R}{2}}}\fint_{(0,\frac{R}{2})}\E\left[ \mathbb{1}_{D^\omega}|\nabla  u_{f,\bsg}^{\omega,1}|^{p_0}\right]^{\frac{p_1}{p_0}}\right)^{\frac{1}{p_1}}\leq\E\left[\left(\fint_{\square_{\frac{R}{2}}}\fint_{(0,\frac{R}{2})} \mathbb{1}_{D^\omega}|\nabla  u_{f,\bsg}^{\omega,1}|^{p_1}\right)^{\frac{p_0}{p_1}}\right]^{\frac{1}{p_0}}.\end{equation*}
From \eqref{eq:upperp} with $g=0$ on $\square_R\times(0,R)$ and $f=0$ on $\Sigma_L\cap\square_R$, we obtain
\begin{equation*}\left(\fint_{\square_{\frac{R}{2}}}\fint_{(0,\frac{R}{2})}\E\left[ \mathbb{1}_{D^\omega}|\nabla  u_{f,\bsg}^{\omega,1}|^{p_0}\right]^{\frac{p_1}{p_0}}\right)^{\frac{1}{p_1}}\lesssim \E\left[\left(\fint_{\square_{R}}\fint_{(0,R)} \mathbb{1}_{D^\omega}|\nabla  u_{f,\bsg}^{\omega,1}|^{p^\star}\right)^{\frac{p_0}{p^\star}}\right]^{\frac{1}{p_0}}.\end{equation*}
Finally with Jensen's inequality (since $p_0\geq p^\star$) we conclude that
\begin{equation*}\left(\fint_{\square_{\frac{R}{2}}}\fint_{(0,\frac{R}{2})}\E\left[ \mathbb{1}_{D^\omega}|\nabla  u_{f,\bsg}^{\omega,1}|^{p_0}\right]^{\frac{p_1}{p_0}}\right)^{\frac{1}{p_1}}\lesssim \E\left[\fint_{\square_{R}}\fint_{(0,R)} \mathbb{1}_{D^\omega}|\nabla  u_{f,\bsg}^{\omega,1}|^{p_0}\right]^{\frac{1}{p_0}}.\end{equation*}We are now in position to apply Lemma \ref{lem:DCZ} with $C=2$,
\begin{align*} F\coloneqq \E[|\nabla u_{f,\bsg}|^{p_0}|]^{\frac{1}{p_0}}, &\quad G\coloneqq \E[|\bsg|^{p_0}+\mu^{-\frac{p_0}{2}}|f|^{p_0}]^{\frac{1}{p_0}},\\
F^0\coloneqq  \E[|\nabla u_{f,\bsg}^0|^{p_0}|]^{\frac{1}{p_0}}, &\quad F^1\coloneqq  \E[|\nabla u_{f,\bsg}^1|^{p_0}|]^{\frac{1}{p_0}}.\end{align*}We obtain for all $p_0<p<p_1$, $R_0>0$
\begin{multline*}\left(\fint_{\square_{R_0}}\fint_{(0,R_0)}\E[|\nabla u_{f,\bsg}|^{p_0}|]^{\frac{p}{p_0}}\right)^{\frac{1}{p}}\lesssim\fint_{\square_{R_0}}\fint_{(0,R_0)}\E[|\nabla u_{f,\bsg}|^{p_0}|\\+\left(\fint_{\square_{2R_0}}\fint_{(0,2R_0)}\E[|\bsg|^{p_0}]^{\frac{p}{p_0}}+\fint_{\square_{2R_0}}\mu^{-\frac{p_0}{2}}|f|^{p_0}]^{\frac{p}{p_0}}\right)^{\frac{1}{p}}.\end{multline*}We let $R_0\to+\infty$ and obtain \eqref{eq:momentuf} for $2\leq p\leq 2+ \frac{1}{C_0}$. 

\noindent\textit{Step 3: Proof of \eqref{eq:momentuf} for $2-{1}/{2C_0}\leq p<q\leq 2$.}
We proceed by duality:
\begin{equation*}\|\nabla u_{f,\bsg}\|_{L^{p}(\R^{d-1}\times\R^+,L^{q}(\Omega))}=\sup_{\bsh\in L^{p'}(\R^{d-1}\times\R^+,L^{q'}(\Omega))}\left\{\E\left[\int_{\R^{d-1}\times\R^+} h\cdot\nabla u_{f,\bsg}\right]| \|h\|_{L^{p'}(\R^{d-1}\times\R^+,L^{q'}(\Omega))}=1\right\}.\end{equation*}Moreover testing the variational formulation verified by $u_{0,\bsh}$ with $v\coloneqq \overline{u_{f,\bsg}}$ yields\begin{equation}\label{eq:vfuh}\int_{\R^{d-1}\times \R^+}\mathbb{1}_{D^\omega} \nabla u_{0,\bsh}\cdot \nabla u_{f,\bsg}=\int_{\R^{d-1}\times\R^+} \bsh\cdot\nabla u_{f,\bsg}.\end{equation}Testing the variational formulation verified by $u_{f,\bsg}$ with $v\coloneqq \overline{u_{0,\bsh}}$ and subtracting \eqref{eq:vfuh} gives
\begin{equation*}\int_{\R^{d-1}\times\R^+} \bsh\cdot\nabla u_{f,\bsg}=\int_{\R^{d-1}\times\R^+}\bsg\cdot\nabla u_{0,\bsh}+\int_{\R^{d-1}} fu_{0,\bsh}.\end{equation*}Moreover \begin{multline*}\sup\left\{\E\left[\int_{\R^{d-1}\times\R^+} \bsg\cdot\nabla u_{0,\bsh}+\int_{\R^{d-1}}fu_{0,\bsh}\right]| \|h\|_{L^{p'}(\R^{d-1}\times\R^+,L^{q'}(\Omega))}=1\right\}\leq\left(\|\bsg\|_{L^p(\R^{d-1},L^q(\Omega))}\right.\\\left.+\|\mu^{-\frac{1}{2}}f\|_{L^p(\R^{d-1},L^q(\Omega))}\right)\sup\left\{\|\nabla u_{0,\bsh}\|_{ L^{p'}(\R^{d-1}\times\R^+,L^{q'}(\Omega))}| \|h\|_{L^{p'}(\R^{d-1}\times\R^+,L^{q'}(\Omega))}=1\right\}.\end{multline*}
Since $2\leq q'<p'\leq 2+\frac{1}{C_0}$ we conclude thanks to step 2.2. By interpolation estimate \eqref{eq:momentuf} holds then for all $|p-2|, |q-2|<\frac{1}{C_0'}$ with $C_0'=8C_0$.
\end{proof}

\section{Proof of propositions \ref{prop:fluctw1} and \ref{prop:localnorm}}\label{sec:appnw1}

  To prove Proposition \ref{prop:fluctw1}, we use the following corollary of Hypothesis \nameref{hyp:mix} \cite[ proposition 1.10]{duerinckx2017multiscale}.
  \begin{lemma}[Control of high moments]\label{lem:highmix} 
    If $\Pa$ satisfies Hypothesis \nameref{hyp:mix}, then for all $q<+\infty$
and  all $\sigma(\Pa)$- measurable random variable $F(\Pa)$
\begin{equation*} 
\E\left[|F(\Pa)-\E[F(\Pa)]|^q\right]^{\frac{1}{q}}\lesssim q^2\E\left[\int_1^{+\infty}\left(\int_{\R^{d-1}} \left(\partial^{osc}_{\Pa, \mathcal{L}_{\ell}(\bsxpar)} F(\Pa)\right)^2 \dd \bsxpar\right)^{\frac{q}{2}} \ell^{-\frac{(d-1)q}{2}} \pi(\ell-1) d\ell\right]^{\frac{1}{q}}.
\end{equation*}
 \end{lemma}

\begin{proof}[Proof of Proposition \ref{prop:fluctw1}]

\noindent\textit{Step 1: Sensitivity of the $\nabla W_1$} \vsd Let $\bsxpar\in \R^{d-1}$ and $\ell\geq 1$. In what follows we use the same notations as in the proof of Proposition \ref{prop:osc}. Let us first  prove that \begin{equation}\label{eq:oscg}\left|\int_{\widetilde{D}^\omega}\bsg\cdot\nabla Z^\omega\right|\lesssim(\ell^\omega_+(\bsxpar))^2\|\mathbb{1}_{D^\omega}\nabla u_{\bsg}^\omega \|_{L^2(\mathcal{L}_{\ell^\omega_+}(\bsxpar))}\|\mathbb{1}_{D^\omega}\nabla W_1^\omega\|_{L^2(\mathcal{L}_{\ell^\omega_+}(\bsxpar))}.\end{equation}

Let  $\bsg\in\mathcal{C}_c^{\infty}(\R^{d-1}\times \R^+, \R^d).$ We introduce the auxiliary problem $u_{\bsg}$ solution of \eqref{eq:uf} with $f\coloneqq 0.$
The weak formulation associated to \eqref{eq:uf} reads for all $\varphi\in W_0(\widetilde{D}^\omega)$
\begin{equation}\label{eq:weakugbis}
  \int_{\widetilde{D}^\omega} \mathbb{1}_{D^\omega}\nabla u_{\bsg}\cdot\overline{\nabla \varphi}=\int_{\widetilde{D}^\omega}  \mathbb{1}_{D^\omega}\bsg\cdot\overline{\nabla \varphi}+\int_{\partial\Pa_\ell^\omega}\mathbb{1}_{D^\omega}\nabla u_{\bsg}\cdot\bsn\, \overline{\varphi}.\end{equation}We evaluate \eqref{eq:weakz} for $\varphi=\overline{u_{\bsg}}$ and \eqref{eq:weakugbis} for $\varphi=\overline{Z}.$ After substracting both expressions we obtain 
\begin{equation*}\int_{\widetilde{D}^\omega} \bsg\cdot\nabla Z^\omega=-\int_{\partial\Pa_\ell^\omega}\nabla u_{\bsg}^\omega\cdot\bsn \,\mathbb{1}_{D'}W_1^{'}-\int_{\partial\Pa_\ell'}\nabla W_1^{'}\cdot\bsn\, \mathbb{1}_{D^\omega}u_{\bsg}^\omega.\end{equation*} 
Using the same arguments as in the proof of proposition \ref{prop:osc}, we can estimate by duality $\|\nabla u_{\bsg}\cdot\bsn\|_{H^{-\frac{1}{2}}(\partial B(\bsx_i^\omega))}$ 
% for $i\in\mathcal{I}_{\ell}(\bsx_\shortparallel)$. We proceed by duality\footnote{\textcolor{red}{On peut raccourcir ici en faisant a la partie dans le corps du texte qui ressemble.}}. Let $\psi\in H^{\frac{1}{2}}(\partial B(\bsx_i^\omega))$ and $\Psi\in H^1(\dst\mathcal{L}_{\ell^\omega_+}(\bsxpar))$ a lifting of $\psi$ such that $\textrm{supp} \Psi\subset \dst\mathcal{L}_{\ell^\omega_+}(\bsxpar)$, $\Psi|_{\partial  B(\bsx_i^\omega)}\equiv\psi$ and $\Psi|_{\partial  B(\bsx_j^\omega)}\equiv0$  for $j\neq i$. By definition, $\Psi$ verifies \begin{equation*}\|\Psi\|_{W_0(\mathcal{L}_{\ell^\omega_+}(\bsxpar))}\leq\|\psi\|_{H^{\frac{1}{2}}(\partial B(\bsx_i^\omega))}.\end{equation*} We write the weak formulation associated to \eqref{eq:uf} in $\mathcal{L}_{\ell^\omega_+}(\bsxpar)$
% \begin{equation*}\int_{\mathcal{L}_{\ell^\omega_+}(\bsxpar)}\mathbb{1}_{D^\omega}\nabla u_{\bsg}\cdot\nabla \overline{\Psi}=\int_{\mathcal{L}_{\ell^\omega_+}(\bsxpar)}\mathbb{1}_{D^\omega}\bsg\cdot\nabla \overline{\Psi}+
% \int_{\partial B(\bsx_i^\omega)}\nabla u_{\bsg}\cdot\bsn\, \overline{\psi}\end{equation*}
% By duality we deduce that
and obtain,
\begin{equation}\label{eq:dugn}\|\nabla u_{\bsg}\cdot\bsn\|_{H^{-\frac{1}{2}}(\partial B(\bsx_i^\omega))}\leq\|\mathbb{1}_{D^\omega}\nabla u_{\bsg}^\omega\|_{L^2(\mathcal{L}_{\ell^\omega_+}(\bsxpar))}.\end{equation}
Combining \eqref{eq:tracelayer} and \eqref{eq:dugn} we conclude that
\begin{equation*}\left|\int_{\widetilde{D}^\omega}\bsg\cdot\nabla Z^\omega\right|\lesssim2(\ell^\omega_+(\bsxpar))\|\mathbb{1}_{D^\omega}\nabla u_{\bsg}^\omega \|_{L^2(\mathcal{L}_{\ell^\omega_+}(\bsxpar))}\|\mathbb{1}_{D'}\nabla W_1^{'}\|_{L^2(\mathcal{L}_{\ell^\omega_+}(\bsxpar))}.\end{equation*}
We get \eqref{eq:oscg} using \eqref{eq:claim1}.\vsd

\noindent\textit{Step 2:}
 Let $g\in\mathcal{C}_c^{\infty}(\R^{d-1}\times \R^+)$ and $2\leq q<+\infty$.
The second step consists in applying the higher-order version (Lemma \ref{lem:highmix}) of the mixing hypothesis (Hypothesis \nameref{hyp:mix}) to $\dst\int_0^{+\infty}\int_{\R^{d-1}}\mathbb{1}_{D^\omega}\bsg\cdot\nabla W_1(\bsy)\dd\bsy$.
\begin{align*}
  &\E\left[\left|\int_0^{+\infty}\int_{\R^{d-1}}\mathbb{1}_{D^\omega} \bsg\cdot\nabla W_1(\bsy)\dd\bsy\right|^q\right]^{\frac{2}{q}}
  \\&\leq q^2\E\left[\int_1^{+\infty}\left(\int_{\R^{d-1}} \left(\partial^{osc}_{\Pa, \mathcal{L}_{\ell}(\bsxpar)} \int_0^{+\infty}\int_{\R^{d-1}}\mathbb{1}_{D^\omega} \bsg\cdot\nabla W_1(\bsy)\dd\bsy\right)^2 \dd\bsxpar\right)^{\frac{q}{2}} \ell^{-\frac{q(d-1)}{2}} \pi(\ell-1) \dd\ell\right]^{\frac{2}{q}}
  \\&\lesssim q^2\E\left[\int_1^{+\infty}\left(\int_{\R^{d-1}} (\ell^\omega_+(\bsxpar))^4\|\mathbb{1}_{D^\omega}\nabla u_{\bsg}^\omega \|^2_{L^2(\mathcal{L}_{\ell^\omega_+}(\bsxpar))}\|\mathbb{1}_{D^\omega}\nabla W_1^\omega\|_{L^2(\mathcal{L}_{\ell^\omega_+}(\bsxpar))}^2 \dd\bsxpar\right)^{\frac{q}{2}} \right.
    \left.\ell^{-\frac{q(d-1)}{2}} \pi(\ell-1) \dd\ell\right]^{\frac{2}{q}}.
  \end{align*}
Next we use Hypothesis \nameref{hyp:Linfty} to bound a.s. and for a.e. $\bsxpar\in R^{d-1}$, $\ell^{\omega}_+(\bsxpar)$ by $\ell+R^+$. 
\begin{equation}\label{eq:f1}
  \begin{aligned}\begin{multlined}
    \E\left[\left|\int_0^{+\infty}\int_{\R^{d-1}}\mathbb{1}_{D^\omega} \bsg\cdot\nabla W_1(\bsy)\dd\bsy\right|^q\right]^{\frac{2}{q}}\lesssim q^2\E\left[\int_1^{+\infty}\left(\int_{\R^{d-1}}\left(\fint_{\square_{\ell+R^+}(\bsxpar)}\int_0^h\mathbb{1}_{D^\omega}|\nabla W_1^\omega|^2\right)\right.\right.
    \\\left.\left.\left(\fint_{\square_{\ell+R^+}(\bsxpar)}\int_0^h\mathbb{1}_{D^\omega}|\nabla u_{\bsg}^\omega|^2\right)\dd\bsxpar\right)^{\frac{q}{2}}(\ell+R^+)^{(d+1)q}\ell^{-\frac{q(d-1)}{2}} \pi(\ell-1) \dd\ell\right]^{\frac{2}{q}}
  \end{multlined}\\
  \begin{multlined}\hfill\lesssim q^2\left(\int_1^{+\infty}\E\left[\int_{\R^{d-1}}\left(\fint_{\square_{\ell+R^+}(\bsxpar)}\int_0^h\mathbb{1}_{D^\omega}|\nabla W_1^\omega|^2\right)\right.\right.\\\left.\left.\left(\fint_{\square_{\ell+R^+}(\bsxpar)}\int_0^h\mathbb{1}_{D^\omega}|\nabla u_{\bsg}^\omega|^2\right)\dd\bsxpar\right]^{\frac{q}{2}}(\ell+R^+)^{(d+1)q}\ell^{-\frac{q(d-1)}{2}} \pi(\ell-1) \dd\ell\right)^{\frac{2}{q}}
  \end{multlined}\end{aligned}\end{equation}
  By duality we can write
\begin{multline*}
  \E\left[\left(\int_{\R^{d-1}}\left(\fint_{\square_{\ell+R^+}(\bsxpar)}\int_0^h\mathbb{1}_{D^\omega}|\nabla W_1^\omega|^2\right)\left(\fint_{\square_{\ell+R^+}(\bsxpar)}\int_0^h\mathbb{1}_{D^\omega}|\nabla u_{\bsg}^\omega|^2\right)\dd\bsxpar\right)^{\frac{q}{2}}\right]^{\frac{2}{q}}
  \\=\sup_{\|X\|_{L^{\tilde{q}}(\Omega)}=1}\E\left[\int_{\R^{d-1}}\left(\fint_{\square_{\ell+R^+}(\bsxpar)}\int_0^h\mathbb{1}_{D^\omega}|\nabla W_1^\omega|^2\right)\left(\fint_{\square_{\ell+R^+}(\bsxpar)}\int_0^h\mathbb{1}_{D^\omega}|\nabla u_{\bsg}^\omega|^2\right)\dd\bsxpar X\right],
\end{multline*}
where $\tilde{q}=\frac{q}{q-2}$ and the supremum is computed over positive random variables $X$ independent of $\bsx$. We artificially introduce a spatial average over $\square_{R_1}$ for a given $R_1\geq 1$ by noticing that by Fubini-Tonelli's Theorem for $h\in L^1(\R^{d-1})$
\begin{equation}\label{eq:aver}
  \int_{\R^{d-1}}\fint_{\square_{\ell+R^+}(\bsxpar)}h(\bsypar)\dd\bsypar\dd\bsxpar=\int_{\R^{d-1}}h(\bsxpar)\dd\bsxpar.
\end{equation} 
We get
\begin{align*}
  &\sup_{\|X\|_{L^{\tilde{q}}(\Omega)}=1}\E\left[\int_{\R^{d-1}}\left(\fint_{\square_{\ell+R^+}(\bsxpar)}\int_0^h\mathbb{1}_{D^\omega}|\nabla W_1^\omega|^2\right)\left(\fint_{\square_{\ell+R^+}(\bsxpar)}\int_0^h\mathbb{1}_{D^\omega}|\nabla u_{\bsg}^\omega|^2\right)\dd\bsxpar X\right]
  \\&=\sup_{\|X\|_{L^{\tilde{q}}(\Omega)}=1}\int_{\R^{d-1}}\E\left[\fint_{\square_{R_1}(\bsypar)}\left(\fint_{\square_{\ell+R^+}(\bsxpar)}\int_0^h\mathbb{1}_{D^\omega}|\nabla W_1^\omega|^2\right)\left(\fint_{\square_{\ell+R^+}(\bsxpar)}\int_0^h\mathbb{1}_{D^\omega}|\nabla u_{\bsg}^\omega|^2\right)\dd\bsxpar\dd\bsypar X\right]
  \\&\begin{multlined}
    \leq\sup_{\|X\|_{L^{\tilde{q}}(\Omega)}=1}\int_{\R^{d-1}}\E\left[\left(\left(1+\frac{R_1}{\ell+R^+}\right)^{\frac{d-1}{2}}\fint_{\square_{\ell+R^++R_1}(\bsypar)}\int_0^h\mathbb{1}_{D^\omega}|\nabla W_1^\omega|^2\right)\right.
    \\\left.\left(\fint_{\square_{R_1}(\bsypar)}\fint_{\square_{\ell+R^+}(\bsxpar)}\int_0^h\mathbb{1}_{D^\omega}|\nabla u_{\bsg}^\omega|^2\dd\bsxpar\right)\dd\bsypar X\right]
  \end{multlined}\end{align*}
  By H\" older's inequality and the stationarity of $\nabla W_1$ we obtain then for $q\gg2$ such that $|\tilde{q}-1|\leq \frac{1}{2C_0}$ 
\begin{align*}&\E\left[\left(\int_{\R^{d-1}}\left(\fint_{\square_{\ell+R^+}(\bsxpar)}\int_0^h\mathbb{1}_{D^\omega}|\nabla W_1^\omega|^2\right)\left(\fint_{\square_{\ell+R^+}(\bsxpar)}\int_0^h\mathbb{1}_{D^\omega}|\nabla u_{\bsg}^\omega|^2\right)\dd\bsxpar\right)^{\frac{q}{2}}\right]^{\frac{2}{q}}\\&\begin{multlined}\lesssim\left(1+\frac{R_1}{\ell+R^+}\right)^{\frac{d-1}{2}}\E\left[\left(\fint_{\square_{\ell+R^++R_1}(0)}\int_0^h\mathbb{1}_{D^\omega}|\nabla W_1^\omega|^2\right)^{\frac{q}{2}}\right]^{\frac{2}{q}}\\\sup_{\|X\|_{L^{\tilde{q}}(\Omega)}=1}\int_{\R^{d-1}}\E\left[\left(\fint_{\square_{R_1}(\bsypar)}\fint_{\square_{\ell+R^+}(\bsxpar)}\int_0^h\mathbb{1}_{D^\omega}|\nabla u_{\bsg}^\omega|^2X\dd\bsxpar \right)^{\tilde{q}}\right]^{\frac{1}{\tilde{q}}}\dd\bsypar,.\end{multlined}\end{align*}
Since  $X\mapsto\E[|X|^{\frac{q}{2}}]^{\frac{1}{\frac{q}{2}}}$is a norm we have by Jensen's inequality and the stationarity of $\nabla W_1$
\begin{equation}\begin{aligned}\label{eq:statw1}&\frac{1}{\left(\ell+R^+\right)^{\frac{d-1}{2}}}\E\left[\left(\int_{\square_{\ell+R^++R_1}(0)}\int_0^h\mathbb{1}_{D^\omega}|\nabla W_1^\omega|^2\right)^{\frac{q}{2}}\right]^{\frac{2}{q}}\\&\leq\frac{1}{\left(\ell+R^+\right)^{\frac{d-1}{2}}}\E\left[\left(\int_{\square_{\ell+R^+}(0)}\int_{\square_{R_1}(\bsxpar)}\int_0^h\mathbb{1}_{D^\omega}|\nabla W_1^\omega|^2\dd\bsxpar\right)^{\frac{q}{2}}\right]^{\frac{2}{q}}\\&\leq\frac{1}{\left(\ell+R^+\right)^{\frac{d-1}{2}}}\int_{\square_{\ell+R^+}(0)}\E\left[\left(\int_{\square_{R_1}(0)}\int_0^h\mathbb{1}_{D^\omega}|\nabla W_1^\omega|^2\right)^{\frac{q}{2}}\right]^{\frac{2}{q}}\dd\bsxpar=\E\left[\left(\int_{\square_{R_1}(0)}\int_0^h\mathbb{1}_{D^\omega}|\nabla W_1^\omega|^2\right)^{\frac{q}{2}}\right]^{\frac{2}{q}}\end{aligned}\end{equation}
Since $X\mapsto\E[|X|^{\tilde{q}}]^{\frac{1}{\tilde{q}}}$ is a norm we get by Jensen's inequality
\begin{align*}&\sup_{\|X\|_{L^{\tilde{q}}(\Omega)}=1}\int_{\R^{d-1}}\E\left[\left(\fint_{\square_{R_1}(\bsypar)}\fint_{\square_{\ell+R^+}(\bsxpar)}\int_0^h\mathbb{1}_{D^\omega}|\nabla u_{\bsg}^\omega|^2X\dd\bsxpar \right)^{\tilde{q}}\right]^{\frac{1}{\tilde{q}}}\dd\bsypar\\&\leq\sup_{\|X\|_{L^{\tilde{q}}(\Omega)}=1}\int_{\R^{d-1}}\fint_{\square_{R_1}(\bsypar)}\fint_{\square_{\ell+R^+}(\bsxpar)}\int_0^h\E\left[\mathbb{1}_{D^\omega}|\nabla u_{\bsg}^\omega|^{2\tilde{q}}|X|^{\tilde{q}} \right]^{\frac{1}{\tilde{q}}}\dd\bsxpar\dd\bsypar\\&=\sup_{\|X\|_{L^{\tilde{q}}(\Omega)}=1}\int_{\R^{d-1}}\int_0^h\E\left[\mathbb{1}_{D^\omega}|\nabla u_{\bsg}^\omega|^{2\tilde{q}}|X|^{\tilde{q}} \right]^{\frac{1}{\tilde{q}}}\dd\bsxpar.\end{align*}In the last inequality we simplified the spatial averages in the right-hand side thanks to \eqref{eq:aver}. Moreover since the supremum is computed over positive $X$ we can reformulate it as follows
\begin{align*}\sup_{\|X\|_{L^{\tilde{q}}(\Omega)}=1}\int_{\R^{d-1}}\int_0^h\E\left[\mathbb{1}_{D^\omega}|\nabla u_{\bsg}^\omega|^{2\tilde{q}}|X|^{\tilde{q}} \right]^{\frac{1}{\tilde{q}}}\dd\bsxpar&=\sup_{\|Y\|_{L^{2\tilde{q}}(\Omega)}=1}\int_{\R^{d-1}}\int_0^h\E\left[\mathbb{1}_{D^\omega}|\nabla u_{\bsg}^\omega Y|^{2\tilde{q}} \right]^{\frac{1}{\tilde{q}}}\dd\bsxpar\\&=\sup_{\|Y\|_{L^{2\tilde{q}}(\Omega)}=1}\|\mathbb{1}_{D^\omega}\nabla u_{\bsg}^\omega Y\|_{L^2(\R^{d-1}\times (0,h), L^{2\tilde{q}}(\Omega))}^2\end{align*}
Appealing to Theorem \ref{thm:aux} we get
\begin{align*}\|\mathbb{1}_{D^\omega}\nabla u_{\bsg}^\omega Y\|_{L^2(\R^{d-1}\times (0,h), L^{2\tilde{q}}(\Omega))}=\|\mathbb{1}_{D^\omega}\nabla u_{Y\bsg}^\omega\|_{L^2(\R^{d-1}\times (0,h), L^{2\tilde{q}}(\Omega))}&\lesssim\|\bsg Y\|_{L^2(\R^{d-1}\times\R^+, L^{2\tilde{q}}(\Omega))}\\&=\|Y\|_{L^{2\tilde{q}}(\Omega)}\|\bsg\|_{L^2(\R^{d-1}\times\R^+)^d}.\end{align*}
Finally we proved that
\begin{equation}\begin{multlined}\label{eq:f2}
  \E\left[\left(\int_{\R^{d-1}}\left(\fint_{\square_{\ell+R^+}(\bsxpar)}\int_0^h\mathbb{1}_{D^\omega}|\nabla W_1^\omega|^2\right)\left(\fint_{\square_{\ell+R^+}(\bsxpar)}\int_0^h\mathbb{1}_{D^\omega}|\nabla u_{\bsg}^\omega|^2\right)\dd\bsxpar\right)^{\frac{q}{2}}\right]^{\frac{2}{q}}\\\lesssim\|\bsg\|_{L^2(\R^{d-1}\times\R^+)^d}^2\E\left[\left(\int_{\square_{R_1}}\int_0^h\mathbb{1}_{D^\omega}|\nabla W_1^\omega|^2\right)^{\frac{q}{2}}\right]^{\frac{2}{q}}.\end{multlined}\end{equation}We inject \eqref{eq:f2} into \eqref{eq:f1} and conclude thanks to the weight $\pi$ with super-algebraic decay.
\end{proof}

\begin{proof}[Proof of proposition \ref{prop:localnorm}]
Let $R\geq R^+$ and $L\geq R+H$. Take $\chi$ a smooth cut-off function such that $0\leq\chi\leq 1$, $\chi|_{\square_R\times(0,L)}\equiv1$, $\textrm{supp}\chi\subset\square_{2R}\times (0,2L)$ and $\vsu\left|\nabla \chi(\bsx)\right|\lesssim \dst\frac{1}{R}$ for all $\bsx\in\R^{d-1}\times\R$. We test \eqref{eq:w10} with $v\coloneqq \chi W_1$. We obtain 
\begin{equation*}
  \int_{\square_R}\int_0^{L}\mathbb{1}_{D^\omega}|\nabla  W_1^\omega|^2\leq-\int_{\square_{2R}}\int_0^{2L}\nabla W_1^\omega\cdot W_1^\omega\nabla \chi+\int_{\square_{2R}\cap\Sigma_H}\chi W_1^\omega.
\end{equation*}
Using Young's inequality with $K\geq 1$ we get
\begin{multline*}
  \int_{\square_R}\int_0^{L}\mathbb{1}_{D^\omega}|\nabla W_1^\omega|^2\lesssim\frac{1}{R^2}\frac{K^2}{2}\int_{\square_{2R}}\int_0^{2L}\mathbb{1}_{D^\omega}|W_1^\omega|^2+\frac{1}{2K^2}\int_{\square_{2R}}\int_0^{2L}\mathbb{1}_{D^\omega}|\nabla  W_1^\omega|^2\\+\frac{1}{2K^2}\int_{\square_{2R}\cap\Sigma_H}|W_1^\omega|^2+K^2R^{d-1}.
\end{multline*}
Next we apply the trace Theorem to the third term of the right-hand side 
  \begin{multline}\label{eq:init}
    \int_{\square_{R}}\int_0^{L} \mathbb{1}_{D^\omega}|\nabla  W_1^\omega|^2\lesssim K^2\left(R^{d-1}+\frac{1}{R^2}\int_{\square_{2R}}\int_0^{2L} \mathbb{1}_{D^\omega}|W_1^\omega|^2\right)\\+\frac{1}{K^2}\left(\int_{\square_{2R}}\int_0^{2L} \mathbb{1}_{D^\omega}|\nabla  W_1^\omega|^2\right).
  \end{multline}
  Recall that under Hypothesis \nameref{hyp:Linfty} Poincar\'e's inequality holds in boxes larger than $R^+$ with a constant linear in $R^+$. We decompose the second term of the right-hand side as follows
\begin{align*}
  \int_{\square_{2R}}\int_0^{2L} \mathbb{1}_{D^\omega}|W_1^\omega|^2&\leq \int_{\square_{2R}}\int_0^{2L} \mathbb{1}_{D^\omega}|W_1^\omega-\Lchi_r*W_1^\omega|^{2}+\int_{\square_{2R}}\int_0^{2L} \mathbb{1}_{D^\omega}|\Lchi_r*W_1^\omega|^{2}.\end{align*} 
where $\Lchi_r(\bsx)\coloneqq r^{-d}\Lchi\left(\frac{\bsx}{r}\right)$. On one hand, we obtain by Poincar\'e's inequality under Hypothesis \nameref{hyp:Linfty}
\begin{equation}\label{eq:cnw}
\int_{\square_{2R}}\int_0^{2L} \mathbb{1}_{D^\omega}|\Lchi_r*W_1^\omega|^{2}\lesssim (R^+)^2\int_{\square_{2R}}\int_0^{2L} \mathbb{1}_{D^\omega}|\Lchi_r*\nabla W_1^\omega|^{2}.
\end{equation}
On the other hand, we write
\begin{equation*}
  \mathbb{1}_{D^\omega}\left(W_1^\omega-\Lchi_r*W_1^\omega\right)(\bsxpar,x_d)=\int_0^r\int_{\square_r}\int_0^1\mathbb{1}_{D^\omega}\Lchi_r(\bsy)\nabla W_1^\omega(\bsx+t\bsy)\cdot \bsy\dd t\dd\bsypar\dd y_d.
\end{equation*}
Using Jensen's inequality, we get
\begin{equation}\begin{aligned}\label{eq:nw} 
  \int_{\square_{2R}}\int_{0}^{2L} \mathbb{1}_{D^\omega}|W_1^\omega-\Lchi_r*W_1^\omega|^{2}&\leq \int_{\square_{2R}}\int_{0}^{2L}\int_{0}^r\int_{\square_r}\int_0^1\mathbb{1}_{D^\omega}|\Lchi_r(\bsy)|^2|\nabla W_1^\omega(\bsx+t\bsy)|^2\,|\bsy|^2\,\dd t\dd\bsy\dd\bsx
  \\&\lesssim r^2\int_{\square_{2R+r}}\int_{0}^{2L+r}\mathbb{1}_{D^\omega}|\nabla W_1^\omega(\bsx)|^2\dd\bsx. 
\end{aligned}\end{equation} 
Plugging back \eqref{eq:cnw} and \eqref{eq:nw} into \eqref{eq:init} we finally proved the following 
\begin{multline}\label{eq:interm}
  \fint_{\square_{R}}\int_{0}^{L}\mathbb{1}_{D^\omega}|\nabla  W_1^\omega|^2\lesssim K^2\left(1+\fint_{\square_{2R}}\int_{0}^{2L} \mathbb{1}_{D^\omega}|\Lchi_r*\nabla W_1^\omega|^{2}\right)\\+\left(1+\frac{r}{2R}\right)^{d-1}\left(\frac{1}{K^2}+K^2 \frac{r^2}{R^2}\right)\fint_{\square_{2R+r}}\int_{0}^{2L+r} \mathbb{1}_{D^\omega}|\nabla  W_1^\omega|^2.
\end{multline}
Next we take the $L^{\frac{q}{2}}(\Omega)$ norm of both sides of \eqref{eq:interm} for $q\geq 2$. By stationarity and Jensen's inequality we have on one hand
\begin{align*}\E\left[\left(\fint_{\square_{2R}}\int_{0}^{2L} \mathbb{1}_{D^\omega}|\Lchi_r*\nabla W_1^\omega|^{2}\right)^{\frac{q}{2}}\right]^{\frac{2}{q}}&\leq\int_{0}^{2L}\E\left[ \mathbb{1}_{D^\omega}|\Lchi_r*\nabla W_1^\omega|^{q}\right]^{\frac{2}{q}}\\&\leq\int_0^{2L}\E\left[ \left|\int_{\R^{d-1}\times\R^+}\mathbb{1}_{D^\omega}\Lchi_r(\bsypar, x_d-y_d)\nabla W_1^\omega(\bsy)\dd\bsy\right|^{q}\right]^{\frac{2}{q}}\dd x_d,\end{align*}
and on the other hand similarly as in \eqref{eq:statw1}\begin{gather}\label{eq:aver2}\left(1+\frac{r}{2R}\right)^{d-1}\E\left[\left(\fint_{\square_{2R+r}}\int_{0}^{2L+r} \mathbb{1}_{D^\omega}|\nabla  W_1^\omega|^2\right)^{\frac{q}{2}}\right]^{\frac{2}{q}}\lesssim\E\left[\left(\fint_{\square_{R}}\int_{0}^{2L+r} \mathbb{1}_{D^\omega}|\nabla  W_1^\omega|^2\right)^{\frac{q}{2}}\right]^{\frac{2}{q}}.\end{gather}
To conclude we want to absorb the last term of the right-hand side of \eqref{eq:interm} into the left-hand side. 
In order to do so we decompose the integral in $y_d$ in two parts: $(0,L)$ and $(L, 2L+r)$. 
Thanks to \eqref{eq:aver2} choosing $K\gg 1$ and $r={R}^\alpha$ for some $\alpha\in(0,1)$ so that $\frac{1}{K^2}+K^2\frac{r^2}{R^2}$ is small enough we can absorb $\E\left[\left(\fint_{\square_{R}}\int_{0}^{L} \mathbb{1}_{D^\omega}|\nabla  W_1^\omega|^2\right)^{\frac{q}{2}}\right]^{\frac{2}{q}}$ into the left-hand side. 

Let us deal with the last term $\E\left[\left(\fint_{\square_{R}}\int_{L}^{2L+r} \mathbb{1}_{D^\omega}|\nabla  W_1^\omega|^2\right)^{\frac{q}{2}}\right]^{\frac{2}{q}}$. We plug in $W_1^{\omega}$'s integral representation \eqref{eq:repint_W10} and use Jensen's inequality
\begin{align*}
  &\E\left[\left(\fint_{\square_{R}}\int_{L}^{2L+r} \mathbb{1}_{D^\omega}|\nabla  W_1^\omega|^2\right)^{\frac{q}{2}}\right]^{\frac{2}{q}}\lesssim\fint_{\square_{R}}\int_{L}^{2L+r}\mathbb{1}_{D^\omega}\E\left[|\nabla  W_1^\omega|^q\right]^{\frac{2}{q}}\\
  &\lesssim\fint_{\square_{R}}\int_{L}^{2L+r}\mathbb{1}_{D^\omega}\left|\int_{\R^{d-1}}|\nabla \partial_d\Gamma(\bszpar, x_d-L)|\E\left[\left(\int_0^{\bsz_{\shortparallel}}|\varphi^\omega|^2(\bsxpar-\bsw)\dd\bsw\right)^{\frac{q}{2}}\right]^{\frac{1}{q}}\dd\bszpar\right|^2,\end{align*}
where $\varphi\coloneqq W_1^\omega|_{\Sigma_{L'}}$ with $L>L'$ and $\partial_d\Gamma$ is defined in \eqref{eq:der_Gamma}. By the trace Theorem we deduce
\begin{equation*}\begin{multlined}
  \fint_{\square_{R}}\int_{0}^{2L+l}\mathbb{1}_{D^\omega}\left|\int_{\R^{d-1}}|\nabla \partial_d\Gamma(\bszpar, x_d-L)|\E\left[\left(\int_0^{\bsz_{\shortparallel}}|\varphi^\omega|^2(\bsxpar-\bsw)\dd\bsw\right)^{\frac{q}{2}}\right]^{\frac{1}{q}}\dd\bszpar\right|^2\\\lesssim\fint_{\square_{R}}\int_{L}^{2L+l}\mathbb{1}_{D^\omega}\left|\int_{\R^{d-1}}|\nabla \partial_d\Gamma(\bszpar, x_d-L')|(2R)^{\frac{d-1}{2}}\E\left[\left(\fint_{\square_{2R}}\int_0^{L'}|\nabla W_1^\omega|^2(\bsw)\dd\bsw\right)^{\frac{q}{2}}\right]^{\frac{1}{q}}\dd\bszpar\right|^2,\end{multlined}\end{equation*}where the constant is proportional to $R^+L'$. We obtain finally
\begin{align*}
  &\fint_{\square_{R}}\int_{L}^{2L+l}\mathbb{1}_{D^\omega}\left|\int_{\R^{d-1}}|\nabla \partial_d\Gamma(\bszpar, x_d-L')|(2R)^{\frac{d-1}{2}}\E\left[\left(\fint_{\square_{2R}}\int_0^{L'}|\nabla W_1^\omega|^2(\bsw)\dd\bsw\right)^{\frac{q}{2}}\right]^{\frac{1}{q}}\dd\bszpar\right|^2\\&\lesssim(2R)^{d-1}\E\left[\left(\fint_{\square_{2R}}\int_0^{L}|\nabla W_1^\omega|^2(\bsw)\dd\bsw\right)^{\frac{q}{2}}\right]^{\frac{2}{q}}\int_L^{2L+l}\left|\int_{\square_R}|\nabla \partial_d\Gamma(\bszpar, x_d-L')|\dd\bszpar\right|^2.
\end{align*} 
We conclude once again by stationarity of $W_1$ and thanks to the integrability property of $\nabla \partial_d\Gamma$ that ensures that both in $2$ and $3$ dimensions
\begin{equation*}
  R^{d-1}\int_L^{2L+l}\left|\int_{\square_R}|\nabla \partial_d\Gamma(\bszpar, x_d-L')|\dd\bszpar\right|^2=o(R,L).\end{equation*}
\end{proof}
\end{document}